\documentclass[12pt, two side]{article}
\usepackage{amsfonts}

\usepackage{latexsym,amsmath,amsfonts,mathrsfs,wasysym,times,lineno,amscd}
\usepackage{amsthm,amscd,amsfonts,newlfont}
\usepackage{amssymb, upref, color}
\usepackage{pst-all}
\usepackage{epsf, graphicx, subfigure, verbatim, epic}
\usepackage{latexsym}
\usepackage[colorlinks]{hyperref}
\usepackage{mathrsfs}
\usepackage{array}

\numberwithin{equation}{section}

\theoremstyle{plain}
\newtheorem{exam}{Example}[section]
\newtheorem{theorem}[exam]{Theorem}
\newtheorem{lemma}[exam]{Lemma}
\newtheorem{remark}[exam]{Remark}

\newtheorem{definition}[exam]{Definition}

\newtheorem{notation}[exam]{Notation}

 \def\eps{{\varepsilon}}
 \def\ov{\overline}
 \def\S{\mathbb S}
 \def\R{\mathbb R}
 \def\N{\mathbb N}
 \def\Q{\mathbb Q}
\def\Z{\mathbb Z}

\def\A{\mathcal A}
\def\B{\mathcal B}

\title{Finite cyclicity of some center graphics through a nilpotent point inside quadratic systems}
\author{Robert Roussarie, Universit\'e de Bourgogne\\
Christiane Rousseau, Universit\'e de Montr\'eal\thanks{This research was supported by NSERC in Canada.}}
\begin{document}


\maketitle

\begin{abstract} In this paper we introduce new methods to prove the finite cyclicity of some graphics through a triple nilpotent point of saddle or elliptic type surrounding a center. After applying a blow-up of the family, yielding a singular 3-dimensional foliation, this amounts to proving the finite cyclicity of a family of limit periodic sets of the foliation. The boundary limit periodic sets of these families were the most challenging, but the new methods are quite general for treating such graphics.    We apply these techniques to prove the finite cyclicity of the graphic $(I_{14}^1)$, which is part of the program started in 1994 by Dumortier, Roussarie and Rousseau (and called DRR program) to show that there exists a uniform upper bound for the number of limit cycles of a planar quadratic vector field. We also prove the finite cyclicity of the boundary limit periodic sets in all graphics but one through a triple nilpotent point at infinity of saddle, elliptic or degenerate type (with a line of zeros) and  surrounding a center, namely the graphics $(I_{6b}^1)$, $(H_{13}^3)$,
 and $(DI_{2b})$.
\end{abstract}

\section{Introduction}
This paper is part of a long term program to prove the finiteness part of Hilbert's 16th problem for quadratic vector fields, sometimes written $H(2) <\infty$, namely the existence of a uniform bound for the number of limit cycles of quadratic vector fields. The DRR program (see paper \cite{DRR94(1)}) reduces this problem to proving that 121 graphics (limit periodic sets) have finite cyclicity inside quadratic vector fields, and the long term program is to prove the finite cyclicity of all these graphics. 

This program has been an opportunity to develop new more sophisticated methods for analyzing the finiteness of the number of limit cycles bifurcating from graphics in generic families of $C^\infty$ vector fields, in analytic families of vector fields, and in finite-parameter families of polynomial vector fields. In this paper, we focus on some graphics in the latter case: graphics through a nilpotent point and surrounding a center inside quadratic systems. The general method is to use the Bautin trick, namely transforming a proof of finite cyclicity of a generic graphic into a proof of finite cyclicity of a graphic surrounding a center. This is possible in quadratic systems since the center conditions are well known: indeed all graphics through a nilpotent point and surrounding a center occur in the stratum of reversible systems. The systems of this stratum are symmetric with respect to an axis, and are also Darboux integrable with an invariant line and an invariant conic. 
In practice, the Bautin trick consists in dividing a displacement map $V$ in a center ideal, i.e. in writing it as a finite sum of \lq\lq generalized monomials\rq\rq\ times non vanishing functions of the form 
\begin{equation}V(z)= \sum_{i=1}^n a_i m_i(1+h_i(z)),\label{type_V}\end{equation} where each $a_i$ belongs to the center ideal in parameter space, $m_i$ is a generalized monomial in $z$ and $h_i(z)=o(1)$ behaves well under derivation. 

To compute the displacement map, we write it as a difference of compositions of regular transitions and Dulac maps near the singular points. The Dulac maps are calculated in $C^k$ normalizing coordinates for a family unfolding the vector field. In this paper, we develop some general additional methods, which allow to prove the finite cyclicity of the graphic  $(I_{14}^1)$ (Figure~\ref{graphics}(a)). In particular, for the unfolding of this graphic, it is very helpful to be able to claim that all regular transitions are the identity in the center case. This is possible if we exploit the fact that the centers occur when the system is symmetric, and if we choose cleverly the sections on which the different transition maps are defined. 
Also, in the center case, the Dulac maps have a simple form since the system is Darboux integrable. 
\begin{figure}
\begin{center}
\qquad\subfigure[$(I_{14}^1)$]
{\includegraphics[width=3.2cm]{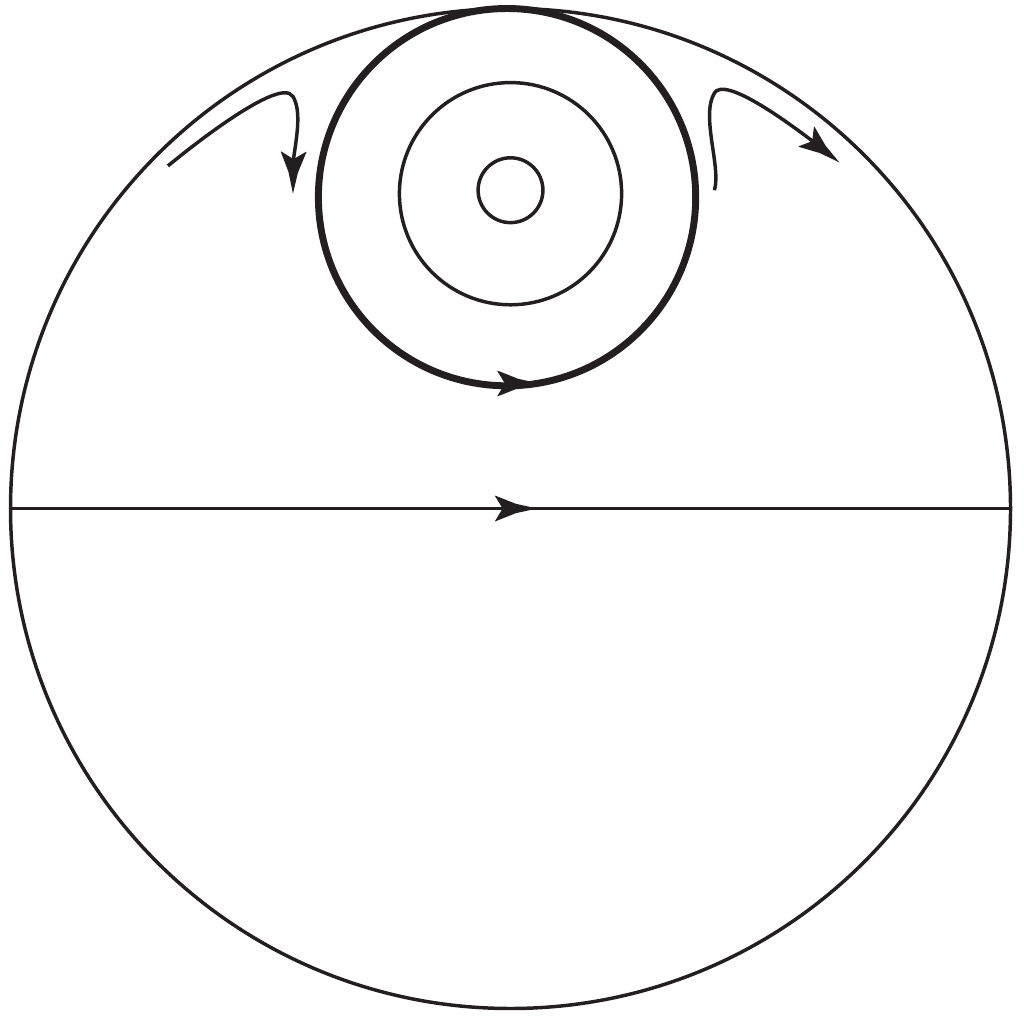}}\quad
\subfigure[$(I_{6b}^1)$]
{\includegraphics[width=3.2cm]{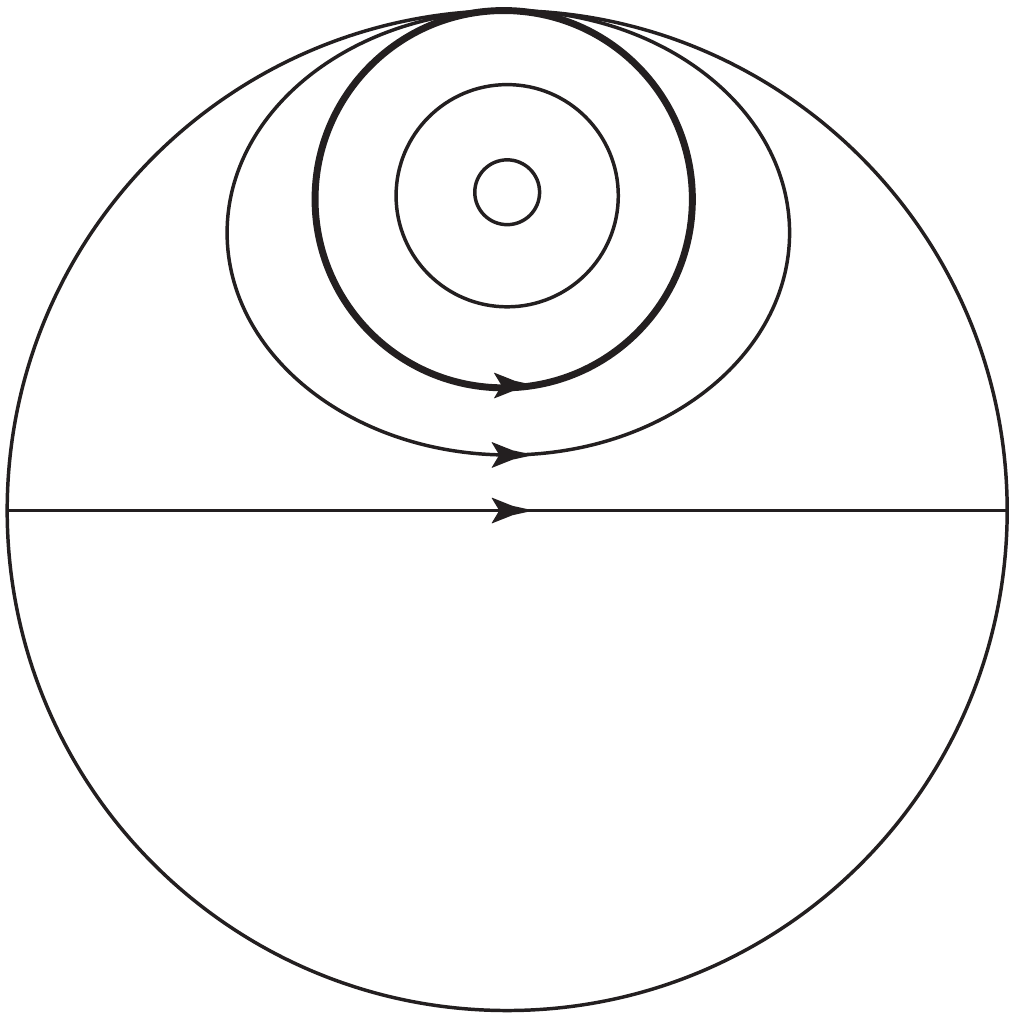}}\quad
\subfigure[$(H_{13}^3)$]
{\includegraphics[width=3.2cm]{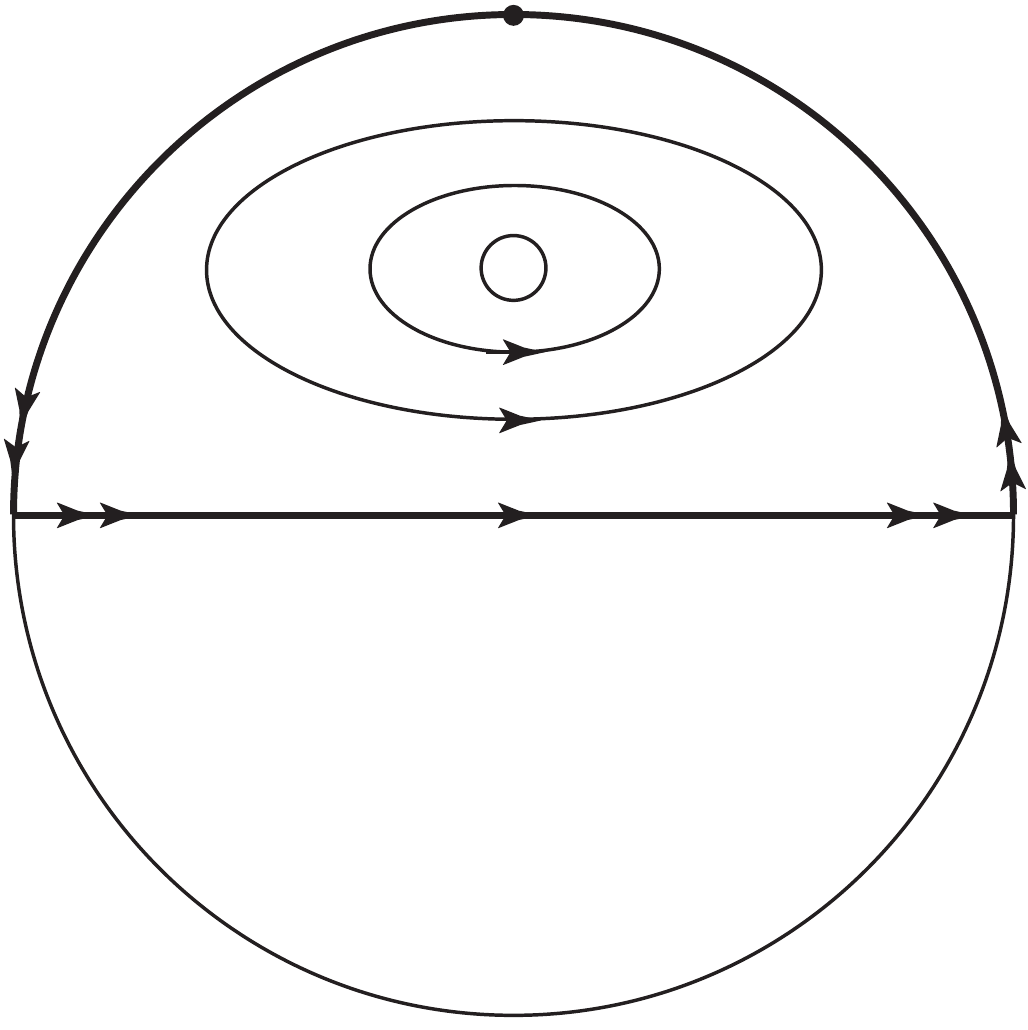}}\quad
\subfigure[$(DI_{2b})$]
{\includegraphics[width=3.2cm]{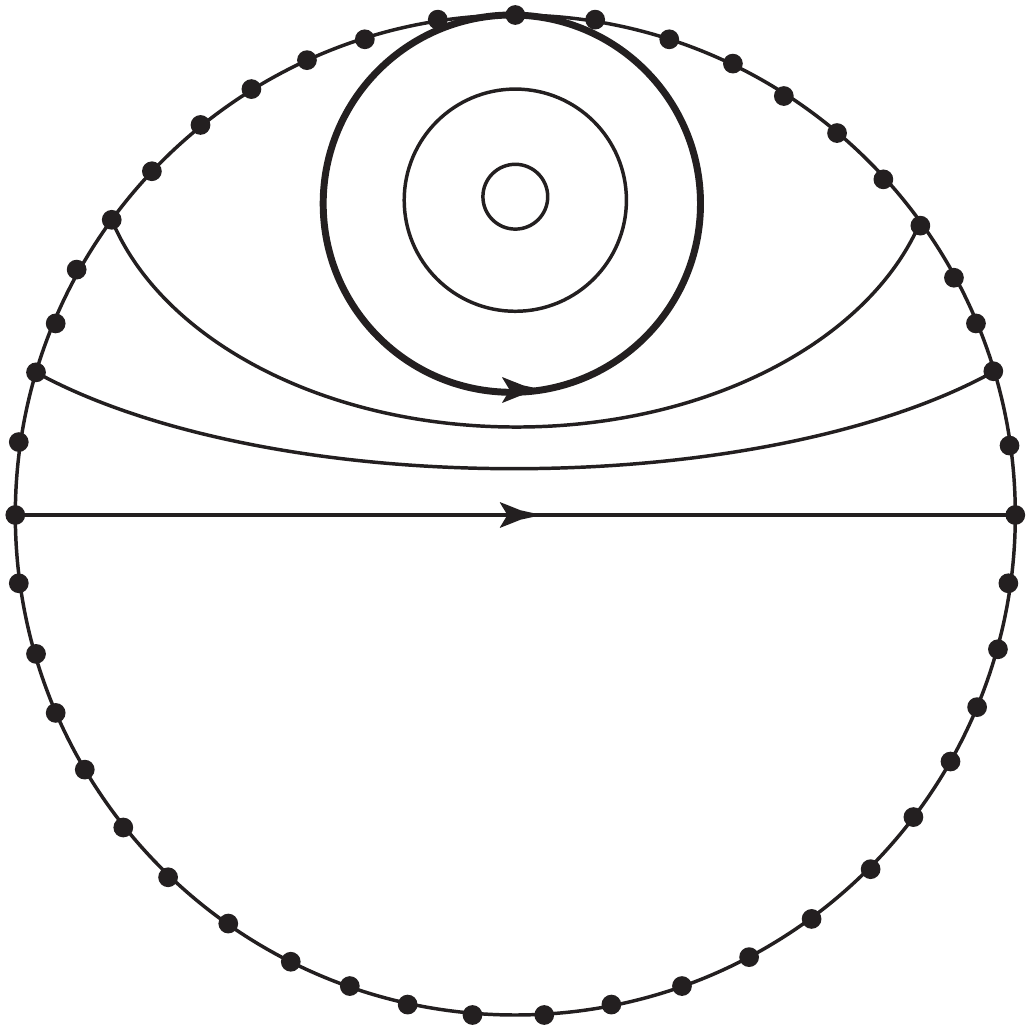}}
\caption{The graphics $(I_{14}^1)$, $(I_{6b}^1)$, $(H_{13}^3)$ and $(DI_{2b})$.}\label{graphics}
\end{center} \end{figure}
\vskip10pt 

The methods can be summarized as follows.
\begin{itemize} 
\item We highlight that the change to $C^k$ normalizing coordinates in the neighborhood of the singular points on the blow-up locus can be done by an operator. This allows preserving the symmetry in the center case when changing to normalizing  coordinates.
\item We introduce a uniform way of calculating the two types of Dulac maps when entering the blow-up through a much shorter proof than the one given in \cite{ZR}. 
\item Although each Dulac map is not $C^k$, we can divide in the center ideal its difference to the corresponding Dulac map in the integrable case.
\item The method of the blow-up of the family allows reducing the proof of finite cyclicity of the graphic to the proof that a certain number of limit periodic sets have finite cyclicity. These limit periodic sets are defined in the blown-up space. The ones obtained in blowing up a nilpotent saddle are shown in Table~\ref{tab.shhconvex}. For all of them but one (the boundary limit periodic set), we can reduce the displacement map to a $1$-dimensional map, the number of zeros of which can be bounded by the Bautin trick and a derivation-division algorithm on a map of type \eqref{type_V}. The boundary limit periodic set is more challenging, since we need to work with a 2-dimensional displacement map, the zeros of which we must study along the leaves of an invariant foliation coming from the blow-up. We introduce a generalized derivation operator, which allows performing a derivation-division algorithm on functions of the type 
\begin{equation}V(r,\rho)= \sum_{i=1}^n a_im_i(1+h_i(r,\rho)),\label{type_V2}\end{equation} 
where $h_i$ are ${\cal C}^k $-functions on monomials and $m_i$ are generalized monomials in $r$, $\rho$ (see definitions in Appendix II). During this process, we have to take into account  that $r\rho=\mathrm{Cst}$.
\end{itemize}

We have a partial  result for every graphic, but one (namely $(H^3_{14})$),  through a triple  point at infinity: 

\begin{theorem}\label{thMain1}
 Let us consider the graphics $(I^1_{14})$, $(I^1_{6b})$, $(H^3_{13})$  and $(DI_{2b})$ through a triple  point at infinity (see Figure 1). Then for any of them,  the boundary periodic limit set obtained in the blowing up   has  a finite cyclicity. 
\end{theorem}

Theorem \ref{thMain1} is not sufficient to prove that  the given graphic  has a finite cyclicity inside the family of quadratic vector fields. The reason is that, beside the boundary limit periodic set,   other limit periodic sets (see for instance Table~\ref{tab.shhconvex} for $(I_{14}^1)$) are obtained in the blowing up and, as explained above, we have to prove that each of them has also a finite cyclicity.  We present here a complete result  for the first graphic:

\begin{theorem}\label{thMain2}
 The graphic $(I^1_{14})$ has a finite cyclicity inside the family of quadratic vector fields.
\end{theorem}

As for the finite cyclicity of the other graphics $(I_{6b}^1)$, $(H_{13}^3)$
and $(DI_{2b})$, we intend to address the problem in the next future. 
The finite cyclicity of $(H_{13}^3)$ should be straightforward with arguments identical to those used for $(I_{14}^1)$. It will be done simultaneously with the corresponding generic graphic $(H_{12}^3)$.
Some of the limit periodic sets to be studied for $(I_{6b}^1)$ will involve four Dulac maps of second type. For these limit periodic sets, it is not possible to reduce the study of the cyclicity to a single equation. Hence, new methods will need to be adapted to treat the center case, when the periodic solutions correspond to a system of two equations in the four variables $r_1, \rho_1, r_2, \rho_2$, with $r_1\rho_1=\nu_1$ and $r_2\rho_2=\nu_2$. As for the graphic $(DI_{2b})$, some of the limit periodic sets to be studied involve four Dulac maps of second type, two of them through the semi-hyperbolic points $P_1$ and $P_2$ on the blown-up sphere. 

The techniques developed in this paper can be adapted for studying the boundary limit periodic sets of graphics of the DRR  program through a nilpotent finite singular point. The only new difficulty in that case is to show that the three parameters of the leading terms in the displacement map do indeed generate the center ideal. We also hope to adapt them to study the boundary graphic of the hemicycle $(H^3_{14})$: there, the additional difficulty is the two semi-hyperbolic points along the equator. 

\bigskip

 Proofs of Theorems \ref{thMain1} and \ref{thMain2} are given in Section 3 and Appendix  II,  where the detailed computations of cyclicity are found in Theorems \ref{thderdiv}, \ref{thpgeq2} and \ref{thp1}.  Theorem \ref{thnormalformhyp} in  Appendix I, gives a statement about normal form for 3-dimensional hyperbolic saddle points in  a way adapted to this paper. Theorem \ref{thtransgeneralhypsaddle} of the same appendix gives a new  proof for Dulac transitions near these saddle points, shorter than the one given in   \cite{ZR}. Precise properties for the specific unfoldings deduced from the quadratic family are proved in  Appendix III.  These properties of some parameter functions are needed to obtain the results of finite cyclicity.

\section{Preliminaries}

\subsection{Normal form for the unfolding of a nilpotent triple point of saddle or elliptic type} We consider graphics through one singular point, which is a triple nilpotent point of saddle or elliptic type. A germ of vector field in the neighborhood of such a point has the form
\begin{align}\begin{split}
\dot x&=y \\
\dot y&=\pm x^3+ bxy +\eta x^2y + yO(x^3)+ O(y^2).\end{split}\label{normal_form_DRS}\end{align}
The saddle case corresponds to the plus sign, and the elliptic case to the minus sign with $|b|\geq 2\sqrt{2}$. In the elliptic case, we limit ourselves here to the case $|b|>2\sqrt{2}$, which corresponds geometrically to a nilpotent point with hyperbolic points on the divisor of the quasi-homogeneous blow-up.

The unfolding of such points has been studied by Dumortier, Roussarie and Sotomayor, \cite{DRS}, including a normal form for the unfolding of the family. A different normal form has been used in \cite{ZR} for studying the finite cyclicity of generic graphics through such singular points, when we limit ourselves to $|b|>2\sqrt{2}$ in the elliptic case. This normal form is particularly suitable for applications in quadratic vector fields, where there is always an invariant line through a nilpotent point of multiplicity $3$.

\bigskip
A germ of $C^\infty$ vector field in the neighborhood of a  nilpotent point of multiplicity $3$ of saddle or elliptic type can be brought by an analytic change of coordinates to the form
\begin{align}\begin{split}
\dot x&=y+ ax^2, \\
\dot y&=y(x + \eta x^2 +o(x^2) + O(y)).\end{split}\label{normal_form_ZR}\end{align}
This requires an additional change of variable and scaling compared to what has been done in \cite{ZR}.
The point is a nilpotent saddle when $a<0$ and a nilpotent elliptic point when $a>0$ (see Figure~\ref{fig.topotype}). The case $|b|=2\sqrt{2}$ corresponds to $a=\frac12$.
\begin{figure}[ht]\begin{center}
    \subfigure[Saddle case]
    {\includegraphics[angle=0, width=5cm]{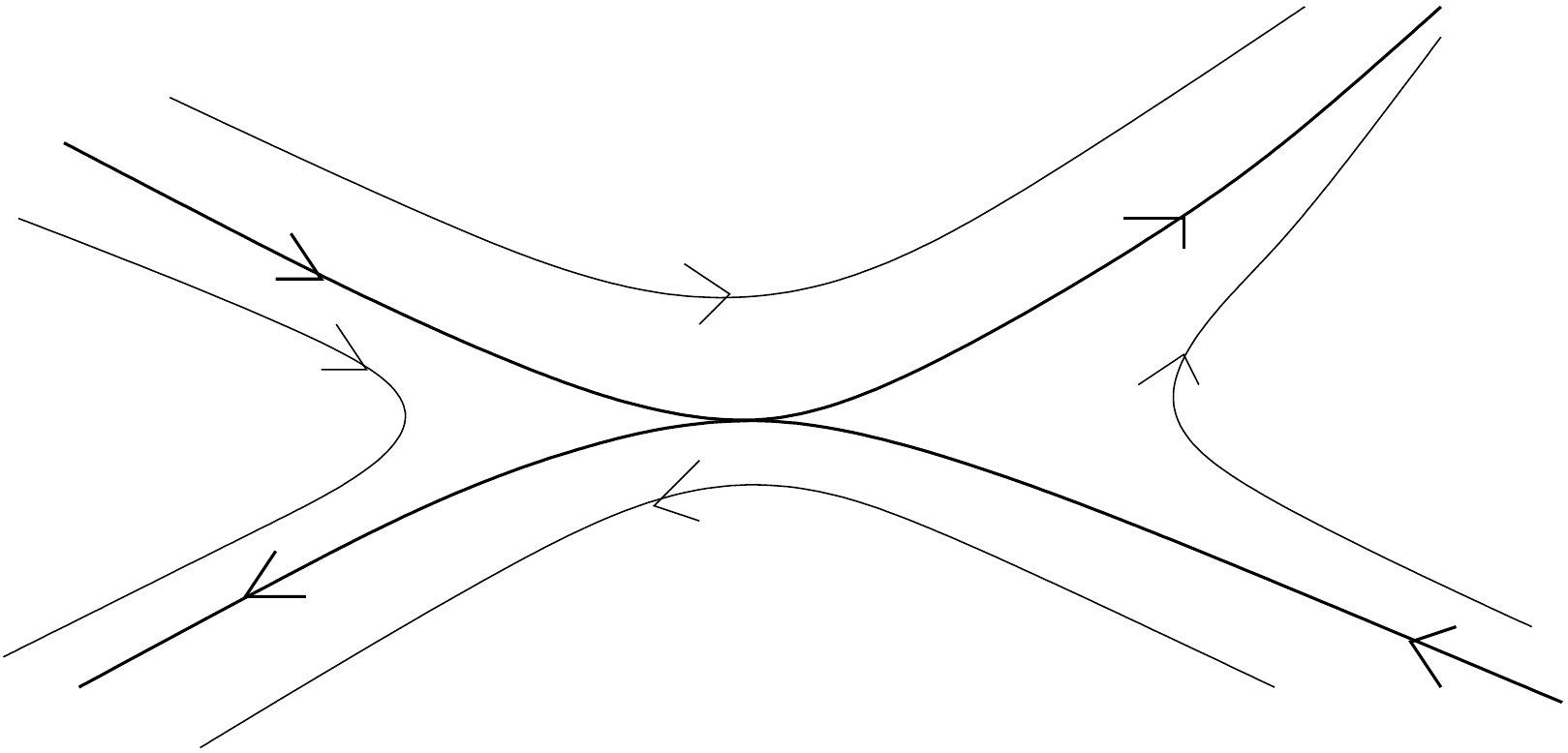}}\qquad\qquad
 \subfigure[Elliptic case]
    {\includegraphics[angle=0, width=5cm]{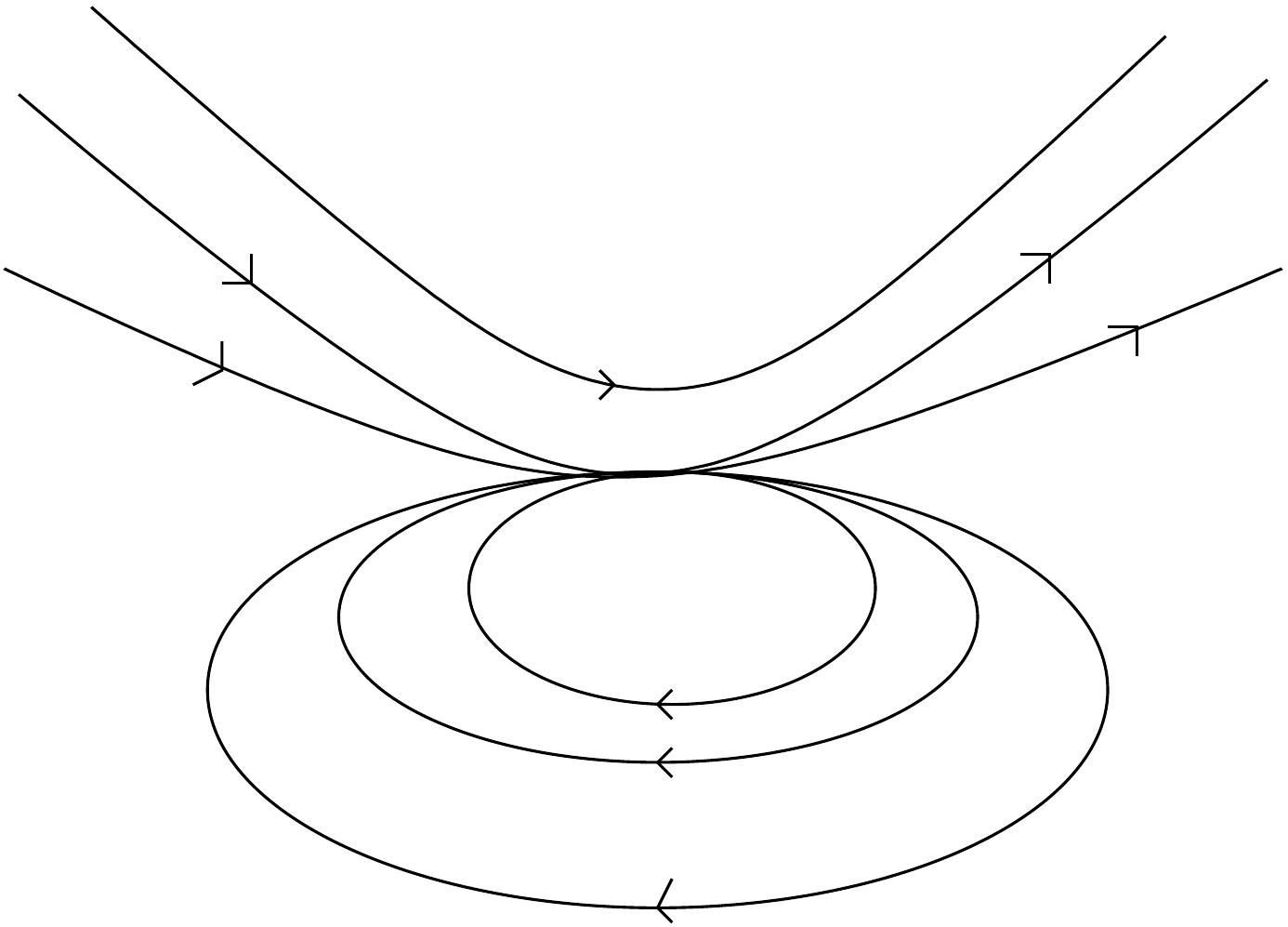}}   
\caption{The different topological types}
\label{fig.topotype}\end{center}\end{figure}

\bigskip

For $a\neq\frac12$, a generic unfolding depending on a multi-parameter $\lambda= (\mu_1,\mu_2,\mu_3, \mu)$ has the form 
\begin{align}\begin{split}
\dot x&=y+ a(\lambda)x^2 +\mu_2, \\
\dot y&=\mu_1+\mu_3 y + x^4h_1(x,\eps)+y(x + \eta x^2 +x^3 h_2(x,\lambda))+ y^2Q(x,y,\lambda),\end{split}\label{normal_form_family}\end{align}
where $h_1(x,\lambda)= O(|\lambda|)$. Moreover, $h_1, h_2, Q$ are $C^\infty$ functions, and $Q$ can be chosen of arbitrarily high order in $\lambda$. 

\subsection{Finite cyclicity of a graphic}
\begin{definition} A \emph{graphic} $\Gamma$ of a vector field $X_0$ , i.e. a union of trajectories and singular points, has \emph{finite cyclicity} inside a family $X_\lambda$ if there exists $N\in\mathbb N$, $\eps>0$ and $\delta>0$ such that any vector field $X_\lambda$ with $|\lambda|<\delta$ has at most $N$ periodic solutions  at a Hausdorff distance less than $\eps$ from  $\Gamma$. If a graphic has a finite cyclicity, its \emph{cyclicity} is the minimum of such numbers $N$.
\end{definition}

This means that when studying the finite cyclicity of a graphic $\Gamma$, we need to find a uniform bound for the number of periodic solutions that can appear from it, for all values of the multi-parameter in a small neighborhood $W$ of the origin. Typically we need to find a uniform bound for the number of fixed points of the Poincar\'e return map or, equivalently, for the number of zeros of some displacement map between two transversal sections to the graphic. With graphics containing a nilpotent singular point there is no way to make a uniform treatment for all $\lambda\in W$, and we typically cover $W$ by a finite number of sectors, on each of which we give a uniform bound. The method for doing this is the \emph{blow-up of the family}, which was first introduced  in \cite{R},
and next applied to slow-fast systems in \cite{DR}.

\subsection{Blow-up of the family}
We take the neighborhood of the origin in parameter-space of the form $\S^2\times [0,\nu_0)\times U$, where $U$ is a neighborhood of $0$ in $\mu$-space and we make the change of parameters
\begin{equation}
(\mu_1,\mu_2,\mu_3)= (\nu^3\overline{\mu}_1,\nu^2\overline{\mu}_2,\nu\overline{\mu}_3),\label{change_parameters}\end{equation}
where $\ov{M}=(\overline{\mu}_1,\overline{\mu}_2,\overline{\mu}_3)\in\S^2$ and $\nu\in [0,\nu_0)$.

\bigskip

Note that $\S^2$ is compact. Hence, to give an argument of finite cyclicity for the graphic $\Gamma$, it suffices to find  a neighborhood of each $\ov{M}=(\overline{\mu}_1,\overline{\mu}_2,\overline{\mu}_3)\in\S^2$ inside $\S^2$, a corresponding $\nu_0>0$ and a corresponding $U$ on which we can give a bound for the number of limit cycles. In our study, we will consider special values $a_0$ of $a$. It is important to note that $a(\lambda)$ depends on $\lambda$, and hence that $a-a_0$ is in some sense a parameter in itself. 

\bigskip

The way to handle this program is to do a \emph{blow-up of the family.}  For this, we introduce the weighted blow-up of the singular point $(0,0,0)$ of the three-dimensional family of vector fields,  obtained  by adding the equation $\dot \nu=0$  to \eqref{normal_form_family}. The blow-up transformation is given by 
\begin{equation}
(x,y,\nu) = (r\ov{x}, r^2\ov{y}, r\rho),\label{blow-up_family}\end{equation}
with $r>0$ and $(\ov{x},\ov{y},\rho)\in \S^2$.
After dividing by $r$  the transformed vector field, we get a family of $C^\infty$ vector fields $\ov{X}_{A}$, depending on the parameters $A=(a- a_0,\ov{M}, \mu)$. The foliation $\{\nu=r\rho=\mathrm{Cst}\}$ is invariant under the flow. The leaves $\{r\rho=\nu\}$, with $\nu>0$, are regular two-dimensional manifolds, while the critical locus $\{r\rho=0\}$ is stratified and contains the two strata (see Figure~\ref{fig.strat}):
\begin{itemize} 
\item $\S^1\times \R^+$ is the blow-up of $X_0$ (for $\lambda=0$);
\item $D_{\ov{\mu}}= \{\ov{x}^2+\ov{y}^2+\rho^2=1\mid \rho \geq0\}$,  for any $\ov{\mu}\in \S^2$.
\end{itemize}

\subsection{Limit periodic sets in the blown-up family}
 The vector field  $\ov{X}_A$ has  singular points on $r=\rho=0$. For $a\neq \frac12$, there will be four distinct singular points (occuring in two pairs) corresponding to $\ov{y}=0$ (for $P_1$ and $P_2$) and $\ov{y}=\frac{1-2a}{2}$ (for $P_3$ and $P_4$): see Figure~\ref{fig.strat}. Their eigenvalues appear in Table~\ref{eigenvalue}.
\begin{figure}[ht]\begin{center}
\subfigure[The saddle case]
   {\includegraphics[angle=0,width=6cm]{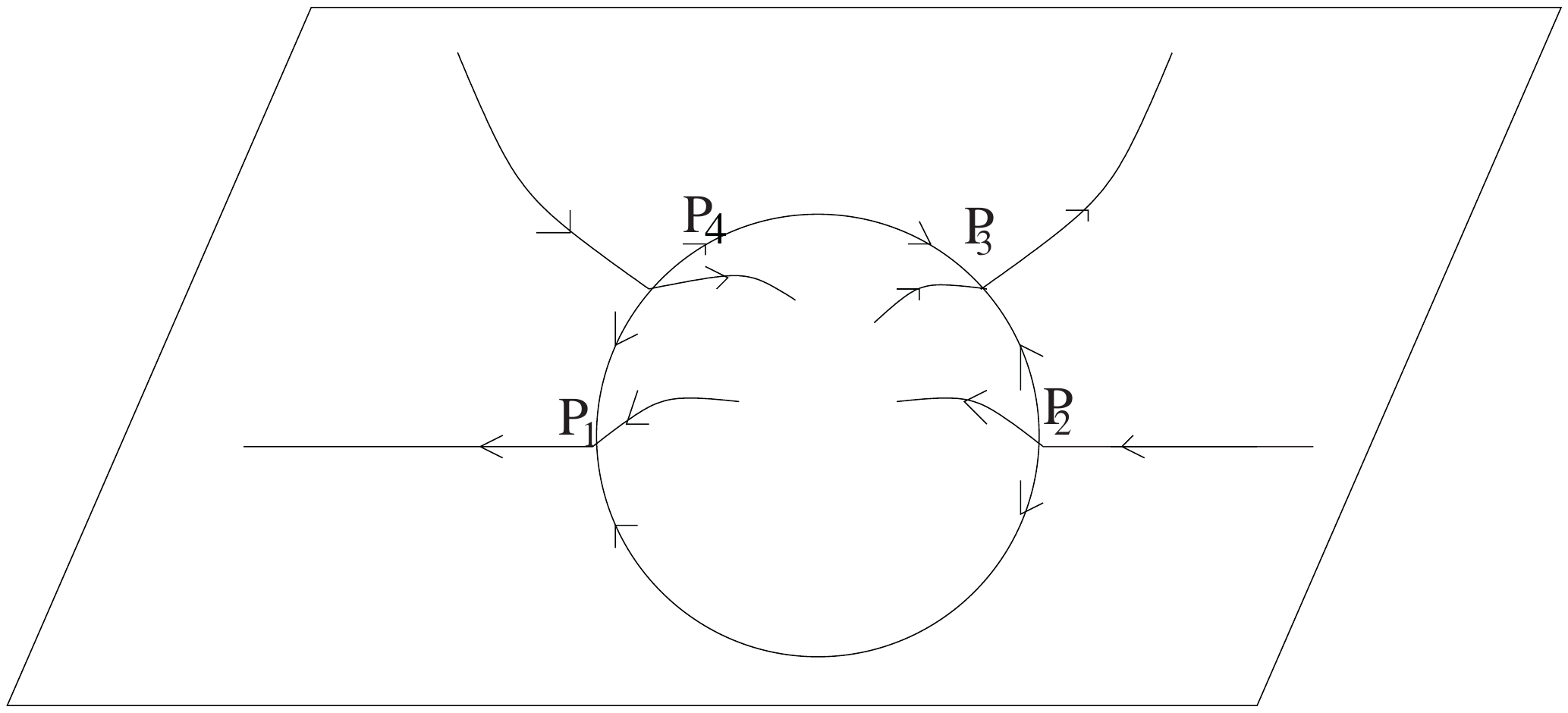}}\qquad
   \subfigure[The elliptic case]
   {\includegraphics[angle=0,width=6cm]{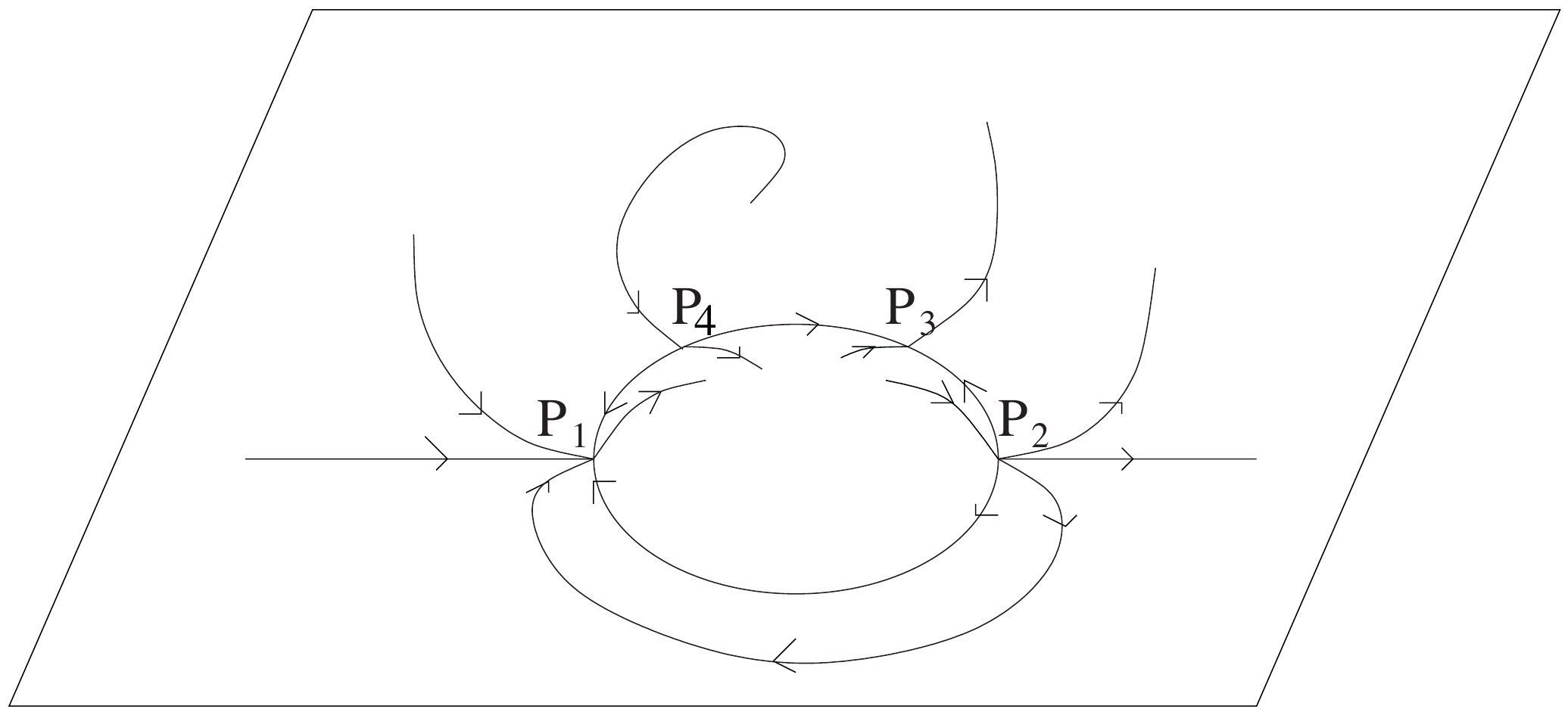}}
\caption{The stratified set $\{r\rho=0\}$ in the blow-up}
\label{fig.strat}\end{center}\end{figure}
\begin{table}[ht]\begin{center}\begin{tabular}{|c|c|c|c|}
\hline
      &  $r$       &  $\rho$      &    $ y $          \\ \hline
$P_1$ & $   -a   $ &  $\ \ a  $   &    $-(1-2a)$      \\ \hline
$P_2$ & $\ \ a   $ &  $   -a  $   &    $\ \ (1-2a)$   \\ \hline
$P_3$ & $\ \ 1/2 $ &  $ -1/2  $   &    $-(1-2a)$      \\ \hline
$P_4$ & $    -1/2$ &  $\ \ 1/2$   &    $\ \ (1-2a)$   \\ \hline
\end{tabular}\label{eigenvalue}
\caption{The eigenvalues at $P_i$ ($i=1,2,3,4$)}
\end{center}\end{table}

\bigskip

{\it We will study the finite cyclicity of a  graphic $\Gamma$ joining a pair of opposite points $P_i$ and $P_{i+1}$ in $\ov{X}$, with $i=1$ or $i=3$. }
We consider a particular value $A_0=(a_0, \ov{M}_0, \mu_0)$. Here is the strategy for finding an upper bound for the number of limit cycles that appear for $A$ in a neighborhood of $A_0$. We determine the phase portrait of the family rescaling \eqref{family_rescaling} on $D_{\ov{\mu}}$: this allows determining \emph{limit periodic sets} $\ov{\Gamma}$, which are formed by the union of $\Gamma$ with a finite number of trajectories and singular points on $D_{\ov{\mu}}$ joining $P_i$ and $P_{i+1}$, so that their orientation be compatible with that of $\Gamma$. The limit periodic sets to be studied appear in Table~\ref{tab.shhconvex} for the saddle case.  They come from studying the phase portrait of the \emph{family rescaling} 
\begin{align} \begin{split} 
\dot{\ov{x}}&= \ov{y}+a\ov{x}^2+\ov{\mu}_2,\\
\dot{\ov{y}}&=\ov{\mu}_1+\ov{\mu}_3\ov{y} +\ov{x}\,\ov{y},\end{split}\label{family_rescaling}
\end{align} 
obtained by putting $\rho=1$ and $r=0$. It then suffices to show that each limit periodic set has finite cyclicity, i.e. to show the existence of an upper bound for the number of periodic solutions of $\ov{X}_A$ for $A$ in a small neighborhood of $A_0$. 
\begin{table}[ht]
\begin{center}\begin{tabular}{|c|c|c|}
\hline
   \includegraphics[angle=0, width=3.2cm]{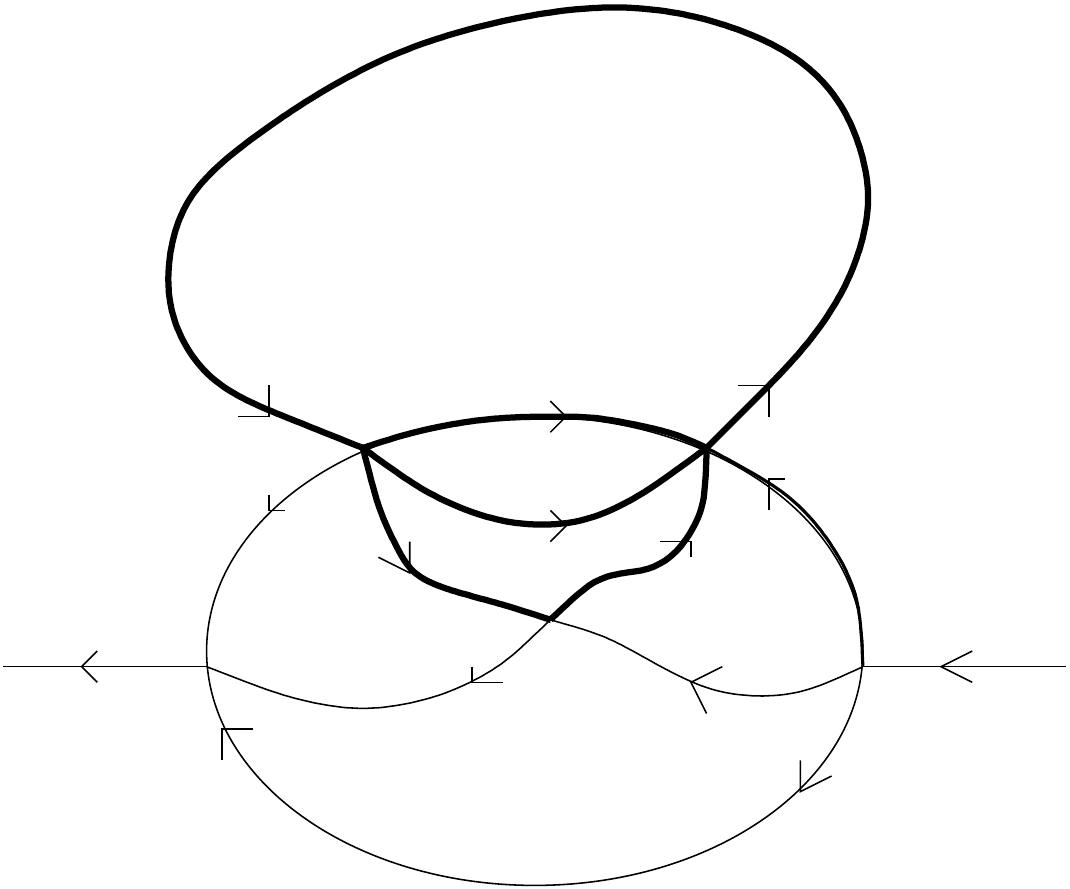}
&  \includegraphics[angle=0,width=3.2cm]{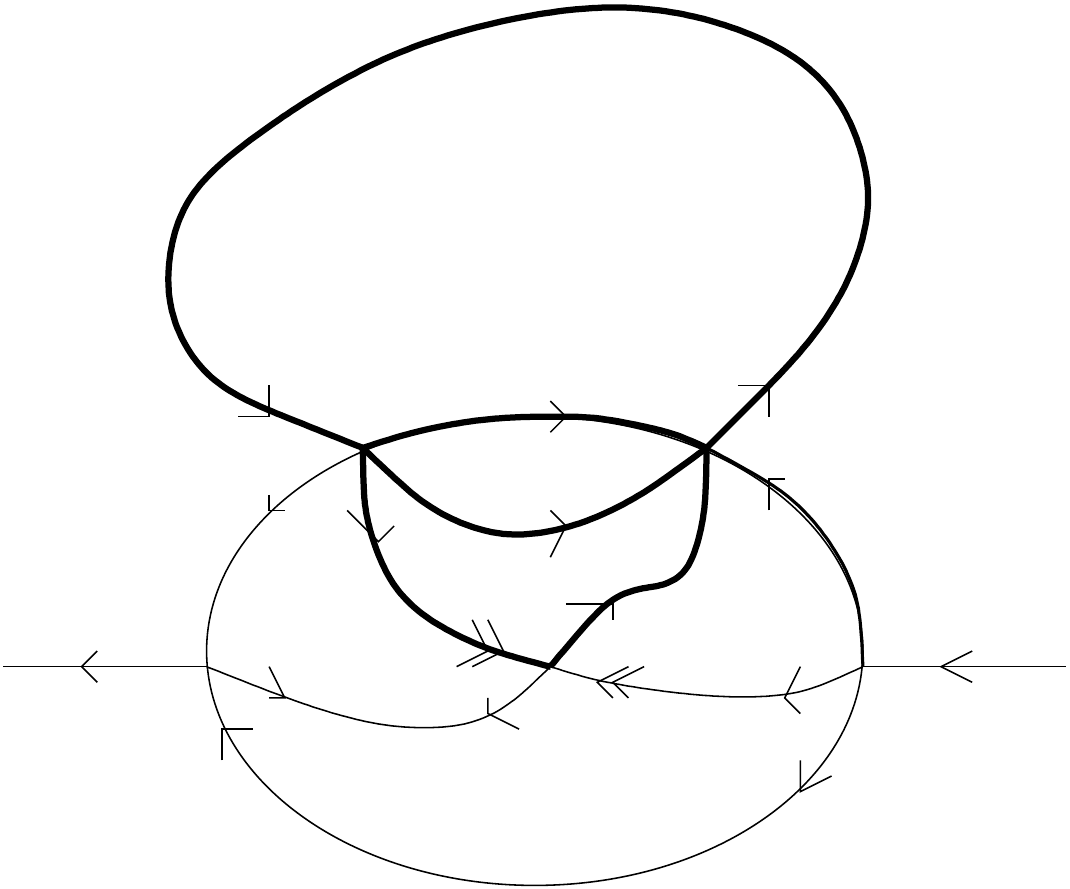}
&  \includegraphics[angle=0,width=3.2cm]{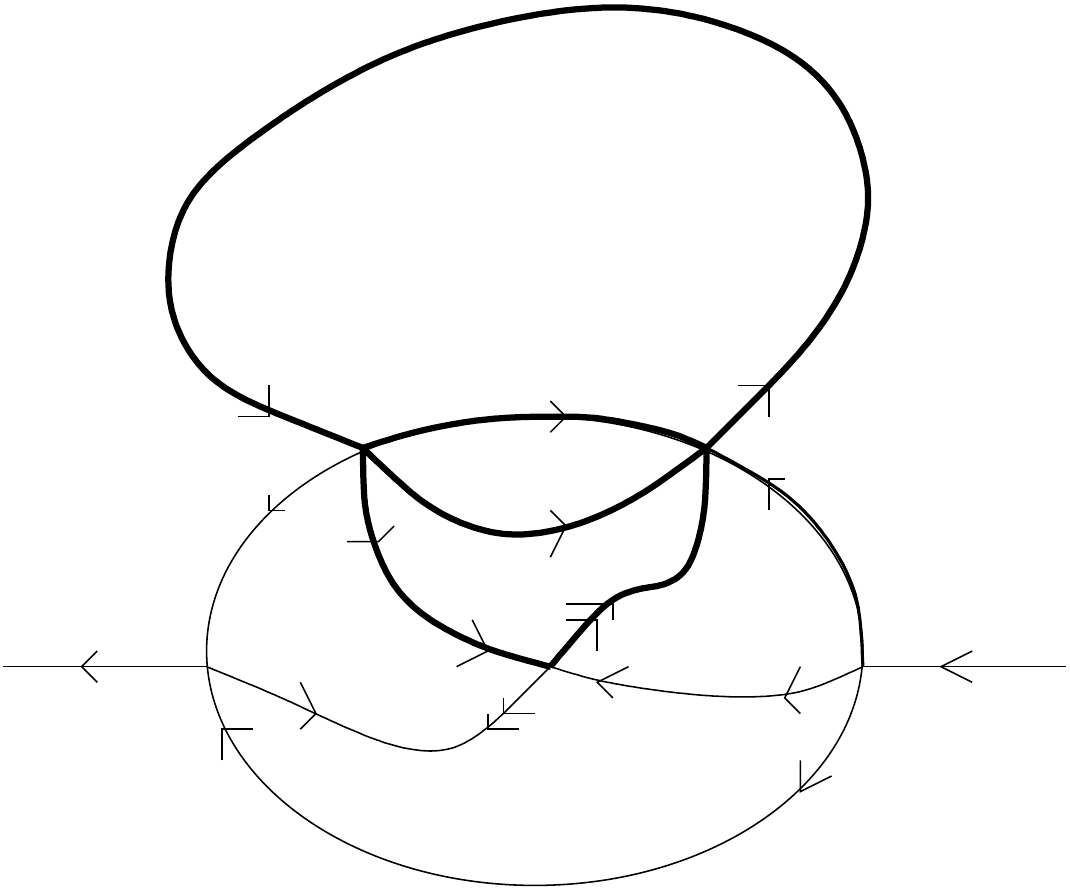}
\\ \hline   Sxhh1  &  Sxhh2 &  Sxhh3
\\ \hline
   \includegraphics[angle=0,width=3.2cm]{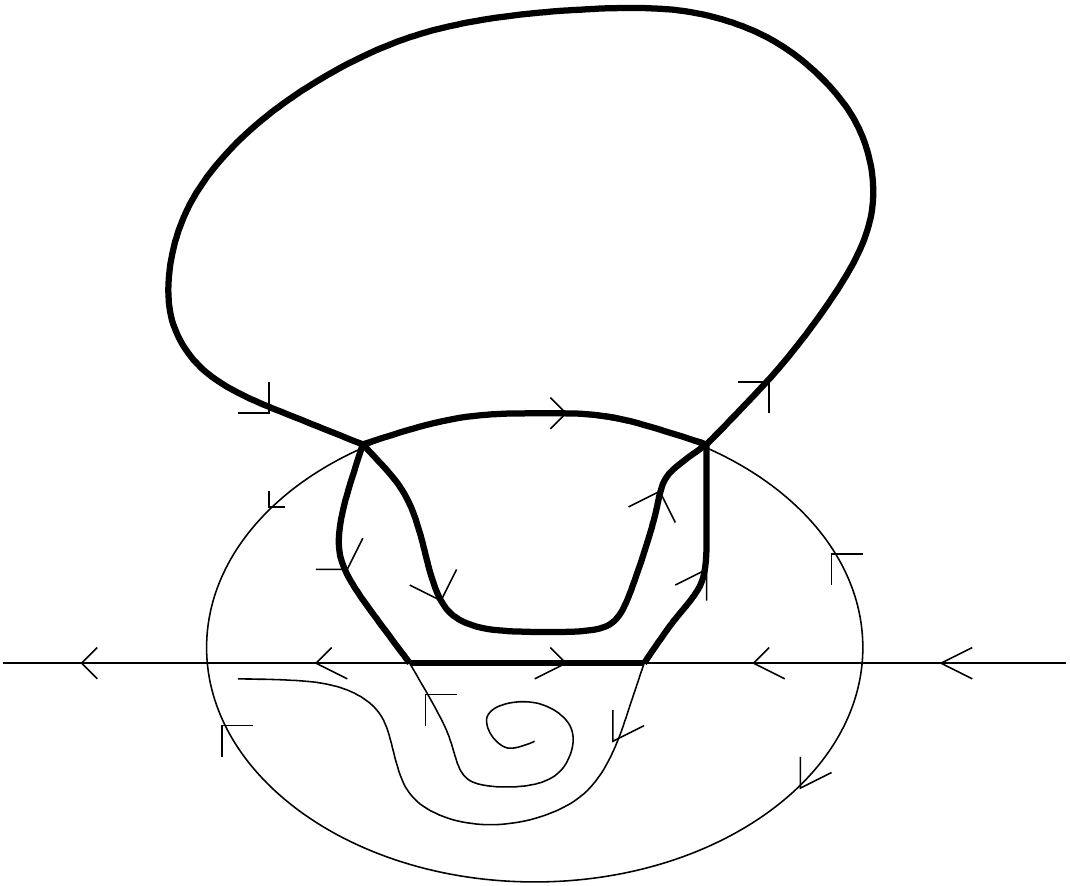}
&  \includegraphics[angle=0,width=3.2cm]{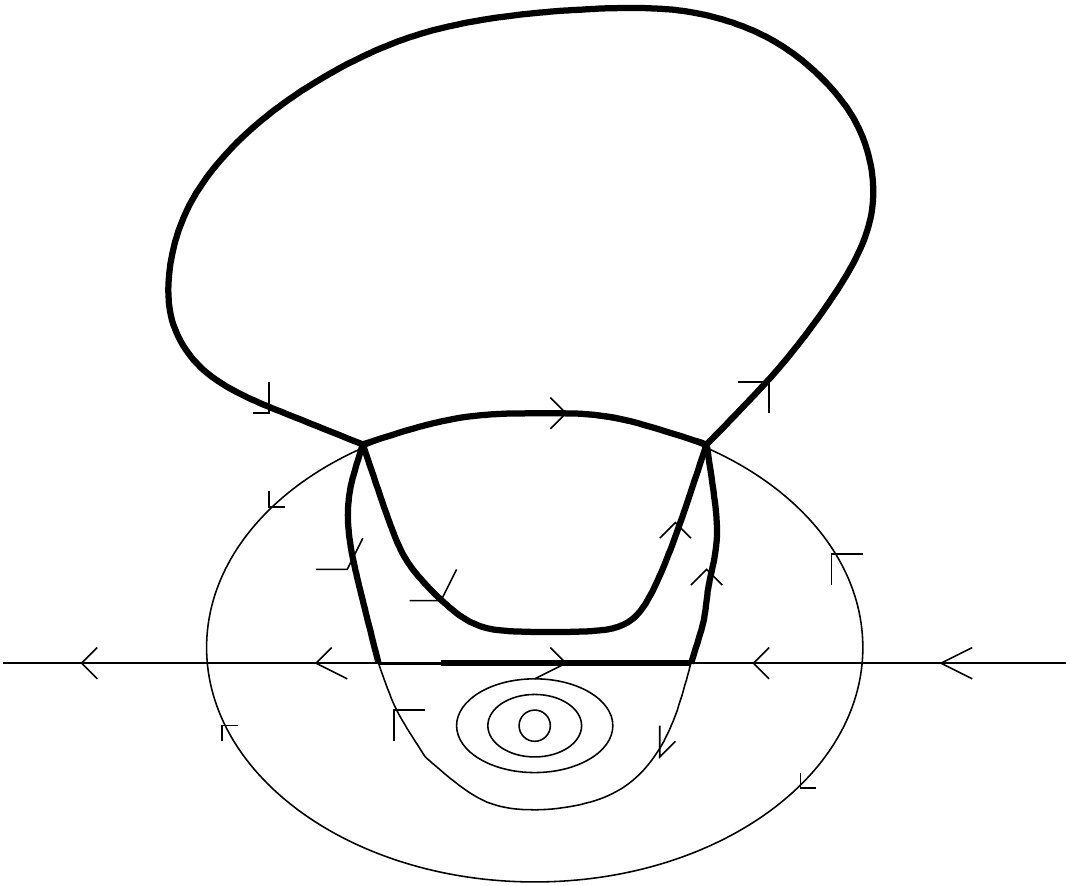}
&  \includegraphics[angle=0,width=3.2cm]{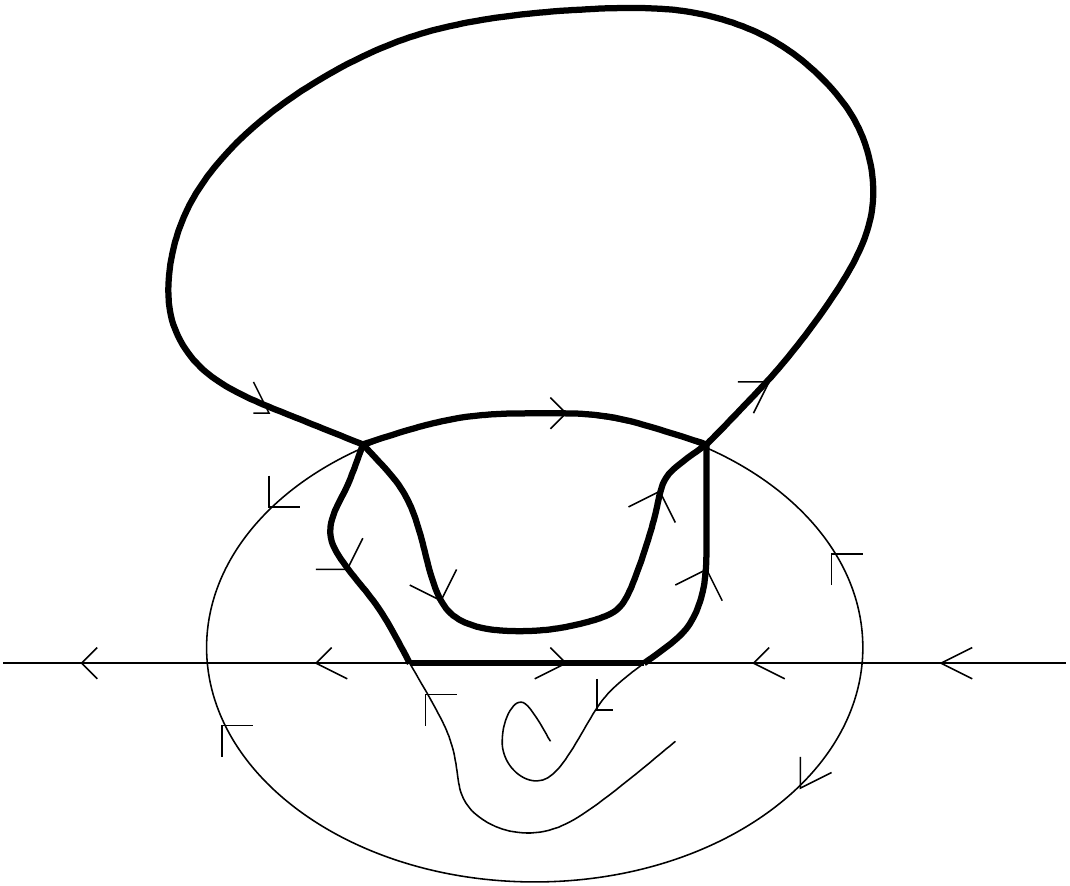}
\\ \hline  Sxhh4  &  Sxhh5  &  Sxhh6
\\ \hline
   \includegraphics[angle=0,width=3.2cm]{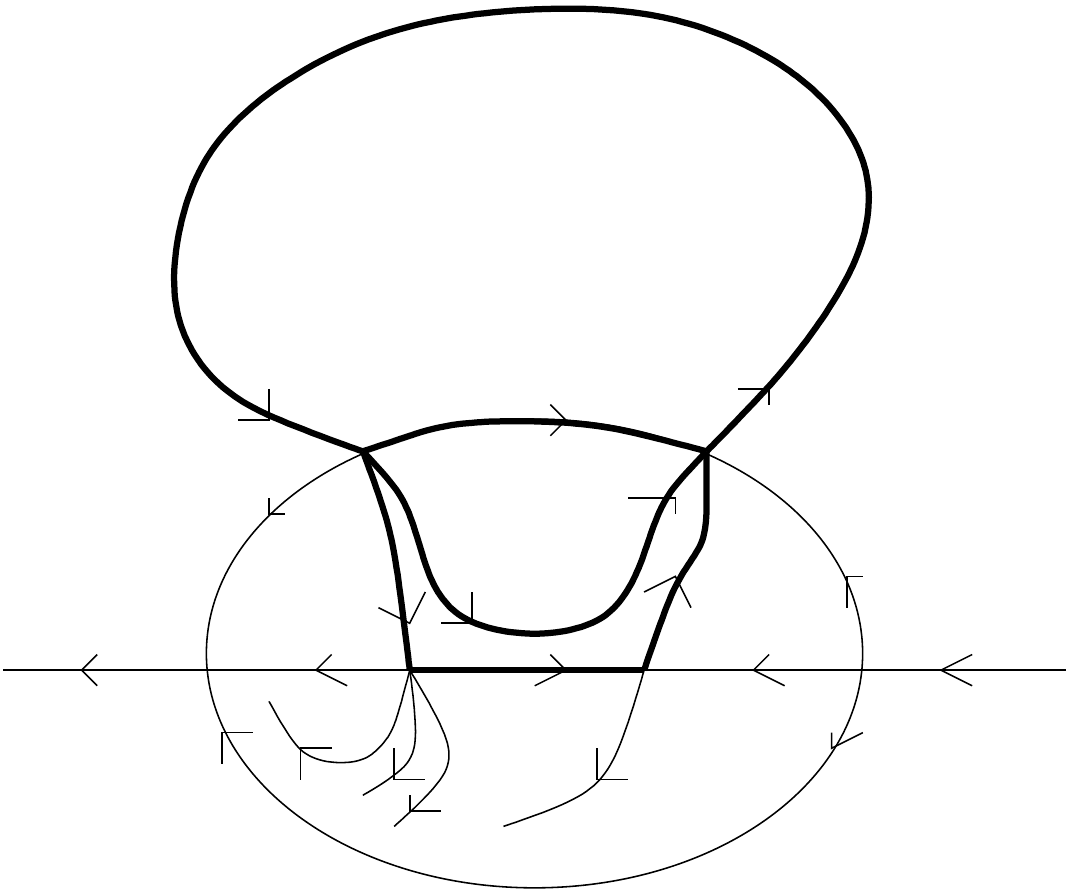}
&  &\includegraphics[angle=0,width=3.2cm]{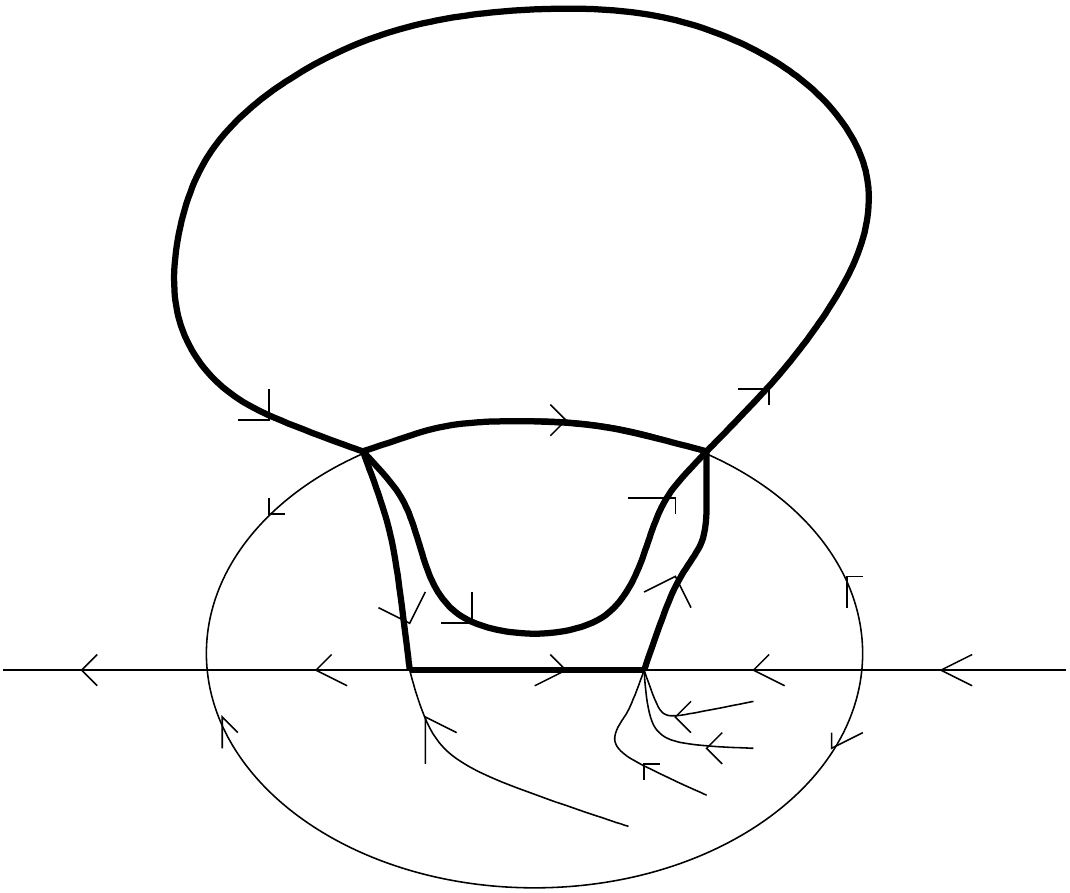}
\\  \hline  Sxhh7  &              &  Sxhh8
\\  \hline
   \includegraphics[angle=0,width=3.2cm]{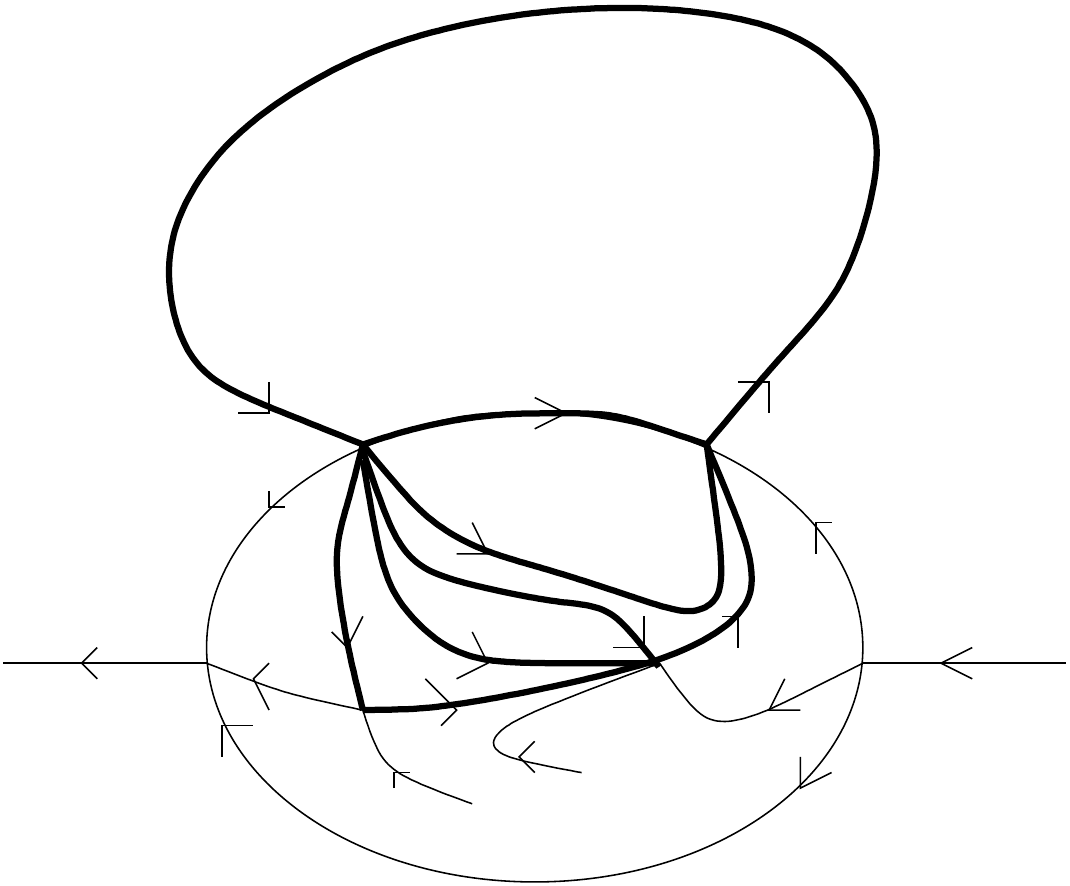}
&  &\includegraphics[angle=0,width=3.2cm]{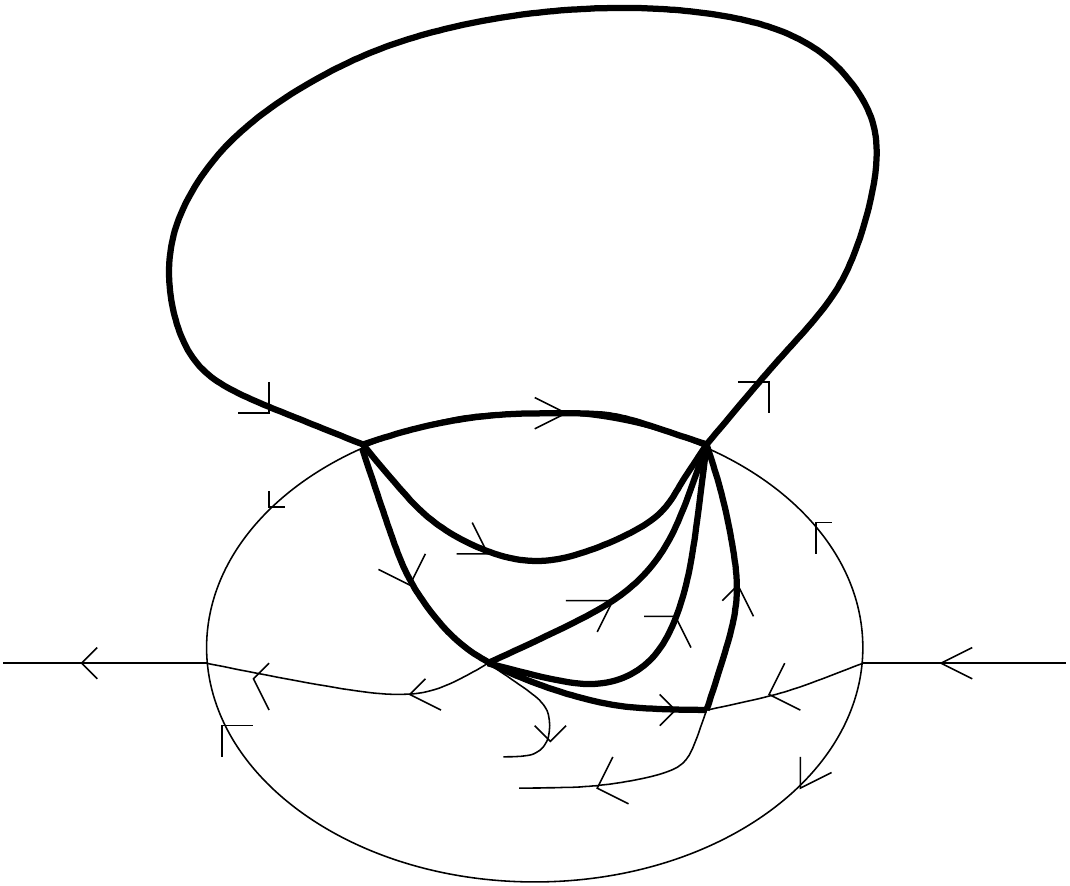}
\\   \hline  Sxhh9  &              &  Sxhh10
\\  \hline
\end{tabular}\end{center}
\caption{Convex limit periodic sets of hh-type for a graphic with
a nilpotent saddle.}
\label{tab.shhconvex}\end{table}

\subsection{Proving the finite cyclicity of a limit periodic set }
Typically, the kind of argument we will use for proving the finite cyclicity of a limit periodic set is the following: we look for the zeroes of a displacement map between two sections. The sections are 2-dimensional but, because of the invariant foliation, the problem can be reduced to a 1-dimensional problem and the conclusion follows by, either an iteration of Rolle's theorem, or its generalization, namely a derivation-division argument. The technique can be adapted to non generic graphics occurring inside integrable systems: the proof in the generic case is transformed into a proof for the corresponding graphic, using some adequate division of the coefficients of the displacement map in the ideal of conditions for integrability. 

\bigskip

To compute the displacement map, we decompose the related transition maps between sections into compositions of Dulac maps in the neighborhood of the singular points and regular $C^k$ transitions elsewhere. 

\subsection{Dulac maps} The Dulac maps are the transition maps in the neighborhood of a singular point on $r=\rho=0$. They are computed when the system is in $C^k$ normal form. The normalizing theorem is Theorem~\ref{thnormalformhyp} of Appendix I. There, it is proved that the normal form is obtained by a normalizing operator ${\cal N}$, a crucial property for this paper. The theorem establishes the existence of  a parameter-depending local change of coordinates of class ${\cal C}^k$ bringing the blow-up of  \eqref{normal_form_family} in the neighborhood of one of the points $P_i$ into the normal form $\ov{ X}_A^N$  (up to $t\mapsto -t$) written in normal form coordinates $(\ov{Y},r,\rho)$ (provided that the eigenvalue  in $r$ has a sign opposite to the two other eigenvalues). Using Table~\ref{eigenvalue}, we take  $\sigma=2(1-2a)$ near $\sigma_0=2(1-2a_0)$ for $P_3$ and $P_4$ when $a_0<\frac12$, and $\sigma=\frac{2a-1}{a} $ near $\sigma_0=\frac{2a_0-1}{a_0} $ for $P_1$ and $P_2$ when $a>\frac12$. The normal form $\ov{ X}_A^N$ is given by
 \begin{enumerate}
\item If $\sigma_0\not\in \Q:$
 \begin{equation}\label{eq2p}
\ov{X}^N_{A}:
\begin{cases}{\dot r}=r,   \\
{\dot \rho} =-\rho,\\
 \dot{\ov{Y} }=-(\sigma+\varphi_{A}(\nu))\ov{Y}.\end{cases}
\end{equation} 

\item  If $\sigma_0=\frac{p}{q}\in \Q,$ with $(p,q)=1$ when $q\not  =1:$
 
  \begin{equation}\label{eq3p}
\ov{X}^N_{A}:
\begin{cases}{\dot r}=r,   \\
{\dot \rho} =-\rho,\\
{\dot{\ov{Y}}} =-\Big(\sigma+\varphi_{A}(\nu)\Big)\ov{Y}+\Phi_{A}(\nu,r^p\ov{Y}^q)\ov{Y}+\rho^p\eta_{A}(\nu) , \end{cases}
\end{equation} 
with  $\eta_{A}\equiv 0$ when $\sigma_0\not\in \N$ ( $q\not  =1$).
\end{enumerate}
The functions  $\varphi_{A},\Phi_{A},\eta_{A}$ are polynomials  of degree $ \leq K(k)$ increasing with $k,$ with smooth coefficients in $A$  and $\Phi_{A}(\nu,0)\equiv 0.$
 
  We introduce the ``compensator''  function $\omega(\xi,\alpha)$, also denoted $\omega_\alpha(\xi)$, defined by
      \begin{equation}
\omega(\xi,\alpha)=\omega_\alpha(\xi)=\begin{cases}\frac{\xi^{-\alpha}-1}{\alpha}, &  \alpha\not =0,  \\
-\ln \xi, &\alpha=0.\end{cases}       \label{compensator}\end{equation}

We propose in Appendix I a new computation of the Dulac maps previously studied in \cite{ZR}. There are two types of Dulac transitions. The first type of transition map goes from a section $\{r= r_0\}$ to a section $\{\rho = \rho_0\}$, or the other way around. This type of transition typically behaves as an affine map, which is a very strong contraction or dilatation. The study of the number of  zeroes of a displacement involving only Dulac maps of the first type is reduced to the study of the number of zeroes of a 1-dimensional map. 

The second type of Dulac map is concerned with a transition map from a section $\{\ov{Y}=Y_0\}$ to, either a section $\{r=r_0\}$, or a section $\{\rho=\rho_0\}$. We take $\nu_0=r_0\rho_0.$

\subsubsection{First type of Dulac map} 

\begin{theorem}\label{thtranstypeI}
 We consider the Dulac map from the section $\{\rho= \rho_0\}$ to the section $\{r=r_0\}$, both parametrized by $(\ov{Y},\nu).$ 
 Let $$\bar\sigma=\bar \sigma(\sigma,\nu)=\sigma+\varphi_{A}(\nu)$$ and $$\alpha=\alpha(\sigma,\nu)=\bar\sigma(\sigma,\nu)-\sigma_0.$$ The $\ov{Y}$-component of the  transition map $D_{A}$ has the following expression:
\begin{enumerate}
\item If $\sigma_0\not\in \Q: $
\begin{equation}\label{eq20d}
D_{A}(\ov{Y},\nu)=\Big(\frac{\nu}{\nu_0}\Big)^{\bar\sigma} \ov{Y}.
\end{equation}
\item    If $\sigma_0=\frac{p}{q}\in \Q$ with $(p,q)=1$  when $\sigma_0\not \in \N:$

\begin{equation}\label{eq21d}
D_{A}(\ov{Y},\nu)=\eta_{A}(\nu)\rho_0^p\Big(\frac{\nu}{\nu_0}\Big)^{\bar\sigma}\omega\Big(\frac{\nu}{\nu_0},\alpha\Big)+\Big(\frac{\nu}{\nu_0}\Big)^{\bar\sigma}\Big(\ov{Y}+\phi_{A}(\ov{Y},\nu)\Big),
\end{equation}
with $\eta_{A}$  as  in (\ref{eq3p}). In particular, $\eta_{A}\equiv 0$ when $\sigma_0\not\in \N.$

The function family $\phi_{A}$ in   (\ref{eq21d}) is of order $O(\nu^{p+q\alpha}\omega^{q+1}\Big(\frac{\nu}{\nu_0},\alpha\Big)|\ln \nu|)$ and for any integer $l\geq 2,$   is of class ${\cal C}^{l-2}$   in $(\ov{Y},\nu^{1/l},\nu^{1/l}\omega\Big(\frac{\nu}{\nu_0},\alpha\Big), \nu,\mu,\sigma)$. 
\end{enumerate}

\end{theorem}

\subsubsection{Second type of Dulac map}

 \begin{theorem}\label{thtranstypeII}
We consider the Dulac map from the section $\{\ov{Y}=Y_0\},$  parametrized by $(r,\rho)$ to a section $\{r=r_0\}$ parameterized by $(\ov{Y},\nu)$. It has the form   $(r,\rho)\mapsto(D_A(r,\rho), \nu)$, with its $\ov{Y}$-component, $(D_A(r,\rho)$, given by:
\begin{enumerate}
\item If $\sigma_0\not\in \Q: $
\begin{equation}\label{eq18b}
D_{A}(r,\rho)=\Big(\frac{r}{r_0}\Big)^{\bar\sigma} Y_0.
\end{equation}
\item    If $\sigma_0=\frac{p}{q}\in \Q$ with $(p,q)=1$  when $\sigma_0\not \in \N:$

\begin{equation}\label{eq19b}
D_{A}(r,\rho)=\eta_{A}(\nu)\rho^p\Big(\frac{r}{r_0}\Big)^{\bar\sigma}\omega\Big(\frac{r}{r_0},\alpha\Big)+\Big(\frac{r}{r_0}\Big)^{\bar\sigma}\Big(Y_0+\phi_{A}(r,\rho)\Big),
\end{equation}
with $\eta_{A}$  as  in (\ref{eq3p}) ($\eta_{A}\equiv 0$ when $\sigma_0\not\in \N).$

The function family $\phi_{A}$ in   (\ref{eq19b}) is of order $O(r^{p+q\alpha}\omega^{q+1}\Big(\frac{r}{r_0},\alpha\Big)|\ln r|)$ and, for any integer $l\geq 2,$   is of class ${\cal C}^{l-2}$   in $(r^{1/l},r^{1/l}\omega\Big(\frac{r}{r_0},\alpha\Big),\rho,\mu,\sigma)$. 
\end{enumerate}

\end{theorem} 
\vskip5pt

\section{Applications to quadratic systems}\label{sect:quadratic}

\subsection{ Quadratic systems with a nilpotent singular point at infinity}

\begin{theorem}\label{thm.infty}
A quadratic system with a triple singularity point of
saddle or elliptic type at infinity and a finite singular point
of center type can be brought to the form
\begin{equation}
\left\{\begin{array}{ll}
\dot x&=-y+B_0x^2,\\
\dot y&=x+xy,
\end{array}\right.
\label{inf}
\end{equation}
with $B_0>0$. For $B_0\neq1$, the full $5$-parameter unfolding inside quadratic systems is given with $B= B_0+\mu_0$ inside the family
\begin{equation}
\left\{\begin{array}{ll}
\dot x&=-y+Bx^2 +\mu_2y^2 + \left(\mu_4+ B\mu_5\right)x\\
\dot y&=x+xy+\mu_3y^2+(1-2B)\mu_5y.
\end{array}\right.
\label{infunfold}
\end{equation}
For $B_0=1$, the full $5$-parameter unfolding inside quadratic systems is rather given with $B= 1+\mu_0$ inside the family
\begin{equation}
\left\{\begin{array}{ll}
\dot x&=-y+(1+\mu_0)x^2 +\mu_2y^2 + \mu_5x\\
\dot y&=x+(\mu_4+\mu_5)x^2+ xy+\mu_3y^2.
\end{array}\right.
\label{infunfold_B1}
\end{equation}
The parameter $\mu_2$ (resp. $\mu_3$) corresponds to a nonzero multiple of the parameter $\mu_2$ (resp. $\mu_3$) in the blow-up of the family at the singular point. There is no parameter $\mu_1$ in this family since the connection along the equator is fixed. 

Moreover for \eqref{inf} we have:
\begin{enumerate}
\item $B_0>1$ for a nilpotent saddle;

   $B_0=\frac{3}{2}$ corresponds to $a=-\frac{1}{2}$ in \eqref{normal_form_ZR}
   ($b=0$ in \eqref{normal_form_DRS}).
\item $B_0<1$ for an elliptic point;
    the elliptic point is of larger codimension, type 1 (the singular points in the blow-up coallesce by pairs)
         if $B_0=\frac{1}{2}$ (corresponding to $a=\frac{1}{2}$ in \eqref{normal_form_ZR}, i.e., $b=2\sqrt{2}$
         in \eqref{normal_form_DRS}).
\item  The system \eqref{infunfold} has an invariant line $y=-1$ if $\mu_3-(1-2B)\mu_5=0$.
\item If $\mu_2=\mu_3=\mu_4=0$, the system \eqref{infunfold} has an invariant parabola 
\begin{equation}
y=\frac{2B-1}{2}x^2+(2B-1)\mu_5 x -\frac1{2B} + (2B-1)\mu_5^2.
\label{Parabola.invariant}
\end{equation}
The parabola $y=\frac12x^2-\frac12$ is invariant for system \eqref{infunfold_B1} when $\mu_0=\mu_2=\mu_3=\mu_4=0$.
\item The integrability condition is $\mu_3=\mu_4=\mu_5=0$, for which we have the following graphics with return map
\begin{itemize} 
\item $B>1$: $(I_{14}^1)$,
\item $\frac12<B<1$: $(I_{6b}^1)$, 
\item $0<B<\frac12$: $(H_{13}^3)$,
\item $B=0$: $(H_{14}^3)$,
\item $B=1$: $(DI_{2b})$. 
\end{itemize} 
\item The value of ``$a$"  in the corresponding normal form \eqref{normal_form_family} is $a=1-B$, and the parameters $\mu_2$ and $\mu_3$ correspond to $\mu_2$ and $\mu_3$ up to a nonzero constant.

\end{enumerate}
\end{theorem}

\begin{proof}
We can suppose that the nilpotent singular point at infinity is
located on the y-axis,  the other singular point at infinity on the x-axis, and the focus or center at the origin.
Then the system can be brought to the form
\begin{equation}
\left\{\begin{array}{ll}
\dot x&=\delta_{10} x +\delta_{01} y   +\delta_{20} x^2   +\delta_{11}xy,\\
\dot y&=\gamma_{10} x  +\gamma_{01}  y   +\gamma_{11}  xy    +\gamma_{02} y^2.
\end{array}\right.
\label{inf.1}
\end{equation}

Localizing the system \eqref{inf.1} at the singular point at infinity
on y-axis by
$v=\frac{x}{y}, \ \ w=\frac{1}{y}$, we have
\begin{equation}
\left\{\begin{array}{ll}
\dot v&=(\delta_{11}-\gamma_{02})v-\delta_{01} w+(\delta_{20}-\gamma_{11})v^2
        +(\delta_{10}-\gamma_{01})vw-\gamma_{10}v^2w,\\
\dot w&=w(-\gamma_{02}-\gamma_{01} w -\gamma_{11}v-\gamma_{10} vw).
\end{array}\right.
\label{inf.2}
\end{equation}
For the singular point $(0,0)$ of system \eqref{inf.2} to be nilpotent,
we should have $\delta_{11}=\gamma_{02}=0$. The point is triple if
$\gamma_{11}\neq 0$. 

We want the finite singular point to be a center, which corresponds in this case to the system being reversible with respect to a line. Because of our choice of singular points at infinity this line can only be the $y$-axis. Then
$\delta_{10}=\gamma_{01}=0$.

By a rescaling and still using the original
coordinates $(x,y)$, we obtain the system \eqref{inf}.

The change of coordinates $W=-w+(B_0-1)v^2$ brings the system \eqref{inf.2} into the
equivalent form
\begin{equation}
\left\{\begin{array}{ll}
\dot V&=W\\
\dot W&=(B_0-1)V^3+(2B_0-3)VW+ o(V^3) + o(VW).
\end{array}\right.
\label{inf.4}
\end{equation}
The classification of the nilpotent singularity
at infinity follows.
\

A general unfolding preserving the singular point at the origin (which is simple) is of the form (after scaling of $x$, $y$, and $t$)
\begin{equation}
\left\{\begin{array}{ll}
\dot x&=-y+Bx^2+ m_{10}x +m_{11}xy+m_{02}y^2\\
\dot y&=x+xy+ n_{01}y+n_{20}x^2+n_{02}y^2,
\end{array}\right.
\label{inf_unfold}
\end{equation}
with $B$ close to $B_0$. We use a change of variable $(X,Y)= (x+ \zeta_1y, \zeta_2x+ y)$ for small $\zeta_1,\zeta_2$.
The terms in $XY$ in the expression of $\dot X$ and the term in $X^2$ in the expression of $\dot Y$ vanish precisely when 
$$\begin{cases}
(2B-1)\zeta_1-m_{11}(1+\zeta_1\zeta_2)+2\zeta_2m_{02}+ 2\zeta_1n_{02}(\zeta_1+\zeta_2)-\zeta_1^2\zeta_2=0,\\
(B-1)\zeta_2+(1+\zeta_2^2)n_{02}- \zeta_2^2n_{11}+\zeta_2^3m_{02}=0,\end{cases}
$$
which can be solved for $(\zeta_1,\zeta_2)$ by the implicit function theorem except for $B_0=1$. 
When $B_0=1$, we replace the second equation by the vanishing of the term in $Y$ in in the expression of $\dot Y$, namely
$$\zeta_1+\zeta_2- n_{01}+ m_{10}\zeta_1\zeta_2=0.$$
Again, we get a system that can be solved for $(\zeta_1,\zeta_2)$ by the implicit function theorem. \end{proof}

\subsection{Finite cyclicity of the boundary limit periodic sets of $(I_{14}^1)$,  $(I_{6b})$ and $(DI_{2b})$}\label{sect:boundary}

 \begin{notation} In the whole paper, $*$ denotes a nonzero constant, which may depend on some parameters.\end{notation}

\begin{theorem}\label{thm:boundary_graphic} The boundary limit periodic sets of $(I_{14}^1)$, $(I_{6b})$ and $(DI_{2b})$ (see Figures~\ref{graphics} (a), (b) and (d) and \ref{boundary_graphic}) have finite cyclicity.\end{theorem}
\begin{figure}\begin{center}
\includegraphics[width=5cm]{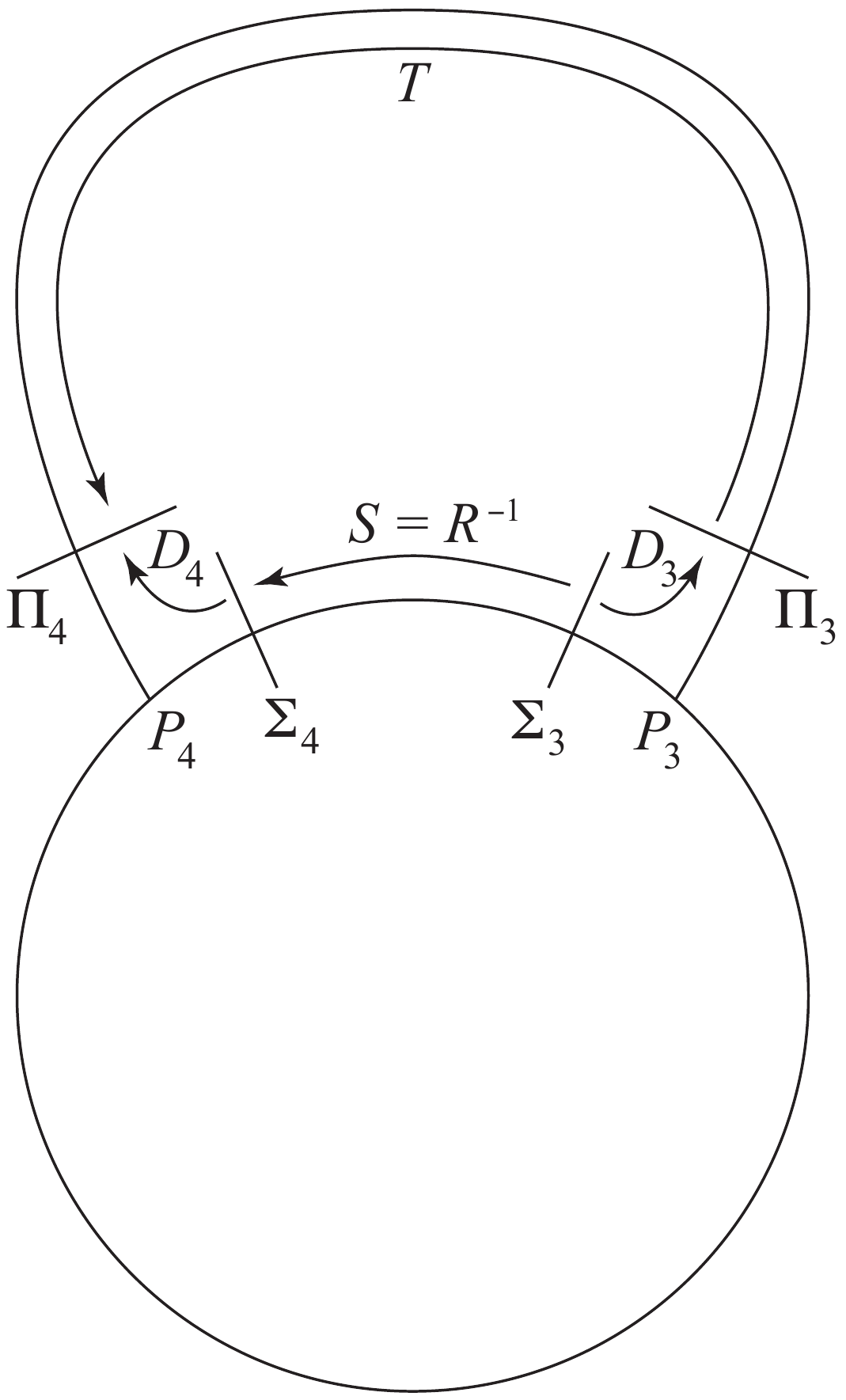}
\caption{The boundary graphic through $P_3$ and $P_4$ and the four sections $\Sigma_i$ and $\Pi_i$, $i=3,4$, in the normalizing coordinates.}\label{boundary_graphic}\end{center}\end{figure}
\begin{proof}
The finite cyclicity of the boundary limit periodic set is studied inside the family \eqref{infunfold} when $B_0\neq 1$, and we will discuss later the adjustment when $B_0=1$.

\medskip
\noindent{\bf Choice of parameters.} We  take as parameters \begin{equation}M=( \ov{\mu}_3, \mu_4,\mu_5, \ov{\mu}_2,B_0-1)= (M_C, \ov{\mu}_2, B_0-1), \label{par_M}\end{equation}
with $(\ov{\mu}_2,\ov{\mu}_3)\in\S_1$ and $(B_0-1,\mu_4,\mu_5)$ in a small ball. 
The parameters 
\begin{equation}M_C= (\ov{\mu}_3, \mu_4,\mu_5)\label{M_C}\end{equation}
unfold the integrable situation. We let $I_{C}$ be the ideal of germs of $C^k$-functions of the parameters generated by $\{\ov{\mu}_3, \mu_4,\mu_5\}$. 

\begin{notation} \begin{enumerate}
\item The symbol $O_P(M_C)$ refers to a function in the parameter $M$ belonging to the ideal $I_{C}$. 
\item The symbol $O_G(M_C)$ refers to a function of  $(X,M)$ which belongs to the ideal generated by $I_{C}$ inside the space of  functions of $(X,M).$ Depending on the limit periodic set, we could have $X=\ov{x}_3$, where $\ov{x}_3$ is the normalizing coordinate near $P_3$, or $X=(r,\rho)$. \end{enumerate}\end{notation}

\medskip

\noindent {\bf The displacement map.}
It is better to consider the chart $\ov{y}=1$ in the blow-up. We take $C^k$ normalizing charts in the neighborhood of $P_3$ and $P_4$. As discussed above, these $C^k$ normalizing charts can be chosen symmetric one to the other under the center conditions. The normalizing coordinates are $(r,\rho,\ov{x}_i)$ near $P_i$.
We consider sections $\Sigma_i= \{\ov{x}_i=X_0\}$ and $\Pi_i= \{r=r_0\}$ in the normalizing charts. The sections $\Sigma_i$ are parameterized by $(r,\rho)$, and the sections $\Pi_i$ by $(\ov{x}_i,\nu)$.

Let $V= D_4\circ S - T\circ D_3$ be the displacement map from $\Sigma_3$ to $\Pi_4$: $T$ and $D_3$ follow the flow forward, while $S$ and $D_4$ follow the flow backwards. 

Let us first give the proof when $\sigma_i(0)\notin\Q$. The Dulac maps are defined from sections $\Sigma_i=\{\ov{x}_i=X_0\}$ to sections $\Pi_i=\{r=r_0\}$, with $X_0$ and $r_0$ fixed.  
Then the Dulac maps $D_i$ have the form
\begin{equation}D_i(r,\rho)= (C_i(M) r^{\ov{\sigma}_i}, r\rho).\label{D_i}\end{equation}
We can choose $X_0$ and $r_0$ so that $C_i(0)=1$, i.e. $X_0r_0^{-\sigma_0}=1$, and $C_3(M)= C_4(M)$ under the center conditions. 

The map $T$ has the form
\begin{equation}T(\ov{x}_3,\nu)= (H(\ov{x}_3,\nu),\nu).\label{map_T}\end{equation}
Because of the symmetry of the sections, then $H\equiv id$ under the center conditions. 

The planes $r=0$ and $\rho=0$ are invariant under the map $S$, which hence has the form
\begin{equation}
S(r,\rho)= (r F(r,\rho), \rho F^{-1}(r,\rho)),\label{map_S} \end{equation}
with $F$ of class $C^k$, since $\nu=r\rho$ is invariant.
Moreover, it is known from \cite{ZR} that $F(0,0)=1$ when the sections $\Sigma_i$ are symmetric. 

The displacement map then has the form 
\begin{equation}
\Delta (r,\rho)= \left(C_4(M) r^{\ov{\sigma}_4}F^{\ov{\sigma}_4}(r,\rho)- H\left(C_3(M)r^{\ov{\sigma}_3}\right),\nu\right).\label{map_V}\end{equation}
Let $V(r,\rho)$ be the first component of $\Delta$. Then periodic solutions correspond to zeroes of $V$.  

We now need to compute $F$ and $H$. 

\medskip
\noindent{\bf Computation of $H$.}

The map $H$ is $C^k$ in $(\ov{x}_3, \nu)$. It has the form 
\begin{equation}
H(\ov{x}_3, \nu)= \ov{x}_3+ \eps_0(M) + \eps_1(M)\ov{x}_3 + O(\ov{x}_3^2)O_G(M_C),\end{equation}
with $\eps_0(M)=O_P(M_C), \eps_1(M)=O_P(M_C)$. 

For $\mu_2=\mu_3=\mu_4=0$, the system \eqref{infunfold} has the invariant parabola \eqref{Parabola.invariant}. The term $\mu_4 x$ in $\dot x$ is without contact,  which yields that 
\begin{equation}\eps_0(M)=*\mu_4(1+O(M))+O(\mu_3) + O(\mu_5)O(M)= *\mu_4(1+O(M))+O(\ov{\mu}_3\nu) + O(\mu_5)O(M),\label{eps_0}\end{equation}
where $*$ denotes a nonzero constant. Lemma~\ref{proof_eps_0} in Appendix II shows that the same is true for \eqref{infunfold_B1}. 
Let us again take $\mu_2=\mu_3=\mu_4=0$. The divergence is then $(2B+1)x+ (1-B) \mu_5$. 
Proposition~\ref{proof_eps_1} in the Appendix II shows that 
\begin{equation}\eps_1(M)= *\mu_5(1+O(M))+ O(\ov{\mu}_3\nu) + O(\mu_4). \label{eps_1}\end{equation}

\medskip
\noindent{\bf The center ideal.}
The equations \eqref{eps_0} and \eqref{eps_1} imply that we can take $\{\eps_0,\eps_1,\ov{\mu}_3\}$ as generators of the center ideal $I_C$.

\medskip
\noindent{\bf Computation of $F$.}
The function $F$ has the form:
\begin{equation}
F(r,\rho) = 1+*\ov{\mu}_3 \rho(1+ O(\rho)) + O(r)O_G(M_C).\label{eq_F}\end{equation}
Indeed, it is proved in Lemma~\ref{coef_rho} in the Appendix that the second derivative of $\rho F(0,\rho)$ is a nonzero multiple of $\ov{\mu}_3$. 
Moreover, the blown-up vector field is integrable on $r=0$ for $\ov{\mu}_3=0$.

\medskip
\noindent{\bf Writing the displacement as a finite sum of terms.}
We need grouping all terms of the displacement map into a finite sum of the form \eqref{type_V2}. We will see that three terms are sufficient and show that 
\begin{equation} V(r,\rho)= -\eps_0(M) (1+ h_0(r, \rho))- C_3(M) \eps_1(M)r^{\ov{\sigma}_3} (1+ h_1(r,\rho)) + *\ov{\mu}_3r^{\ov{\sigma}_3} \rho (1+ h_2(r,\rho)).\label{form_V}\end{equation}

We now explain how to group the different terms.  

\begin{notation}The symbol $O(r^\delta)$ used in the sequel, is for an unspecified  $\delta>0,$ which may vary from one formula to  the other.
\end{notation}

 Let us first consider the terms coming from 
$H\circ D_3$. Remember that $H$ is the identity when we have a center. Moreover, the map $H$ really takes place in the initial $(x,y)$-plane, where the center ideal is generated by $\{\eps_0,\eps_1,\mu_3\}$. Hence, the higher order terms of $H\circ D_3$ are of the form $$r^{2\ov{\sigma}_3} \left(\eps_0(M) k_0(r,\rho) + \eps_1(M)k_1(r,\rho) + \mu_3k_2(r,\rho)\right).$$
The first two terms contribute to $h_0(r,\rho)$ and $h_1(r,\rho),$ as contributions of order $O(r^\delta).$ As for the third term, we use the fact that $\mu_3=r\rho \ov{\mu}_3$. Hence it contributes to $h_2(r,\rho),$ also as a term of order $O(r^\delta).$  The term $C_3(M) r^{\ov{\sigma}_3}$ will be later grouped with the corresponding term $C_4(M) r^{\ov{\sigma}_4}$ coming from $D_4\circ S.$

Let us now consider the other terms coming from $D_4\circ S(r,\rho)= C_4(M) r^{\ov{\sigma}_4}F(r,\rho)^{\ov{\sigma}_4}$. Again we use that $F$ is the identity when there is a center, i.e. all its terms are divisible in the ideal $I_C$. One of them is the term $*\ov{\mu}_3 r^{\ov{\sigma}_4}\rho$ coming from the term  $*\ov{\mu}_3 \rho$ of $F$. As mentioned above, all higher order terms $r^{\ov{\sigma}_4}o(\rho)$ have coefficients divisible by $\ov{\mu}_3$. Also, all terms in  $r^{\ov{\sigma}_4}\rho O(r)$ can be distributed in $h_0$, $h_1$ and $h_2,$  as terms of order $O(r^\delta)$. Hence, we only need to consider the pure terms in $o(r^{\ov{\sigma}_4})$. It suffices to show that all such terms can be divided in $\{\eps_0,\eps_1\}$. This comes from the fact that the computation of the pure terms in $r$ can be done in the plane $\rho=0$, and that the system restricted to this plane does not contain any term in $\ov{\mu}_3$. Since
 \begin{equation} \ov{\sigma}_4-\ov{\sigma}_3=\nu O_P(M_C) f(\nu) = r\rho O_P(M_C) f(\nu),\label{sigma_3_m_4} \end{equation}
  with $f$ of class $C^k$, we can replace everywhere $\ov{\sigma}_4$ by $\ov{\sigma}_3,$ up to terms of order $O(r^\delta),$ distributed in $h_0,h_1$ and $h_2.$

We are left with the terms $C_3(M) r^{\ov{\sigma}_3} -C_4(M)r^{\ov{\sigma}_4}$. We write this as 
\begin{align}\begin{split}
C_3(M) r^{\ov{\sigma}_3} -C_4(M)r^{\ov{\sigma}_4}&= (C_3(M) - C_4(M)) r^{\ov{\sigma}_3} + C_4(M)(r^{\ov{\sigma}_3} - r^{\ov{\sigma}_4})\\
&=(C_3(M) - C_4(M)) r^{\ov{\sigma}_3} + C_4(M)(\ov{\sigma}_3-\ov{\sigma}_4)r^{\ov{\sigma}_3} \omega(r, \ov{\sigma}_3-\ov{\sigma}_4).\end{split} \end{align}
The difference $C_3(M) - C_4(M)$ is $X_0r_0^{-\ov{\sigma}_3}(1- r_0^{\ov{\sigma}_3-\ov{\sigma}_4})$. Using \eqref{sigma_3_m_4}, the two  terms can be decomposed in sums of terms contributing to $h_0,h_1,h_2$,  as terms of order $O(r^\delta).$

\medskip
\noindent{\bf Finite cyclicity in the case $\sigma_0$ irrational.}
The displacement map $V$ in \eqref{form_V} is a special case of a universal family
\begin{equation}a_0(1+ h_0(r,\rho)) + a_1r^{\ov{\sigma}_3} (1+ h_1(r,\rho)) +a_2 r^{\ov{\sigma}_3}\rho (1+ h_2(r,\rho)),\label{eq_V_irrational}\end{equation}
with $h_0,h_1$ of order $O(r^\delta)$ and $h_2$ is of order $O(\rho)+O(r^\delta).$ Using that these three functions are of order $o(1),$ 
we show in Theorem~\ref{thderdiv}
 below that this family has at most two small zeros along any curve $r\rho=\mathrm{Cst}$ for $r,\rho<\delta$ for some small $\delta$. 
This implies that, either $V$ has at most two small zeros, or $V$ is identically zero, in which case we have a center.

\medskip
\noindent{\bf Adjustment of the proof when $\sigma_0=\frac{p}{q}$ with $q>1$.}
The adjustments are minimal. Indeed, the formula of the Dulac map is more complicated:
\begin{equation}D_i(r,\rho)= (r^{\ov{\sigma}} (C_i(M) + \phi(r, \rho)), r\rho),\label{D_i_rational}\end{equation}
with $\phi(r,\rho)$ as in Theorem~\ref{thtranstypeII}.
Hence $\phi(r, \rho)$ produces in $V$ new terms of order $O(r^\delta),$ distributed  in $h_0,h_1,h_2.$ 

\medskip
\noindent {\bf Adjustement of the proof when $\sigma_0=p$.} 
Here the first component of $D_i(r,\rho)$ has an additional term of the form 
$$\kappa_i(r,\rho) = \eta_i(\nu)\rho^p r^{\ov{\sigma}_i}\omega\left(\frac{r}{r_0},\ov{\sigma}_i-p\right).$$
All higher order terms can be distributed  in $h_0,h_1,h_2$ and we need only consider the term $\tilde E=\kappa_4\circ S-(1+\eps_1(M))\kappa_3=\left(\kappa_4\circ S-\kappa_4\right)+E$ with  $E=\kappa_4(r,\rho)-(1+\eps_1(M))\kappa_3(r,\rho) $.

 \begin{enumerate} 
 \item We consider first the term $\kappa_4\circ S-\kappa_4.$ Let $\beta=\ov{\sigma}_4-p$.
 We have that
$$\kappa_4(rF)-\kappa_4(r)=\eta_4\nu^pr^\beta\underbrace{\Big[F^\beta\omega_\beta\Big(\frac{Fr}{r_0}\Big)-\omega_\beta\Big(\frac{r}{r_0}\Big)\Big]}_{G(r,\rho)}.$$
Let us consider $G(r,\rho)$:
$$G(r,\rho)=F^\beta\Big(\omega_\beta\Big(\frac{Fr}{r_0}\Big)-\omega_\beta\Big(\frac{r}{r_0}\Big)\Big)+(F^\beta-1)\omega_\beta\Big(\frac{r}{r_0}\Big).$$ 
Since
$$\omega_\beta\Big(\frac{Fr}{r_0}\Big)-\omega_\beta\Big(\frac{r}{r_0}\Big)=\Big(\frac{r}{r_0}\Big)^{-\beta}\frac{F^{-\beta}-1}{\beta},$$
we obtain that 
$$G(r,\rho)=-\frac{F^\beta-1}{\beta}\Big(\frac{r}{r_0}\Big)^{-\beta}+(F^\beta-1)\omega_\beta\Big(\frac{r}{r_0}\Big)=\frac{F^\beta-1}{\beta}\Big(-\Big(\frac{r}{r_0}\Big)^{-\beta}+\beta\omega_\beta\Big(\frac{r}{r_0}\Big)\Big),$$
i.e. $G(r,\rho)=-\frac{F^\beta-1}{\beta}$,  and then $\kappa_4(rF)-\kappa_4(r)=-\eta_4\nu^pr^\beta\frac{F^\beta-1}{\beta}.$

As $F=1+*\bar\mu_3\rho (1+\rho \bar g(\rho))+rO_G(M_C),$  we have that $$\frac{F^\beta-1}{\beta}=*\bar\mu_3\rho (1+\rho \bar g(\rho))+rO_G(M_C),$$
 and then that
$$\kappa_4(rF)-\kappa_4(r)=-\eta_4\nu^pr^\beta(*\bar\mu_3\rho (1+\rho g(\rho))+rO_G(M_C)).$$

The term $rO_G(M_C))$ gives contributions of order $O(r^\delta)$ in $h_0,h_1,h_2.$  Next, the term $*\bar\mu_3\rho (1+\rho \bar g(\rho))$  gives the contribution $-*\eta_4\nu^{p-1}\rho (1+\rho\bar g(\rho))$ in $h_2.$ If $p\geq 2,$ this term is also of order $O(r\rho),$ and it is of order $O(\rho)$ if $p=1.$

 \item
We consider now:
\begin{align*} \begin{split}
E&=\rho^p\left[\left(\eta_4(\nu)-\eta_3(\nu)(1+\eps_1(M)\right)r^{\ov{\sigma}_3}\omega\left(\frac{r}{r_0},\ov{\sigma}_3-p\right)\right. \\
&\qquad+ \eta_4(\nu) \left(r^{\ov{\sigma}_4}-r^{\ov{\sigma}_3}\right)\omega\left(\frac{r}{r_0},\ov{\sigma}_3-p\right)\\
&\qquad \left.+ 
\eta_4(\nu) r^{\ov{\sigma}_4}\left(\omega\left(\frac{r}{r_0},\ov{\sigma}_3-p\right)- \omega\left(\frac{r}{r_0},\ov{\sigma}_4-p\right)\right)\right].\end{split} \end{align*}

The second term in the bracket is of the form 
$$ \eta_4(\nu) (\ov{\sigma}_3 - \ov{\sigma}_4)r^{\ov{\sigma}_3}\omega(r, \ov{\sigma}_3-\ov{\sigma}_4)\omega\left(\frac{r}{r_0},\ov{\sigma}_3-p\right).$$
Using \eqref{sigma_3_m_4},
this term can be distributed in $h_0,h_1,h_2,$ {as terms of order $O(r^\delta).$}
A similar argument holds for the third term. Indeed, we introduce a compensator 
\begin{equation}\Omega(\xi, \alpha, \beta)= \Omega_{\alpha,\beta}(\xi)=\begin{cases}  \frac{\omega(\xi,\alpha) - \omega(\xi,\beta)}{\alpha-\beta}, &\alpha\neq\beta, \\
\frac12(\ln \xi)^2, &\alpha=\beta,\end{cases}\label{Omega}\end{equation}
allowing to rewrite this term as 
$$\eta_4(\nu) r^{\ov{\sigma}_4}(\ov{\sigma}_3 - \ov{\sigma}_4)\Omega\left(\frac{r}{r_0}, \ov{\sigma}_3-p, \ov{\sigma}_4-p\right).$$
Again, using \eqref{sigma_3_m_4}, this term can be distributed in $h_0,h_1,h_2,$ as terms of order $O(r^\delta)$.
\end{enumerate}

This allows writing the displacement map as a sum of four terms 
\begin{align}\begin{split} V(r,\rho)&= -\eps_0(M) (1+ h_0(r, \rho))- C_3(M) \eps_1(M)r^{\ov{\sigma}_3} (1+ h_1(r,\rho)) \\
&\qquad+ *\ov{\mu}_3r^{\ov{\sigma}_3} \rho (1+ h_2(r,\rho)) + K(M) r^{\ov{\sigma}_3}\rho^p \omega\left(\frac{r}{r_0},\ov{\sigma}_3-p\right),\label{form_Vp}\end{split}\end{align}
with $h_0,h_1$ of order $O(r^\delta).$ Moreover,  $K(M)=\eta_4(\nu)-\eta_3(\nu)(1-\eps_1(M))= O_P(M_C)$. For $p\geq2$,  we conclude that the cyclicity is at most $3$ by  Theorem \ref{thpgeq2}.

For $p=1$, we  will prove in Theorem \ref{thp1} that  the cyclicity is at most $2.$ To this end, we will  use that $\eta_4(0)= - \eta_3(0)=\ov{\mu}_3$ and then that   $K(M)=*\ov{\mu}_3+O(\nu)O_P(M_C),$ in order to rewrite $V$ as:
\begin{align}\begin{split} V(r,\rho)&= -\eps_0(M) (1+ h_0(r, \rho))- C_3(M) \eps_1(M)r^{\ov{\sigma}_3} (1+ h_1(r,\rho)) \\
&+*\ov{\mu}_3r^{\ov{\sigma}_3} \rho (1+ h_2(r,\rho))+*\ov{\mu}_3r^{\ov{\sigma}_3}\rho \omega\left(\frac{r}{r_0},\ov{\sigma}_3-p\right)(1+ h_3(r,\rho)), \label{form_Vp3}\end{split}\end{align}
with $h_0,h_1$ and $h_3$ of order $O(r^\delta).$
\end{proof}

\subsection{Finite cyclicity of the boundary limit periodic sets of $(H_{13}^3)$}\label{sect:boundaryH_13_3}

\begin{theorem}\label{thm:boundary_hemicycle} The boundary limit periodic set of $(H_{13}^3)$ (see Figures~\ref{graphics}(c) and \ref{boundary_hemicycle}) has finite cyclicity.\end{theorem}
\begin{figure}\begin{center}
\includegraphics[width=7cm]{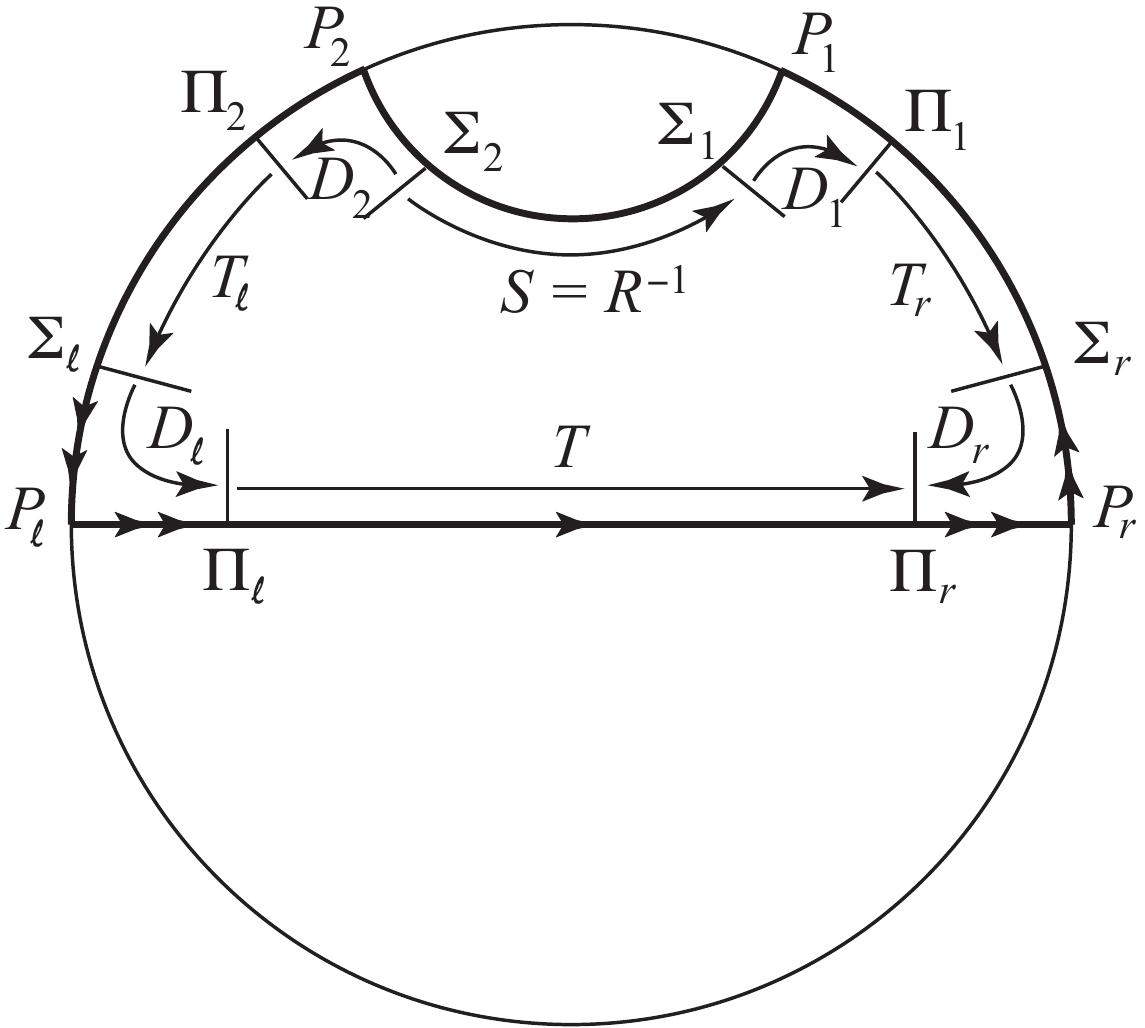}
\caption{The boundary graphic through $P_1$ and $P_2$ and the four sections $\Sigma_i$ and $\Pi_i$, $i=1,2$, in the normalizing coordinates.}\label{boundary_hemicycle}\end{center}\end{figure}
\begin{proof} The proof is very similar to that of Theorem~\ref{thm:boundary_graphic}. The graphic occurs in the family \eqref{inf} for $B<\frac12$, which corresponds to $\frac12<a<1$, but we prefer to use the following equivalent unfolding inside quadratic systems (only parameters' names are changed so that they play similar role as in Theorem~\ref{thm:boundary_graphic})
\begin{equation}
\left\{\begin{array}{ll}
\dot x&=-y+Bx^2 +\mu_2y^2 + \mu_5 x\\
\dot y&=x+xy+\mu_3y^2+\mu_4y.
\end{array}\right.
\label{infunfold_bis}
\end{equation}
The point $P_4$ (resp. $P_3$) is replaced by $P_1$ (resp. $P_2$). The quantity $\sigma_i$ is now given by $\sigma_i=\frac{2a-1}{a}$.
The main difference with Theorem~\ref{thm:boundary_graphic} is that the transition from $\Pi_2$ to $\Pi_1$ is replaced by the composition $T_r^{-1} \circ D_r^{-1} \circ T\circ D_\ell\circ T_\ell$. The transitions $T_\ell$ and $T_r$ are along the equator of the Poincar\'e sphere and hence preserve the connection (no translation terms). 
The saddle points $P_\ell$ and $P_r$ have inverse hyperbolicity ratios: $\tau_\ell= 1/\tau_r= \frac{1-B}{B}<1$.
Hence, it is better to consider a displacement map  
\begin{equation}V: \Sigma_2\rightarrow \Pi_r, \qquad V=T\circ D_\ell\circ T_\ell \circ D_2  - D_r\circ T_r\circ D_1\circ S.\label{displ_hemicycle}\end{equation}
The computation of $S$ is the same as before. 

\medskip
\noindent{\bf Computation of $T_\ell$ and $T_r$.} 
$T_r$ and $T_\ell$ are regular $C^k$-transitions with no translation terms. They can be computed in the coordinates $(v,w)= (-\frac{x}{y},\frac1{y})$. The transformed system in these coordinates is given in \eqref{coord_vw}.
The transitions take place along $w=0$. Along this line, $\mathrm{div}=(3-2B)v -2\mu_3$.
Hence $T_r'(0)-T_\ell'(0)=  O(\mu_3) = \nu O(\ov{\mu}_3)$.
This property is preserved in the normalizing coordinates. 

\medskip
\noindent{\bf Computation of $T$.} 
The transition $T$ in studied in \eqref{infunfold}. The line $y=-1$ is invariant under $\mu_3=\mu_4$. 
Hence, the constant term is of the form 
\begin{equation}T(0)=\eps_0(M)=*(\mu_4-\nu\ov{\mu}_3).\label{coef_const}\end{equation}
Under the condition $\eps_0=0$, we have  $\mathrm{div}|_{y=-1}= (2B+1)x+\mu_5-\nu\ov{\mu}_3$. 
Hence, \begin{equation}T'(0)= \eps_1(M)= *\mu_5+O(\mu_4) + O(\nu)O(\ov{\mu}_3).\label{coef_linear}\end{equation}
The equations \eqref{coef_const} and \eqref{coef_linear} remain valid in the normalizing coordinates, and we call the corresponding coefficients $\tilde{\eps}_0$ and $\tilde{\eps}_1$. 

\medskip
\noindent{\bf The Dulac maps $D_\ell$ and $D_r$. } We first localize the system \eqref{infunfold_bis} using coordinates $(u,z)= (\frac{y}{x},\frac1{x})$. The normalizing coordinates are of the form $(\ov{u}_i,z)$, $i\in\{\ell,r\}$. Then,
\begin{equation} 
D_i(z)=  \begin{cases} C_i(M) z^{\tau_\ell}, &\frac{1-B_0}{B_0} \notin \Q,\\ 
C_i(M) z^{\tau_\ell} (1 + \zeta(z, M)), &\frac {1-B_0}{B_0} \in \Q, \end{cases}\end{equation} 
with $\zeta$, a ${\cal C}^k$-function on monomials (see Appendix II).

\medskip
\noindent{\bf The Dulac maps $D_1$ and $D_2$. } They are given in Theorem~\ref{thtranstypeI}. Since the connection along the equator is fixed, then the coefficient $\eta_i$ vanishes identically when $\sigma_0\in \N$.

Hence, the displacement map $V(r,\rho)$ has the form
\begin{equation} V(r,\rho) =  \tilde{\eps}_0(1+ h_0(r,\rho)) + *\tilde{\eps}_1r^{\ov{\sigma}_2+\tau_\ell}(1+ h_1(r,\rho)) - *\ov{\mu}_3 r^{\ov{\sigma}_2+\tau_\ell}\rho(1+ h_2(r,\rho)).\end{equation}
This equation contains no resonant monomials since 
$\ov{\sigma}_2+\tau_\ell= \frac{1-B-B^2}{B(1-B)}\neq1$ as soon as $B\neq\frac12$. We conclude that the cyclicity is at most two by Theorem~\ref{thderdiv}. 
\end{proof}

\subsection{Finite cyclicity of $(I_{14}^1)$}
 We now prove Theorem~\ref{thMain2}, i.e. that 
the graphic $(I_{14}^1)$ has finite cyclicity inside quadratic systems (see Figure~\ref{graphics}(a)). 

\bigskip

\noindent{\it Proof of Theorem~\ref{thMain2}.} Such a graphic occurs for system \eqref{inf} when $B_0>1$, and its deformation in quadratic systems is given in \eqref{infunfold}. As usual, we should normally consider all limit periodic sets of Table~\ref{tab.shhconvex}. It was shown in \cite{ZR} that a graphic through a nilpotent saddle point has finite cyclicity inside any $C^\infty$-unfolding under the generic conditions that the return map $P$ along the graphic has a derivative different from one and that the nilpotent saddle point has codimension 3. But the only limit periodic sets of Table~\ref{tab.shhconvex} for which we use the genericity hypotheses are the boundary limit periodic sets which have been treated in Theorem~\ref{thm:boundary_graphic}, and the intermediate and lower limit periodic sets of Sxhh1 and Sxhh5.

For these limit periodic sets, we only have Dulac maps of the first type as in Theorem~\ref{thtranstypeI}. Hence, we can work with a 1-dimensional displacement map, which we take as $V: \Sigma_3\longrightarrow\Pi_4$, $V=D_4\circ S-T\circ D_3$ (see figure~\ref{other_graphic}). As before the sections $\Sigma_i$ and $\Pi_i$ are parameterized by the normalizing coordinate $\ov{x}_i$ near $P_i$, which are chosen so that $S$ and $T$ are the identity in the center case.  
\begin{figure}\begin{center}
\includegraphics[width=5cm]{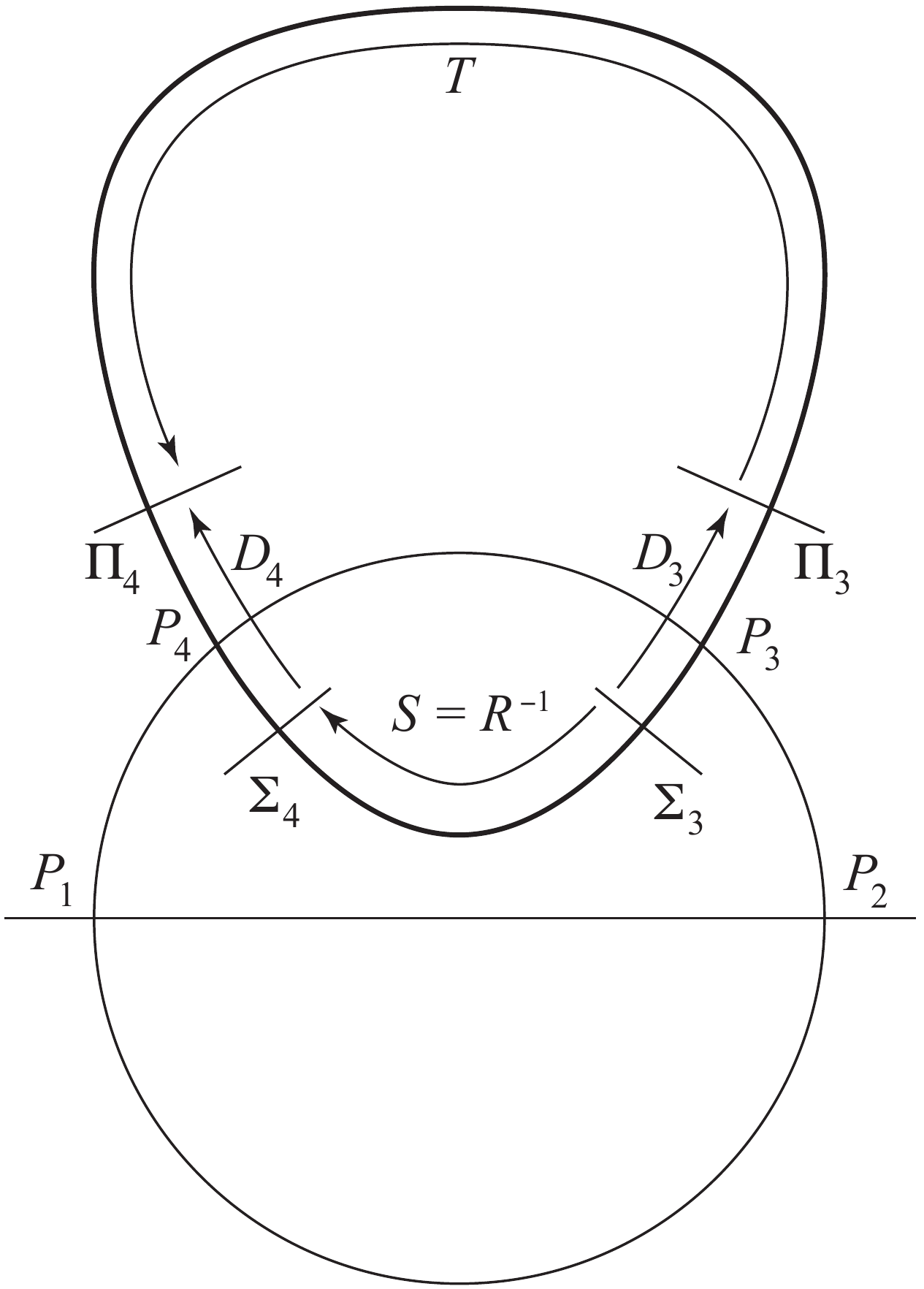}
\caption{Intermediate and lower limit periodic sets of Sxhh1 and Sxhh5: the four sections $\Sigma_i$ and $\Pi_i$, $i=3,4$, in the normalizing coordinates near $P_3$ and $P_4$.}\label{other_graphic}\end{center}\end{figure}

 The technique is to write $V$ in the form of a finite sum
 
\begin{equation}V(\ov{x}_3, \mu)=\tilde{\epsilon}_0+\nu^{\ov{\sigma}}\left( \sum_{i=1}^n \tilde{\eps}_ih_i(\ov{x}_3,\mu)\right), \label{finite_form}\end{equation}
for some $\ov{\sigma}>0$. The parameters are the same as in \eqref{par_M} and \eqref{M_C}. We write little details since they are very similar to \cite{RR}.
\medskip

\noindent{\bf The intermediate graphics.} For these graphics,  the map  $V(\ov{x}_3, \mu)$ is $C^k$ in $\ov{x}_3$. 
Under the condition $\mu_2=\mu_3=0$ for a nilpotent saddle, \eqref{infunfold} has an invariant parabola for $\mu_4=0$, which is the only possible connection at a nilpotent saddle. Hence, $T$ has a constant term of the form $*\mu_4+ O(\mu_3) + \mu_5O(M)$.
The constant term of the transition $S$ has the form $ O(\ov{\mu}_3)$ since $\ov{\mu}_2$ respects the symmetry, and hence does not contribute to the breaking of the connection. 

When $\sigma_0\notin \N$, this yields that the constant term $\tilde{\eps}_0$ in the displacement map has the form
$\tilde{\eps}_0= *\mu_4 + O(\nu)O(\ov{\mu}_3) + \mu_5 O(M).$

When $\sigma_0=p\in\N$, there are additional terms 
\begin{align}\begin{split}& \eta_3\rho_0^p\left(\frac{\nu}{\nu_0}\right)^{\ov{\sigma}_3}\omega\left(\frac{\nu}{\nu_0},\alpha_3\right)- \eta_4\rho_0^p\left(\frac{\nu}{\nu_0}\right)^{\ov{\sigma}_4}\omega\left(\frac{\nu}{\nu_0},\alpha_4\right)\\
&\qquad=(\eta_3- \eta_4)\rho_0^p\left(\frac{\nu}{\nu_0}\right)^{\ov{\sigma}_3}\omega\left(\frac{\nu}{\nu_0},\alpha_3\right)\\
&\qquad\qquad+ \eta_4(\alpha_3-\alpha_4)\rho_0^p\left(\frac{\nu}{\nu_0}\right)^{\ov{\sigma}_3}\omega\left(\frac{\nu}{\nu_0},\alpha_3-\alpha_4\right) \omega\left(\frac{\nu}{\nu_0},\alpha_3\right) \\
&\qquad \qquad+ \eta_4(\alpha_3-\alpha_4)\rho_0^p\left(\frac{\nu}{\nu_0}\right)^{\ov{\sigma}_4}\Omega\left(\frac{\nu}{\nu_0},\alpha_3, \alpha_4\right).
\end{split}\end{align}
In this expression $\eta_3-\eta_4=O_P(M_C)$ and $\alpha_3-\alpha_4= O_P(M_C)O(\nu)$. Hence, in all cases we have
\begin{equation}\tilde{\eps}_0= *\mu_4 + O(\nu)O(\ov{\mu}_3) + \mu_5 O(M)+ O(\nu)O_P(M_C).\label{eps_tilde_0}\end{equation}

The linear term has the form $\nu^{\ov{\sigma}_3} T'(0) - \nu^{\ov{\sigma}_4}S'(0))$. Moreover, $S'(0)|_{\rho=0}\equiv 1$ precisely when $\ov{\mu}_3=0$. Also, Lemma~\ref{proof_eps_1} shows that $T'(0)-1= *\mu_5+O(\mu_4)+O(\mu_3) $.
Considering that $\ov{\sigma}_3-\ov{\sigma}_4= O(\nu)$, then 
$$ \nu^{\ov{\sigma}_4}= \nu^{\ov{\sigma}_3}(1+ (\ov{\sigma}_3-\ov{\sigma}_4)\omega(\nu,\ov{\sigma}_3-\ov{\sigma}_4)= \nu^{\ov{\sigma}_3} (1+ O(\nu)).$$ This yields
\begin{equation}\tilde{\eps}_1= \nu^{\ov{\sigma}_3}\left(*\mu_5+ O(\mu_4)+O(\nu) O(\ov{\mu}_3)\right).\label{eps_tilde_1}\end{equation}

Now, because of the funneling effect, any nonlinearity on the side of $T$ has a high coefficient in $\nu$ which damps it. Hence, the only significant nonlinearities are on the side of $S$. We are sure that $S$ is nonlinear when $\ov{\mu}_3\neq 0$. This comes from the fact that the graphic belongs to a family of graphics. In the case of Sxhh1, this family ends in a lower graphic with a saddle point and its hyperbolicity ratio $\tau$ is different from $1$ precisely when $\ov{\mu}_3\neq0$, yielding that $S(\ov{x}_3)= C_0+ C_1\ov{x}_3^{\tau} + o(\ov{x}_3^\tau),$ with $C_1\not =0,$  for graphics near the saddle point, and hence  that $S$ is nonlinear on the whole section $\Sigma_3$. Then, for any graphic occuring for a value $\ov{x}_{3,0}$, there exists $n$ such that $S^{(n)}(\ov{x}_3)= c_{n,3}\ov{\mu}_3\neq0$. Hence, $V^{(n)}(\ov{x}_{3,0})= \nu^{\ov{\sigma}_4}\left[c_{n,3}\ov{\mu}_3+ O(\nu)O_P(M_C)\right]= \tilde{\eps}_n$. Moreover, for all graphics except a few isolated ones we have that $n=2$. The same argument can be applied for Sxhh5 since the connection is fixed between the two saddles and the product of their hyperbolicity ratios is different from $1$ precisely when $\mu_3\neq0$. 
Hence, we have written $V$ under the form \eqref{finite_form} with $h_i(\ov{x}_3)=\ov{x}_3^i(1+O(\ov{x}_3))$. We conclude to finite cyclicity by means of Theorem~\ref{thderdiv}.\medskip

{\bf The lower graphic of Sxhh1. } The study  is very similar and divided in two cases. When $\ov{\mu}_3\neq0$, it was already shown in \cite{ZR} that the lower graphic of Sxhh1 has finite cyclicity. This comes from the fact that the hyperbolicity ratio $\tau$ at the saddle point is non equal to $1$ precisely when $\ov{\mu}_3\neq0$, in which case we conclude to finite cyclicity because of the nonlinearity of $S$.  Hence the difficult case is the neighborhood of $\ov{\mu}_3=0$ since, for this value, $\tau_0=1$. In that case  we reparameterize the section $\Sigma_3$ by means of $\tilde{x}_3= \ov{x}_3- c_0(M)$, so that $\tilde{x}_3=0$ corresponds to the unstable manifold of the saddle point on the blow-up sphere. Then, as before, we write $V$ as a sum of  terms:
\begin{equation}
 V(\tilde{x}_3,M)=\tilde{\eps}_0h_0(\tilde{x}_3,M)+\ov{\mu}_3 \tilde{x}_3\omega(\tilde{x}_3,\tau-1)h_3(\tilde{x}_3,M) + \tilde{\eps}_1\tilde{x}_3h_1(\tilde{x}_3,M), \end{equation}
 with $h_i(0,0)\neq0$. We conclude to finite cyclicity by means of Theorem~\ref{thderdiv}.

\medskip

{\bf The lower graphic of Sxhh5. } Such a graphic occurs for $\ov{\mu}_2>0$. Because the connection is fixed between the two saddles, the map $S$ can easily be computed and has the form $c_0+c_1\ov{x}_3^\tau + o(\ov{x}_3^\tau)$, where $\tau= 1-\frac{2\ov{\mu}_3}{\sqrt{-\frac{\ov{\mu}_2}{a}}+\ov{\mu}_3}$ is the product of the two hyperbolicity ratios. Again, we reparameterize the section $\Sigma_3$ by means of $\tilde{x}_3= \ov{x}_3- c_0(M)$, so that $\tilde{x}_3=0$ corresponds to the unstable manifold of the right saddle point on the blow-up sphere. This allows writing the map $V$ in the form 
$$\begin{cases}
V(\tilde{x}_3)=\sum_{i=0}^{\max(\lfloor \tau\rfloor,1)} \tilde{\eps}_i \tilde{x}_3^ih_i(\tilde{x}_3,M) + \ov{\mu}_3 \tilde{x}_3^\tau h_\tau(\tilde{x}_3,M),&\tau_0\notin\N,\\
V(\tilde{x}_3)=\sum_{i=0}^{\tau_0} \tilde{\eps}_i \tilde{x}_3^ih_i(\tilde{x}_3,M) + \ov{\mu}_3 \tilde{x}_3^{\tau_0}\omega(\tilde{x}_3,\tau-\tau_0) h_\tau(\tilde{x}_3,M),&\tau_0\in\N,\end{cases}$$
with $h_i(0,0)\neq0$. We conclude to finite cyclicity by means of Theorem~\ref{thderdiv}. \hfill $\Box$
\vfill\eject

\section{Appendix I --- Hyperbolic fixed points}
We will consider  
 germs of smooth  family of $3$-dimensional vector fields $X_{\mu,\sigma}$ at $(0)\in \R^{3},$ with coordinates $(u,v,y),$ which are quasi-linear of the form:

\begin{equation}\label{eq1}
X_{\mu,\sigma}:
\begin{cases}{\dot u}= u,  \\
{\dot v} =-v,\\
 {\dot y} = -\sigma y+F_\mu(u,v,y),    \end{cases}
\end{equation}
where   $\sigma$  is a parameter in a neighborhood of $\sigma_0\in \R^+$,  and $\mu$ a parameter in a neighborhood of $\mu_0$ in some Euclidean space.
Moreover, $F_\mu= O(|(u,v,y)|^2)$ at the origin, for any value of the parameter $(\mu,\sigma).$ {\it  The system has the first integral: $\nu=uv.$}
\vskip10pt

\subsection{\bf Normal form}

It is possible to find  local normal form coordinates for $X_{\mu,\sigma}$ by a coordinate change preserving the coordinates $u$ and $v.$ 
More precisely, 
 we have  the following normal form result:
\begin{theorem}\label{thnormalformhyp}
There exists a normalizing operator $\mathcal{N}$ defined on each pair $(X_{\mu,\sigma},k)$, where $X_{\mu,\sigma}$ is a family  as above  and $k\in \N^*$, such that 
, $$\mathcal{N}(X_{\mu,\sigma},k)= \left(\delta_k, K(k), \eps_k, \eta_k,G_{\mu,\sigma}\right),$$
where 
 $$(u,v,y)\rightarrow (u,v,Y=G_{\mu,\sigma}(u,v,y)),$$
is a parameter-depending change of coordinates of class $C^k$ defined 
 defined for $|\sigma-\sigma_0|\leq \delta$, $|\mu-\mu_0|<\eps_k,$ and $|(u,v,y)|<\eta_k$, such that $dG_{\mu,\sigma}(0,0,0)=\mathrm{Id},$ 
 which brings $X_{\mu,\sigma}$ to the following polynomial normal form of degree $K(k)$:

\begin{enumerate}
\item If $\sigma_0\not\in \Q:$
 \begin{equation}\label{eq2}
X^N_{\mu,\sigma}:
\begin{cases}{\dot u}=u,   \\
{\dot v} =-v,\\
 {\dot Y} =-(\sigma+\varphi_{\mu,\sigma}(\nu))Y.\end{cases}
\end{equation} 

\item  If $\sigma_0=\frac{p}{q}\in \Q,$ with $(p,q)=1$ when $q\not  =1:$
 
  \begin{equation}\label{eq3}
X^N_{\mu,\sigma}:
\begin{cases}{\dot u}=u,   \\
{\dot v} =-v,\\
 {\dot Y} =-\Big(\sigma+\varphi_{\mu,\sigma}(\nu)\Big)Y+\Phi_{\mu,\sigma}(\nu,u^pY^q)Y+v^p\eta_{\mu,\sigma}(\nu) , \end{cases}
\end{equation} 
with  $\eta_{\mu,\sigma}\equiv 0$ when $\sigma_0\not\in \N$ ( $q\not  =1$).
\end{enumerate}
The functions  $\varphi_{\mu,\sigma},\Phi_{\mu,\sigma},\eta_{\mu,\sigma}$ are polynomials  of degree $ \leq K(k)$, with $C^\infty$ coefficients in $(\mu,\sigma)$  and $\Phi_{\mu,\sigma}(\nu,0)\equiv 0.$

\end{theorem}
\begin{proof}
The proof is standard in the literature, and we only recall the main steps. 

The degree $K(k)$ can be determined algorithmically from the eigenvalues $\{1,-1,-\sigma_0\}$. 

The number $\delta_k$ is chosen sufficiently small so as not to introduce any new resonant terms of degree $\leq K(k)$ for some $\sigma\in [\sigma_0-\delta_k, \sigma_0+\delta_k]$.

The first step is to bring the system to normal form up to degree $K(k)$
\begin{equation}\label{eqR}
X^p_{\mu,\sigma}:
\begin{cases}{\dot u}=u,   \\
{\dot v} =-v,\\
 {\dot z} =P(\sigma,\mu, u,v, z) + R(\sigma,\mu, u,v, z).\end{cases}
\end{equation} 
where $P(\sigma,\mu, u,v, z)$ is a polynomial in $u,v,z$ of degree $K(k)$ containing only resonant terms, and $R(\sigma,\mu, u,v, z)= o(|(u,v,z)|^{K(k)}$. This can be done by means of a polynomial change of coordinate $$y= z+ \sum_{\substack{i+j+\ell=2\\ i-j+\sigma_0(\ell-1)\neq0}}^{K(k)} a_{ij\ell}r^i\rho^jz^\ell.$$
Because this change of coordinate is tangent to the identity and contains no resonant monomial, then it is uniquely determined.

The second step is to kill the remainder $R$ in \eqref{eqR}. For this purpose, we decompose $R$ as $R= R_1+R_2$, with 
$R_1= O(u^{\lfloor K(k)/2\rfloor})$ and $R_2=O(|(v,z)|^{\lfloor K(k)/2\rfloor})$. Each part is killed by the homotopy method. The details are exactly the same as in \cite{IY}. Again, this step is algorithmic. 
\end{proof}

\vskip20pt
\subsection{\bf Properties of compensators}

This section is devoted to properties of different fonctions useful for the expression of the results, and in particular the so-called compensators $\omega_\alpha(\xi)$ and $\Omega_{\alpha,\beta}(\xi)$ defined in \eqref{compensator} and \eqref{Omega}.

First, we  introduce  the analytic function
\begin{equation}\kappa(\eta)=
\begin{cases} \frac{e^\eta-1}{\eta},&\eta\not =0,\\
1, &\eta=0.\end{cases}\label{def:kappa}\end{equation}
 The following Lemma gives some useful properties of $\kappa$:

\begin{lemma}\label{lemkappa}

 The function $\kappa$  is an entire analytic real function whose series is given by
$\kappa(\eta)=\sum_0^{+\infty}\frac{\eta^n}{(n+1)!}$. It follows that $\frac{d\kappa}{d\eta}(\eta)<\kappa(\eta)<e^\eta$ for $\eta>0.$  Moreover,  $\kappa(\eta)>0,$  $\frac{d\kappa}{d\eta}(\eta)>0$, and $\frac{d^2\kappa}{d\eta^2}(\eta)>0,$ for all $\eta\in \R.$
\end{lemma}
\begin{proof}

We have that $\kappa(\eta)=\frac{1}{\eta}(\sum_0^{+\infty}\frac{\eta^n}{n!}-1)=\sum_0^{+\infty}\frac{\eta^n}{(n+1)!}$  and then:  $\frac{d\kappa}{d\eta}(\eta)=\sum_0^{+\infty}\frac{\eta^n}{n!(n+2)}.$
The inequalities $\frac{d\kappa}{d\eta}(\eta)<\kappa(\eta)<e^\eta$ for $\eta>0,$ follow trivially.

Clearly, $\kappa(\eta)\not =0$ for all $\eta\in \R\setminus \{0\}$ and as $\kappa(0)=1,$ it follows that $\kappa(\eta)>0$ for all $\eta\in \R.$

 Next,  as $\frac{d\kappa}{d\eta}(\eta)=\frac{\eta e^\eta-e^\eta+1}{\eta^2},$  any root $\eta\not =0$ of  $\frac{d\kappa}{d\eta}(\eta)=0$ verifies that $e^\eta=\frac{1}{1-\eta}.$  Comparing the series of these two functions, we see that $e^\eta<\frac{1}{1-\eta}$ for $\eta\in ]0,1[.$  The inequality $\frac{d\kappa}{d\eta}(\eta)>0$ is trivially verified when $\eta\geq 1.$ Finally, when $\eta<0,$ we put $\eta=-\delta,$ with $\delta\in \R^+$. The trivial inequality: $e^\delta>1+\delta,$ for $\delta\in \R^+$  implies that $e^\eta<\frac{1}{1-\eta}$ for $\eta<0.$ As $\frac{d\kappa}{d\eta}(0)=\frac{1}{2},$ we have that  $\frac{d\kappa}{d\eta}(\eta)>0$ for all $\eta\in \R.$

To finish,  since $\frac{d^2\kappa}{d\eta^2}(\eta)=\frac{(\eta^2-2\eta+2 )e^\eta-2}{\eta^3},$  any root $\eta\not =0$ of  $\frac{d^2\kappa}{d\eta^2}(\eta)=0$ verifies that $e^\eta=\frac{1}{1-\eta+\frac{1}{2}\eta^2}$, or equivalently $e^{-\eta}=1-\eta+\frac{1}{2}\eta^2$. Let $g(\eta)=e^{-\eta}-1+\eta-\frac{1}{2}\eta^2$. Let us show that $g(\eta)\neq0$ for $\eta\neq0$. Indeed,  $g'(\eta)= -e^{-\eta}+1-\eta<0$. The numerator of $\frac{d\kappa}{d\eta}(\eta)$ is $-e^{\eta}g'(\eta)$ and is positive for $\eta\neq0$.  Hence, $g'(\eta)<0$ for $\eta\neq0$, and since $g(0)=0$, then $\eta g(\eta)<0$ for $\eta\neq0$. As $\frac{d^2\kappa}{d\eta^2}(0)=\frac{1}{3},$ we have that  $\frac{d^2\kappa}{d\eta^2}(\eta)>0$ for all $\eta\in \R.$ \end{proof}
\vskip5pt

 The following lemma gives the relation of $\omega$ defined in \eqref{compensator} with $\kappa$, and interesting properties which can be easily deduced using this relation:
  
  \begin{lemma} \label{lemomega}
We have that 
 $\omega(\xi,\alpha)=-\kappa(-\alpha\ln \xi)\ln \xi.$ The compensator $\omega $ verifies the following estimates
\begin{enumerate}
\item

$ \omega(\xi,\alpha)\leq -\ln \xi$ if  $\alpha\leq 0$   and $\omega(\xi,\alpha)\leq  -\xi^{-\alpha}\ln \xi$  if  $\alpha\geq  0,$
and then
\begin{equation}\label{eq23}
 \omega(\xi,\alpha)=O(\xi^{-|\alpha|}|\ln \xi|).
\end{equation}
\item 
\begin{equation}\label{eq24}
\omega(\xi,\alpha)\rightarrow +\infty \   \  \mathrm{when}\  \  (\xi,\alpha)\rightarrow (0,0).
\end{equation}
\end{enumerate}
\end{lemma}
\begin{proof}

  Using  properties of $\kappa$ given in Lemma \ref{lemkappa}, it follows that:
 \begin{enumerate}
\item If  $\alpha\geq 0,$ i.e $-\alpha \ln \xi \geq 0,$ then $\omega(\xi,\alpha)=-\kappa(-\alpha \ln \xi) \ln \xi$ is less than $-e^{-\alpha \ln \xi} \ln \xi =-\xi^{-\alpha}\ln \xi.$ 
\item If  $\alpha\leq 0,$ i.e $-\alpha \ln \xi \leq  0,$  then $\omega(\xi,\alpha)=-\kappa(-\alpha \ln \xi) \ln \xi \leq -\ln \xi$ (indeed, $\kappa$ is  increasing, $\kappa(0)= 1,$ yielding $\kappa(\eta)\leq 1$ when $\eta\leq 0).$
\end{enumerate}
The estimate (\ref{eq23}) follows from  these two inequalities.  In order to prove  (\ref{eq24}), we take any $K>0.$
\begin{enumerate}
\item If  $-\alpha\ln \xi \geq -K,$ we have that $\kappa(-\alpha\ln \xi)\geq \kappa(-K),$ as $\kappa$ is increasing, and then $\omega(\xi,\alpha)\geq -\kappa(-K)\ln \xi.$ 
\item If  $-\alpha\ln \xi \leq -K$  (in particular $\alpha\leq 0$), we have that 
$$\omega(\xi,\alpha)=\frac{1-e^{-\alpha\ln \xi}}{|\alpha|}\geq \frac{1-e^{-K}}{|\alpha|},$$
\end{enumerate}
from which (\ref{eq24}) follows. \end{proof}
\vskip5pt

In parallel with the compensator $\Omega$ introduced in \eqref{Omega}, we introduce the symmetric function 
\begin{equation}{\cal K}(\eta,\delta)=\begin{cases} \frac{\kappa(\eta)-\kappa(\delta)}{\eta-\delta}, &\eta\not =\delta,\\\frac{d\kappa}{d\eta}(\eta),&\eta=\delta. \end{cases}\label{def:K}\end{equation}  This yields
$$\Omega(\xi,\alpha,\beta)={\cal K}(-\alpha\ln \xi,-\beta\ln \xi)\ln^2\xi.$$
The useful properties of $\Omega(\xi,\alpha,\beta)$ are given by the following lemma:
\vskip5pt

\begin{lemma}\label{lemOmega} 
 $\Omega_{\alpha,\beta}(\xi)=O(\xi^{-\gamma}\ln^2 \xi), $
where $\gamma=\mathrm{max}\{|\alpha|,|\beta|\}$, and   $\Omega_{\alpha,\beta}(\xi)\rightarrow +\infty, $
when $(\xi,\alpha,\beta)\rightarrow (0,0,0).$ 
\end{lemma} 
\begin{proof} To prove the two claims, we just have to use the Mean Value Theorem for the function ${\cal K}$:  there exists $\theta\in [\eta,\delta],$  such that ${\cal K}(\eta,\delta)=\frac{d\kappa}{d\eta}(\theta).$

Let us begin by the first claim.  Let us start with the case  $\alpha\geq \beta$.  Then ${\cal K}(-\alpha\ln \xi,-\beta\ln \xi)= \frac{d\kappa}{d\eta}(\theta),$
for some $\theta\in [-\beta\ln \xi,-\alpha\ln \xi].$ As $\frac{d\kappa}{d\eta}(\eta)$ is an increasing function (see Lemma \ref{lemkappa}), we have that 
${\cal K}(-\alpha\ln \xi,-\beta\ln \xi)\leq \frac{d\kappa}{d\eta}(-\alpha\ln \xi).$
If $\alpha\leq 0,$ we use that $\frac{d\kappa}{d\eta}(-\alpha\ln \xi)\leq \frac{d\kappa}{d\eta}(0)=\frac{1}{2}$ to obtain that $\Omega_{\alpha,\beta}(\xi)\leq \frac{1}{2}\ln^2\xi.$ If $\alpha\geq 0,$ again using Lemma \ref{lemkappa}, we have that  $\frac{d\kappa}{d\eta}(-\alpha\ln \xi)\leq e^{-\alpha\ln \xi}=\xi^{-\alpha}$,  and then that:  $\Omega_{\alpha,\beta}(\xi)\leq \xi^{-\alpha}\ln^2\xi.$ We can summarize the two possibilities by writing that $\Omega_{\alpha,\beta}(\xi)\leq\xi^{-|\alpha|}\ln^2\xi,$ as soon as $\alpha\geq \beta$  and $\xi$ and $|\alpha|$ sufficiently small. 
Using the symmetry of $\Omega_{\alpha,\beta}(\xi)$ we can permute $\alpha $ and $\beta$ in the above argument to obtain finally  that  $\Omega_{\alpha,\beta}(\xi)=O(\xi^{-\gamma}\ln^2 \xi), $
where $\gamma=\mathrm{max}\{|\alpha|,|\beta|\}.$
\vskip5pt
We now prove  the second claim.  By symmetry on $\alpha$ and $\beta$ it suffices to prove the claim for $\alpha\geq \beta.$ As above, we can write that 
$\Omega_{\alpha,\beta}(\xi)=\frac{d\kappa}{d\eta}(\theta)\ln^2\xi,$ for some $\theta\in  [-\beta\ln \xi,-\alpha\ln \xi].$ Now, we want to bound $\Omega_{\alpha,\beta}$ from below. 
Since $\frac{d\kappa}{d\eta}$ is increasing, $\Omega_{\alpha,\beta}(\xi)\geq \frac{d\kappa}{d\eta}(-\beta\ln\xi)\ln^2\xi.$  If $\beta\geq 0,$ we just use that  $\frac{d\kappa}{d\eta}(-\beta\ln\xi)\geq \frac{d\kappa}{d\eta}(0)=\frac{1}{2},$ to obtain that  $\Omega_{\alpha,\beta}(\xi)\geq \frac{1}{2}\ln^2\xi.$ If $\beta\leq 0,$ we have to compute $\frac{d\kappa}{d\eta}(-\beta\ln\xi)=\frac{d\kappa}{d\eta}(|\beta|\ln\xi)=\frac{d\kappa}{d\eta}(\ln\xi^{|\beta|}).$  As $\frac{d\kappa}{d\eta}(\eta)=\frac{(\eta-1) e^\eta+1}{\eta^2},$ we have that 
$\frac{d\kappa}{d\eta}(-\beta\ln\xi)
=\frac{(|\beta|\ln \xi-1)\xi^{|\beta|} +1}{|\beta|^2\ln^2\xi} $ and then: $\Omega_{\alpha,\beta}(\xi)\geq \frac{(|\beta|\ln \xi-1)\xi^{|\beta|} +1}{|\beta|^2},$ yielding that  $\Omega_{\alpha,\beta}(\xi)\rightarrow +\infty.$
This yields the conclusion. \end{proof}

\subsection{\bf Transition along the trajectories }

\vskip5pt
{\it We want to study transition maps for (\ref{eq1}), in the region $Q=\{u\geq 0,v\geq 0,\}\subset \R^{3}$ near the origin.} More precisely, let $W$ be a neighborhood  of the origin in $\R^3$,  and $\Pi\subset \{u=u_0\},$ for $u_0>0 $, be a section.   The neighborhood $W$ can be chosen sufficiently small so that the trajectory starting at any point in $W\cap \{u>0\}$ reaches $\Pi$ for a finite positive time (in particular, $W\cap \Pi=\emptyset).$  We consider the transition $T_{\mu,\sigma}$ from the points in $W\cap \{u>0\}$ to the section $\Pi.$ 

We will compute $T_{\mu,\sigma},$ in  the ${\cal C}^k$-coordinates given by  Theorem \ref{thnormalformhyp}.  In this system of coordinates the  family is the smooth family of polynomial vector fields $X^N_{\mu,\sigma}$ (this means polynomial in $(u,v,y)$ with smooth coefficients in $(\mu,\sigma)).$

   We take $\Pi=[-Y_0,Y_0]\times [0,v_0]\times \{u_0\}$  for some $Y_0>0,v_0>0.$  On $\Pi$,  we replace the  coordinate $v$  by $\nu=u_0v,$ with $\nu\in [0,\nu_0=u_0v_0].$ Then, we can write  $T_{\mu,\sigma}(u,v,Y)=(\widetilde Y_{\mu,\sigma}(u,v,Y),\nu=uv).$

           The expression of the $Y$-component $\widetilde Y_{\mu,\sigma}$ is given by the following Theorem: 

\begin{theorem}\label{thtransgeneralhypsaddle}

Let $\bar\sigma=\bar \sigma(\sigma,\nu)=\sigma+\varphi_{\mu,\sigma}(\nu)$ and $\alpha=\alpha(\sigma,\nu)=\bar\sigma(\sigma,\nu)-\sigma_0$, where  $\varphi_{\mu,\sigma}$  is the polynomial family  introduced in Theorem \ref{thnormalformhyp}. 
  The $Y$-component of the  transition map $T_{\mu,\sigma}$ has the following expression on $W\cap\{u>0\}$:
\begin{enumerate}
\item If $\sigma_0\not\in \Q: $
\begin{equation}\label{eq4}
\widetilde Y_{\mu,\sigma}(u,v,Y)=\Big(\frac{u}{u_0}\Big)^{\bar\sigma} Y.
\end{equation}
\item    If $\sigma_0=\frac{p}{q}\in \Q$ with $(p,q)=1,$  when $\sigma_0\not \in \N:$

\begin{equation}\label{eq5}
\widetilde Y_{\mu,\sigma}(u,v,Y)=\eta_{\mu,\sigma}(\nu)v^p\Big(\frac{u}{u_0}\Big)^{\bar\sigma}\omega\Big(\frac{u}{u_0},\alpha\Big)+\Big(\frac{u}{u_0}\Big)^{\bar\sigma}\Big(Y+\phi_{\mu,\sigma}(Y,u,v)\Big),
\end{equation}
where $\eta_{\mu,\sigma}$  is the same as  in (\ref{eq3}) (in particular, $\eta_{\mu,\sigma}\equiv 0$ when $\sigma_0\not\in \N).$
\end{enumerate}

The function family $\phi_{\mu,\sigma}$ in   (\ref{eq5}) is of order $O\left(u^{p+q\alpha}\omega^{q+1}\Big(\frac{u}{u_0},\alpha\Big)|\ln u|\right)$ and, for any integer $l\geq 2,$   is of class ${\cal C}^{l-2}$   in $(Y,u^{1/l},u^{1/l}\omega\Big(\frac{u}{u_0},\alpha\Big),v,\mu,\sigma)$. 
\end{theorem}
\vskip5pt\noindent{\it Proof.}
The time to go from a point $(u,v,Y)\in W\cap\{u>0\}$ to the section $\Pi$ along  the flow of $X^N_{\mu,\sigma}$ is equal to $-\ln\frac{u}{u_0}.$  Expression (\ref{eq4}) follows trivially from the integration of the third line of the system (\ref{eq2}).

    Then, from now on, we will assume that $\sigma_0\in \Q$  and we will study the integration of the system (\ref{eq3}). The trajectory through the point $(u,v,Y)$ is equal to $(ue^t,ve^{-t},Y(t))$ where $Y(t)$ is solution of the $1$-dimensional non-autonomous differential equation:
\begin{equation}\label{eq6}
 {\dot Y}(t) =-\bar \sigma Y(t)+\Phi_{\mu,\sigma}(\nu,u^pe^{pt}Y(t)^q)Y(t)+e^{-pt}v^p\eta_{\mu,\sigma}(\nu),
\end{equation}
with initial condition $Y(0)=Y.$
\vskip5pt

 In order to eliminate the linear term in (\ref{eq6}) we look for $Y(t)$ in the form $Y(t)=e^{-\bar\sigma t}Z(t).$ As $\dot Y(t)=e^{-\bar\sigma t}\dot Z(t)-\bar\sigma Y(t),$ and letting $\bar \sigma=\frac{p}{q}+\alpha$, we obtain the following  differential equation for $Z(t):$
\begin{equation}\label{eq7}
\dot Z=\Phi_{\mu,\sigma}(\nu,e^{-q\alpha t}u^pZ^q)Z+ e^{\alpha t}v^p\eta_{\mu,\sigma}(\nu),
\end{equation}
with initial condition $Z(0)=Y.$ Note that the term in $\eta_{\mu,\sigma}$ is only present when $q=1.$ 

The $1$-dimensional non-autonomous differential equation (\ref{eq7})  is smooth in $(t,Z,\sigma,\nu,u,v,\mu)$ and can be integrated for any time $t\in [0,-\ln\frac{u}{u_0}].$  If $Z(t)$ is the solution of (\ref{eq7}) with initial condition $Z(0)=Y,$ we will have that 
\begin{equation}\label{eq8}
\widetilde Y_{\mu,\sigma}(u,v,Y)=\Big(\frac{u}{u_0}\Big)^{\bar\sigma}Z\Big(-\ln\frac{u}{u_0}\Big).
\end{equation}

The above expression has to be studied for $u>0$ (we extend $\widetilde Y$ along $\{u=0\}$ by $\widetilde Y_{\mu,\sigma}(0,v,Y)=0).$
We  first  study the integration of (\ref{eq7}). 

To begin, it is easy to get rid of the term $e^{\alpha t}v^p\eta_{\mu,\sigma}(\nu)$ in (\ref{eq7}). Let us consider  the analytic function

$$
  \Theta(t,\alpha)=\begin{cases}\frac{e^{\alpha t}-1}{\alpha}, &  \alpha\not =0,\\
  t,&\alpha=0.\end{cases}
 $$
  which verifies $\dot \Theta=e^{\alpha t}.$  We have that  $\Theta(t,\alpha)=t\kappa(\alpha t)$  and then
 $ \omega(\xi,\alpha)=\Theta(-\ln \xi,\alpha).$
  
  Putting $Z(t)=v^p\eta_{\mu,\sigma}(\nu)\Theta(t,\alpha)+\bar Z(t),$ we see that $\bar Z(t)$  is the solution of the differential equation
\begin{equation}\label{eq9}
\dot{\bar Z}=\Phi_{\mu,\sigma}\Big(\nu,u^pe^{-q\alpha t}(v^p\eta_{\mu,\sigma}(\nu)\Theta(t,\alpha)+\bar Z)^q\Big)(v^p\eta_{\mu,\sigma}(\nu)\Theta(t,\alpha)+\bar Z),
\end{equation}
with initial condition $\bar Z(0)=Y.$  
\vskip5pt
 
As $\Phi_{\mu,\sigma}(\nu,0)\equiv 0,$ we can write 
$\Phi_{\mu,\sigma}(\nu,\xi)=\xi H_{\mu,\sigma}(\nu,\xi),$
where $H_{\mu,\sigma}$ is a smooth function. Now, let us notice that $e^{\alpha t}=\dot  \Theta= 1+\alpha\Theta.$ Moreover the map $t\rightarrow \Theta(t,\alpha)$ is invertible (for any $\alpha$).  Then, we can change the time $t$ by the time $\Theta$ in the differential equation (\ref{eq9}). We obtain the new equation
\begin{equation}\label{eq10}
\frac{d\bar Z}{d\Theta}=u^p\bar H(\Theta,\bar Z,u,v,\nu,\alpha,\mu,\sigma)
\end{equation}
with 
\begin{equation}\label{eq11}
\bar H=(1+\alpha \Theta)^{-(1+q)}(v^p\eta\Theta+\bar Z)^{q+1}H_{\mu,\sigma}\Big(\nu,u^p(1+\alpha \Theta)^{-q}(v^p\eta\Theta+\bar Z)^q\Big),
\end{equation}
where $\eta=\eta_{\mu,\sigma}(\nu).$ 
Let  $\Psi\left(\Theta,Y,u,v,\nu,\alpha,\mu,\sigma\right)$ be the solution of (\ref{eq10}), with the ``time'' $\Theta.$  {\it Up to now, $\Theta$ is seen as an independent variable;  in particular it is independent from $\alpha$}. For $t=-\ln\frac{u}{u_0},$  then $\Theta=\omega_\alpha(\frac{u}{u_0}),$ yielding
  
 \begin{equation}\label{eq12}
 Z\Big(-\ln\frac{u}{u_0}\Big)=\Psi\left(\omega\Big(\frac{u}{u_0},\alpha\Big),Y,u,v,\nu,\alpha,\mu,\sigma\right)+v^p\eta_{\mu,\sigma}
(\nu)\omega\Big(\frac{u}{u_0},\alpha\Big),
 \end{equation}
 and then, the computation of $\widetilde Y_{\mu,\sigma}(u,v,Y)$ reduces to the computation of  $\Psi\left(\omega\Big(\frac{u}{u_0},\alpha\Big),Y,u,v,\mu,\sigma\right).$

One  difficulty in the study of $\Psi\left(\omega\Big(\frac{u}{u_0},\alpha\Big),Y,u,v,\nu,\alpha,\mu,\sigma\right)$ is that $\omega\Big(\frac{u}{u_0},\alpha\Big)\rightarrow+\infty$ if $u\rightarrow 0.$ To overcome this difficulty  we will exploit the fact that the right hand side of (\ref{eq10})  is divisible   by $u^p.$

  We first study the differential equation (\ref{eq10}). We put  $u=U^l$ and change the time $\Theta$ by the time $\tau=U\Theta$  (and not just by $u\Theta,$ as it could seem more natural).  The equation  (\ref{eq10}) is replaced by the following equation

\begin{equation}\label{eq13}
\frac{d\bar Z}{d\tau}=U^{pl-1}\bar H\Big(\frac{\tau}{U},\bar Z,U^p,v,\nu,\alpha,\mu,\sigma\Big),
\end{equation}
where $\bar H$ is given by (\ref{eq11}). Let $\bar G$ be the right hand side of (\ref{eq13}). It is smooth for $U>0,$ but since it is function of $\alpha\frac{\tau}{U},$ it is not well-defined in a whole neighborhood 
of the point $\{(\tau,\bar Z,U,v,\nu,\alpha,\mu,\sigma)=(0,0,0,0,0,0,\mu_0,\sigma_0)\}.$  Fortunately,  we only need to integrate (\ref{eq13}) in a closed domain $\ov{\cal D}$:
\smallskip

\noindent{\bf Definition of  $\ov{\cal D}$.} The domain  $\ov{\cal D}$ is defined in the space $(\tau,U,\bar Z,v,\nu,\alpha,\mu,\sigma)$  defined by
\begin{enumerate}
\item $U\in [0,U_1],$ $|\alpha|\leq \alpha_0$ and $\tau\in [0,U\omega (\frac{U^l}{u_0},\alpha)]$, where $U_1,\alpha_0>0$ are  chosen arbitrarily small (the time  $\tau=U\omega( \frac{U^l}{u_0},\alpha)$ corresponds to the time $t=-\ln \frac{u}{u_0}=-l\ln \frac{U}{u_0})$,
\item $(\bar Z,v,\nu,\alpha,\mu,\sigma)\in \A,$ an arbitrarily small closed neighborhood of the value  $(0,0,0,0,\mu_0,\sigma_0).$
\end{enumerate}
We want to prove that $\bar G$ is  of class ${\cal C}^{l-2}$ on $\ov{\cal{D}}.$ We will first prove a technical lemma about the partial derivatives of the function $\bar G.$ 
Let us denote by $\partial_m \bar G$ any  partial derivative of $\bar G$ corresponding to a multi-index $m=(m_1,\ldots,m_s)$ associated to  the variables $\tau,U,\bar Z,v,\nu,\alpha,\mu,\sigma$ and the coordinates of $\mu.$ Let $|m|=m_1+\cdots +m_s$ be the degree of $m.$ We will note by $\delta,$ a strictly positive number, which can be made arbitrarily small by appropriately choosing $U_1$ and $\A.$  We have the following:

\begin{lemma}\label{lembarD}
Let be $\sigma_0=\frac{p}{q}$ as  above. 
Let $m$ be any multi-index such that $|m|\leq l-2.$ Then, for any $\delta>0$, there exists a domain $\ov{\cal{D}}$ as above, such that \emph{on the restriction to  the domain $\ov{\cal{D}}$} we have that
\begin{equation}\label{eq13_bis}
\partial_m \bar G =O(U^{pl-|m|-1-\delta}).
\end{equation}
\end{lemma}
\begin{proof}
Recall that $\bar G=U^{pl}\bar H,$ where $\bar H$ is given by (\ref{eq11}) and $\Theta$ is replaced by $\frac{\tau}{U}.$ 
The proof is straightforward, but rather tedious, and we just give the main steps. First, let us notice that on $\ov{\cal{D}}$ we have that, for any $s\in \Z$: 
\begin{equation}\label{eq16}
\Big(1+\alpha \frac{\tau}{U}\Big)^s=(1+\alpha \Theta)^s=e^{s\alpha t}= O(U^{-|sl\alpha|}).
\end{equation}
Also, using Lemma \ref{lemomega}, we have that:
$$ \frac{\tau}{U}=\Theta=\kappa(\alpha t)t \leq 
e^{|\alpha|t}t\leq  l U^{-|l\alpha|}|\ln U|.$$
 These estimations imply that $\Big(1+\alpha \frac{\tau}{U}\Big)^{-(q+1)}$  and $\frac{\tau}{U}$ have an order $O(U^{-\delta}).$ As  $\bar H$ is bounded on  $\bar {\cal D},$  we have that
$\bar G=O(U^{pl-1-\delta}).$ This  is the expected result for $m=0.$

Next, we use the expression of the partial derivatives of $\bar G,$ in terms of  the functions  $\Theta,$  $(1+\alpha\Theta)^{-q}$ or  $(1+\alpha\Theta)^{-(q+1)}$ and the partial derivatives of $H_{\mu,\sigma},$  evaluated on $\ov{\cal D}$ (these partial derivatives are bounded on  $\ov{\cal D}).$ We have for instance that:
$$\frac{\partial }{\partial U}(1+\alpha\Theta)^{-q}=-ql\alpha (1+\alpha\Theta)^{-(q)}\frac{1}{U}=O(U^{-1-\delta}).$$

 As $(1+\alpha\Theta)^{-q}=O(U^{-\delta}),$ we remark that the order in $U$ has  discreased by one unit (modulo an order in $\delta).$

{\it It is easy to see that this observation can be generalized for any   partial derivative: the previous order in $U$ decreases by one unity for each first order partial derivation (modulo an order in $\delta).$}

Then, starting with  $\bar G=O(U^{pl-1-\delta})$ for $m=0,$ the estimation (\ref{eq13}) for any multi-index $m$ follows directly by recurence from this fall of order (let us notice that, in a symbolic way, we have: $``\delta+\delta=\delta").$ \end{proof}

\vskip5pt\noindent{\bf End of the proof of Theorem \ref{thtransgeneralhypsaddle} }

\vskip5pt Lemma \ref{lembarD} says that each partial derivative $\partial_m\bar G$ can be extended continuously on $\tau=U=0$ by giving it the value zero at these points.  Then, 
 as the function $\bar G$ is smooth on $\ov{\cal D}\setminus \{\tau=U=0\},$  the restriction of $\bar G$  to $\ov{\cal{D}}$ is a function of differentiability class ${\cal C}^{l-2},$ on the whole domain $\ov{\cal{D}},$  {\it including  the points on $\{\tau=U=0\},$}  when we give to each partial derivative of  $\bar G$ or order less than $l-2$ the value $0$ at these points.  Let  $\B$ be a closed neighborhood of $(0,0,0)$ in the $(\tau,\alpha,U)$-plane, containing the closed  set $$\{(\tau,\alpha,U)\  | \   \tau\in [0,-lU\ln\frac{U}{U_0}], \ |\alpha |\leq \alpha_0,\   U\in [0,U_1]\}$$ 
 that we have introduced  above in the definition of $\ov{\cal{D}}$. The closed domain   $\ov{\cal{D}}$ is contained in the neighborhood $\A\times \B.$  Using the Whitney Theorem for the extention of differentiable functions (see [M] for instance),  we can find a ${\cal C}^{l-2}$-function $\widetilde G$ on a $\A\times \B$ such that $\widetilde G|_{\ov{\cal{D}}}\equiv \bar G$ (here, this extention can also be easily constructed by hand, in an elementary way).

For  times $\tau \in [0,-lU\ln\frac{U}{U_0}]$   the flow $\Psi(\tau,\bar Z,U,v,\nu,\alpha,\mu,\sigma)$ of the differential equation (\ref{eq10}): $\frac{d\bar Z}{d\tau}=\bar G$ coincides with the flow $\widetilde \Psi(\tau,\bar Z,U,v,\nu,\alpha,\mu,\sigma)$ of the differential equation 
 $\frac{d\bar Z}{d\tau}=\widetilde G.$   This equation  is of differentiability class  ${\cal C}^{l-2}$  on $\A\times \B,$ as well as its flow $\widetilde \Psi.$

  In particular, we have that
 \begin{equation}\label{eq17}
 \bar Z\Big(-\ln\frac{u}{u_0}\Big)=\widetilde \Psi\left(U\omega\Big(\frac{U^l}{u_0},\alpha\Big),Y,U,v,\nu,\alpha,\mu,\sigma\right),
 \end{equation}
  is a ${\cal C}^{l-2}$-function of $(Y,U,U\omega\Big(\frac{U^l}{u_0},\alpha\Big),v,\nu,\alpha,\mu,\sigma)$, i.e. is a  ${\cal C}^{l-2}$-function in the variables $\left(Y,u^\frac{1}{l}, u^\frac{1}{l}\omega(\frac{u}{u_0},\alpha),v,\nu,\alpha,\mu,\sigma\right)$,  {\it  a function which is defined on a neighborhood of the point $(0,0,0,0,0,\mu_0,\sigma_0).$} We can replace $\alpha$ (outside $\omega$) by its expression in $(\sigma,\nu)$ and $\nu$ by $uv$  to obtain finally that  $ \bar Z\Big(-\ln\frac{u}{u_0}\Big)$ is a  ${\cal C}^{l-2}$-function of $\left(Y,u^\frac{1}{l},u^\frac{1}{l}\omega(\frac{u}{u_0},\alpha),v,\mu,\sigma\right)$.  As $\bar Z (0)=Y,$ we can write 
\begin{equation}\label{eq18}
 \bar Z\Big(-\ln\frac{u}{u_0}\Big)=Y+\phi_{\mu,\sigma}(Y,u,v),
 \end{equation} 
where
\begin{equation}\label{eq19} 
\phi_{\mu,\sigma}=\widetilde \Psi\left(U\omega\Big(\frac{U^l}{u_0},\alpha\Big),Y,U,v,\nu,\alpha,\mu,\sigma\right)-Y
\end{equation} 
 is a  ${\cal C}^{l-2}$-function of $\left(Y,u^\frac{1}{l},u^\frac{1}{l}\omega(\frac{u}{u_0},\alpha),v,\mu,\sigma\right).$ Finally, collecting the different terms in (\ref{eq8}),  (\ref{eq12}), (\ref{eq18}) and  (\ref{eq19}),
  we obtain the expression  (\ref{eq5})  in Theorem \ref{thtransgeneralhypsaddle}, for the transition function $\widetilde Y_{\mu,\sigma}(u,v,Y).$

We can estimate $\phi_{\mu,\sigma}$ from the differential equation (\ref{eq9}) for $\bar Z(t).$
 If $G(t,\bar Z, u,v,\nu,\alpha,\sigma,\mu)$ is the right hand side of    (\ref{eq9}), we have that 
$G=O(u^pe^{-q\alpha t}\Theta^{q+1})$
 on the domain $\ov{\cal D}$ defined above.  As $t\leq -\ln\frac{u}{u_0} $ on  $\ov{\cal D},$ then $\Theta(t,\alpha)\leq \omega\Big(\frac{u}{u_0},\alpha\Big)$, yielding $G=O(u^{p+q\alpha}\omega^{q+1}(\frac{u}{u_0},\alpha)).$
From this estimate of the order of $G$, it follows that
$$\phi_{\mu,\sigma}=\bar Z\Big(-\ln \frac{u}{u_0}\Big)-Y=O(u^{p+q\alpha}\omega^{q+1}\Big(\frac{u}{u_0},\alpha\Big)|\ln u|),$$
 which is the estimation in the statement of Theorem  \ref{thtransgeneralhypsaddle}. \hfill $\Box$
\vskip10pt

\subsection{\bf Transitions between sections}
\vskip5pt
Theorem  \ref{thtransgeneralhypsaddle} gives the expression of the transition $T_{\mu,\sigma}=(\nu,\widetilde Y_{\mu,\sigma}),$ starting from any point $(u,v,Y)$ in the domain $W\cap\{u>0\}$ and landing on a section $\Pi\subset \{u=u_0\},$  for some $u_0>0$ (we can extend trivially $T_{\mu,\sigma}$ to the whole neighborhood $W$ by taking $\widetilde Y_{\mu,\sigma}(u,v,0)=0).$ We apply this to get Theorems~\ref{thtranstypeI} and ~\ref{thtranstypeII}
after changing $(u,v)\mapsto(r,\rho)$. 

\vskip5pt
\noindent{\bf Discussion of Theorems \ref{thtranstypeI} and  \ref{thtranstypeII}.}
A previous version of Theorems \ref{thtranstypeI} and  \ref{thtranstypeII} was given in Theorems 4.10 and 4.14 of  \cite{ZR}. It is interesting to compare their proofs and formulations with the proofs and formulations in the present paper.
\begin{enumerate}
\item The proof in the present version is unified: Theorem  \ref{thtransgeneralhypsaddle} gives a formula for a global transition from any point in a $3$-dimensional neighborhood $W,$  formula which is easy to restrict on the two different types of section $\Sigma.$  Next, the proof of Theorem \ref{thtransgeneralhypsaddle}, even if it is based on the same normal form, is much  shorter than the proofs of Theorems 4.10 and 4.14 given in  \cite{ZR}. The reason seems to be that in  \cite{ZR} the transition function $\widetilde Y$ and its partial derivatives are directly estimated by a   variational method.  In the present paper, we have replaced the $1$-dimensional non-autonomous  differential system:  $\dot{\bar Z}=\bar G,$  which is not defined in a neighborhood of the point $\{(\tau,\bar Z,U, v, \nu,\alpha,\mu,\sigma)=(0,0,0,0,0,0,\mu_0,\sigma_0)\},$ by a differential equation:  $\dot{\bar Z}=\widetilde G,$  differentiable on a neighborhood of this point.  As a consequence,  we obtain almost  without computation that   the function $\phi_{\mu,\sigma}$  is differentiable (in terms of fractional power and a compensator  of some variable). In fact, the heavy computations made in  \cite{ZR} are replaced by an  implicit use of the Cauchy Theorem for differential equations.
\item We can compare the statements in  \cite{ZR} and in the present paper. We restrict the comparison to the only non-trivial case: $\sigma_0\in \Q$. The transition function called here $\widetilde Y_{\mu,\sigma}$ is given by the formula (4.11) of  Theorem 4.10 of  \cite{ZR}.  We can observe that it is quite similar to the above formula (\ref{eq16}),  up to the changes of notations.     The same remarks are valid for the transition of type II which is treated in Theorem 4.14 in  \cite{ZR}. The only important difference is in the form and properties of the function $\phi_{\mu,\sigma},$ which is called $\phi$ or $\theta$  in  \cite{ZR}. We will comment on this in the next items.
\item  The function  $\phi_{\mu,\sigma}$ in   Theorem \ref{thtranstypeI} is of order $O(\nu^{p+q\alpha}\omega^{q+1}\Big(\frac{\nu}{\nu_0},\alpha\Big)|\ln \nu|).$ This order has to be compared with the order given for the function $\phi$ in Theorem 4.10 of  \cite{ZR} which is  exactly the same order for $\alpha<0$, but equal to   $O(\nu^{p}\omega^{q+1}\Big(\frac{\nu}{\nu_0},\alpha\Big)|\ln \nu|)$ for $\alpha>0.$ This minor difference is probably due to the difference in the method of proof. It is less easy to compare the order  of $\phi_{\mu,\sigma}$ in   Theorem \ref{thtranstypeII} with the order of $\theta$ in  Theorem 4.14 of  \cite{ZR}.

\item In Theorem 4.10 of  \cite{ZR},   $\phi$  is a ${\cal C}^\infty$-function of $\omega\Big(\frac{\nu}{\nu_0},\alpha\Big)$ and other variables. Since $\omega\rightarrow +\infty$ for $\nu\rightarrow 0,$ this means that  the  domain of $\phi$ has to be unbounded. This implies that it is not  possible to deduce directly the order of the partial derivatives of $\phi.$ This order is obtained by using variational methods and heavy computations.  On the contrary, the formulation given in  Theorems \ref{thtranstypeI} and \ref{thtranstypeII}, permits a direct deduction of the order of any partial derivative of $\phi_{\mu,\sigma}.$ Let us show this on an example for a transition map of type I. Considering any $l\in \N$  and observing that  $\phi_{\mu,\sigma}$ is of order $O(\nu^{p-\delta}),$ we can write
$$\phi_{\mu,\sigma}=\nu^{p-\frac{1}{l}}\bar\phi_{\mu,\sigma},$$
where $\bar\phi_{\mu,\sigma}$ is a ${\cal C}^{l-p-3}$-function in  
$(Y,\nu^{1/l},\nu^{1/l}\omega\Big(\frac{\nu}{\nu_0},\alpha\Big),\mu,\sigma)$.
 
 As a consequence any partial derivative of $\phi_{\mu,\sigma}$ in terms of $Y,\mu,\sigma,$ of degree less than $l-p-3,$ is of order $O(\nu^{p-\frac{1}{l}}).$ Taking into account that we can take $l$ arbirarily large, this order in very similar to the order obtained in Theorem 4.10 of  \cite{ZR}.
\end{enumerate}
\section{\bf Appendix II---Counting the number of roots}

\vskip10pt
\subsection{\bf Differentiable functions on monomials}
We come back to the notations of Section~\ref{sect:quadratic}: $r,\rho$ are variables defined in a compact neighborhood $\A$ of $(0,0)$ in the first quadrant $Q=\{r\geq 0, \rho\geq 0\}.$  {\it We will always choose $\A$ to be a rectangle $[0,r_1]\times [0,\rho_1],$ in order to have connected curves $l_\nu=\{(r,\rho)\in \A\ | \  r\rho=0\}.$} In the following definitions we will use also  compensators  $\omega_\gamma$ and $\Omega_{\gamma,\delta},$ depending on other 
parameters  $\gamma,\delta.$ We will often use the shortened notation $\omega_\gamma,\Omega_{\gamma,\delta}$ for  $\omega_\gamma\Big(\frac{r}{r_0}\Big),\Omega_{\gamma,\delta}\Big(\frac{r}{r_0}\Big).$ Moreover, changing $r$ to $\frac{r}{r_0}$, we can of course suppose that $r_0=1$.

We consider a multi-parameter $\lambda$ in a  compact neighborhood $\B$ of  a value
 $\lambda_0$  in  some euclidean space  ${\cal E}.$  The neighborhood $\B$ will be chosen sufficiently small to have the desired properties.

We also consider functions which are  differentiable on real powers  of $r,\rho$  and compensators in $r.$
We give a precise definition of this notion.

\begin{definition}\begin{enumerate}
\item
 A  \emph{primary monomial (monomial in short),} is an expression $M=r^a,\ \rho^b,$ $\  r^a\omega_\gamma(r)^{c},\  r^a\Omega_{\gamma_1,\gamma_2}(r)^d$  or $\omega_\gamma(r)^{-e}$  where $a,b,c,d,e$  and $\gamma,\gamma_1,\gamma_2$ are smooth functions of $\lambda.$ Moreover  $a,b,e$ are strictly positive and  $\gamma(\lambda_0)=\gamma_1(\lambda_0)=\gamma_2(\lambda_0)=0$ (we can have $\gamma=\alpha$ or $\beta$ and
 $(\gamma_1,\gamma_2)=(\alpha,\beta)).$ 
For instance,   $r^\frac{2}{3},\ \rho^\frac{1}{5}, \omega_\alpha^{-1},  r\Omega_{\alpha,\beta}$ are primary monomials but not $r^\alpha$ or $\omega_\alpha^\alpha.$
 
 A monomial $M$ defines a $\lambda$-family of functions $M(r,\rho,\lambda)$ on 
$Q=\{r\geq 0,\ \rho\geq 0\},$  $M$ is smooth for $r>0$ and, by Lemmas \ref{lemomega} and \ref{lemOmega}, it can be extended continuously  along $\{r=0\});$  we have that  $M(0,0,\lambda_0)=0$ (i.e. $M=o(1),$ in terms of some distance of $(r,\rho,\lambda)$ to $(0,0,\lambda_0)).$

\item
 We say that a function $f(r,\rho,\lambda)$  on $\A\times\B$  is   a \emph{${\cal C}^k$-function on the monomials $M_1,\ldots ,M_l$} if there exists a  ${\cal C}^k$-function $\tilde f(\xi_1,\ldots,\xi_l,\lambda)$ defined on $\widetilde \A\times \B,$ where $\widetilde \A$ is  a neighborhood of $0\in \R^l$  such that 
$ f(r,\rho,\lambda)=\tilde f(M_1,\ldots ,M_l,\lambda).$
If the number of monomials and their type is not specified, we just say that $f$   is   a {\rm ${\cal C}^k$-function on monomials.} \end{enumerate}
\end{definition}

Clearly, the space of  {\rm ${\cal C}^k$-functions on monomials,} defined on $\A\times\B$ is a ring.  The classical  theorems of  differential calculus (Taylor formula, division theorem and so on) can be extended to these functions  by applying them to the  function $\tilde f.$  Since the differentiability class $k$ is  finite,  there will be falls of differentiability class in these operations: Lemma \ref{lemderivfunct} is one example. For this reason, we will consider functions $f$  with the property to be  {\it ${\cal C}^k$-functions on monomials, for any $k\in \N$} (but with 
a choice of monomials and a size of the neighborhood $\A\times\B$ that may  depend on $k$).  The functions $\psi_{\mu,\sigma}(Y,u,v),\psi_{\mu,\sigma}(Y,\nu)$ and $\psi_{\mu,\sigma}(u,v)$ introduced  in the statements of Theorems \ref{thtransgeneralhypsaddle} are,  \ref{thtranstypeI}, and \ref{thtranstypeII} are examples of   ${\cal C}^k$-functions on monomials for any $k,$ which  use only  the single  compensator $\omega_\alpha.$ The functions $h_i$ entering in the expression of the displacement map $V$ in Section 3 are using other compensators $\omega_\gamma$, and also $\Omega_{\alpha,\beta}.$
\vskip10pt

\subsection{\bf Procedure of division-derivation for functions with $2$ variables}
\vskip5pt

\begin{notation} In this section, $h(r,\rho,\lambda)=o(1)$ will mean that $h(0,0,\lambda_0)=0.$\end{notation}

We want to bound the  number of roots of an equation $\{V(r,\rho,\lambda)=0\}$   along the 
curves $l_\nu=\{r\rho=\nu\ |\  (r,\rho)\in \A\},$ for $\nu>0$ and a neighborhood $\A\times\B$ sufficiently small. The function $V$ is expressed using ${\cal C}^k$-functions on monomials. To obtain this bound, we will apply Rolle's Theorem, and to this end we will use recurrently the Lie-derivative $L_{\cal X}$ of $V$  by the vector field 
\begin{equation} {\cal X}=r\frac{\partial}{\partial r}-\rho\frac{\partial}{\partial \rho}.\label{vf}\end{equation} Hence, we need some properties of $L_{\cal X}$ acting on ${\cal C}^k$-functions on monomials. It is easy to see that:
\begin{equation}\begin{cases}L_{\cal X}r^a=ar^a,\\ 
L_{\cal X}\rho^b=-b\rho^b,\\  
L_{\cal X}\omega_\gamma=-(1+\gamma \omega_\gamma),\\ L_{\cal X}\Omega_{\gamma_1,\gamma_2}=-(\omega_{\gamma_1}+\gamma_2 \Omega_{\gamma_1,\gamma_2}).\end{cases}\end{equation}
From this, it follows  that
\vskip5pt
\begin{lemma}\label{lemderivfunct}
If  $f$   is   a  ${\cal C}^k$-function on monomials, then $L_{\cal X}f$ is a  ${\cal C}^{k-1}$-function on monomials and  $L_{\cal X}f=o(1).$
\end{lemma}
\begin{proof} If  $M$  is any monomial,  $L_{\cal X}M$ is a linear combinaison of monomials. Then,    $L_{\cal X}f=\sum_i\frac{\partial \tilde f}{\partial \xi_i}L_{\cal X}M_i,$   is a ${\cal C}^{k-1}$-function on monomials and, since each monomial is  $o(1),$   this function  $L_{\cal X}f$ is also $o(1)$. \end{proof}
\vskip5pt
For the procedure of division-derivation we will need more general monomials than the admissible ones:
\vskip5pt

\begin{definition}\begin{enumerate}
\item
A \emph{general monomial} is an expression $M=r^a\rho^b\prod_i \omega_i^{c_i}\prod_j\Omega_j^{d_j}$ where $i$ and $j$ belong to  finite sets of indices.  The coefficients  $a,b,c_i,d_j,$  as well as the internal parameters of the compensators $\omega_i,\Omega_j,$ are  smooth functions of $\lambda$ (without any restriction on sign).  Let $a(\lambda_0)=a^0,b(\lambda_0)=b^0.$ 
\item A general  monomial is \emph{resonant} if $a^0=b^0$ (in this case the \lq\lq polynomial\rq\rq\ part $r^{a^0}\rho^{b^0}$ of $M$ reduces to the first integral $\nu^{a^0}$).
Seen as a function of $(r,\rho,\lambda),$ such a monomial is in general not defined   for $r=0$ and $\rho=0.$  \end{enumerate}
\end{definition}

\begin{remark} An interesting property is that 
if $M$ is a general monomial, then $M^{-1}$ is also a general monomial.\end{remark}

\begin{notation} For convenience, if $\omega_i=\omega(r,\gamma_i)$ we will use the contracting expressions: $\omega=(\omega_i)_i,\ \gamma=(\gamma_i)_i,\  c=(c_i) _i,\  \prod_i\omega_i^{c_i}=\omega^c, \sum_i \gamma_ic_i=\gamma c.$  \end{notation}

A first easy result, which will be the principal tool in the proof of Theorem \ref{thderdiv}  below, is  the following: 
\vskip5pt

\begin{lemma}\label{lemderivmon}
We consider an expression   $f=M(1+h)$  where $M=r^a\rho^b\omega^c$ is a general non-resonant monomial {\rm without $\Omega$-factor}  and $h$  is   a  ${\cal C}^k$-function on monomials, of order $o(1).$ Then,  on a sufficiently small neighborhood $\B$, we can write:
\begin{equation}\label{eq25}
L_{\cal X}f=(a-b+\gamma c)M(1+g),
\end{equation}
 with  $g,$  a  ${\cal C}^{k-1}$-function on monomials, of order $o(1).$
\end{lemma}
\begin{proof}We have that $L_{\cal X}f=L_{\cal X}M(1+h)+ML_{\cal X}h.$ Using the formula of derivation for $\omega$, we obtain that $L_{\cal X}M=(a-b+\gamma c+c\omega^{-1})M.$ As $M$ is non-resonant, we have that $a^0-b^0\not =0$ and, if $\B$ is a sufficiently small neighborhood of $\lambda_0,$ we will also have that $a-b+\gamma c\not =0$ on $\B.$ Then, we obtain that:
$$L_{\cal X}f=(a-b+\gamma c)\Big(1+\frac{c\omega^{-1}}{a-b+\gamma c}\Big)M(1+h)+ML_{\cal X}h.$$
We can write this expression as $L_{\cal X}f=(a-b+\gamma c)M(1+g),$ with
$$g=h+\frac{c\omega^{-1}(1+h)+L_{\cal X}h}{a-b+\gamma c}.$$
It follows from Lemmas \ref{lemderivfunct} and \ref{lemderivmon} that $g$ is a  ${\cal C}^{k-1}$-function on monomials,  of order $o(1).$ \end{proof}
\vskip10pt
We want to use the algorithm of division-derivation  in order to prove the following result:
\vskip5pt

\begin{theorem}\label{thderdiv}

Let $V(r,\rho,\lambda)$ be a function on $\A\times\B\cap \{r>0,\ \rho>0\},$ of the form
\begin{equation}\label{eq26}
V(r,\rho,\lambda)=\sum_{i=1}^l A_i(\lambda)M_i\Big(1+g_i(r,\rho,\lambda)\Big),
\end{equation}
where: 
\begin{enumerate}
\item  the {\rm  leading monomials}  $M_i=r^{a_i}\rho^{b_i}\omega^{c_i}$ are general monomials, without $\Omega$-factor ($\omega=(\omega_j)_j,$ $c_i=(c_i^j)_j$ with $j\in J,$ a finite set),
\item  the functions $g_i$ are ${\cal C}^k$-functions on monomials, with $k\geq l,$  and of order $o(1),$ 
\item the functions $A_i(\lambda)$ are continuous,
\item the monomials $M_jM_i^{-1}$  for $i\not =j$  are non-resonant,  i.e. 
\begin{equation}\label{eq27}
    a_j^0-a_i^0-b_j^0+b_i^0\not =0\  \  \mathrm{ for}\  \   i\not =j.
\end{equation}
\end{enumerate}
Then, if $\A\times\B$ is chosen sufficiently small, 
\begin{description}
\item{i)} either the function $V$ has at most $l-1$ isolated roots counted with their multiplicity, on each curve $l_\nu=\{r\rho=\nu\}\subset \A,$  
\item{ii)}
or $V$ is identically zero.
\end{description}
\end{theorem}
\begin{proof}
We suppose that $V$ is defined for $\lambda\in \B$ (some neighborhood of $\lambda_0)$ and we define the following closed subsets: 
$$\B_i=\{ \lambda\in \B \  | \  A_i(\lambda)\geq A_j(\lambda), \forall  j=1,\ldots,l\}.$$

Of course we have $\B=\cup_i \B_i$, and it is sufficient to prove the result for any $\B_i$ (and $\B$ sufficiently small).  Then let us pick any $i=1,\ldots,l.$ By reordering the indices, we can suppose that we have picked $i=l.$

\vskip5pt
The algorithm of division-derivation consists in the production of a sequence of functions: 
$V_0=V,V_1,\ldots,V_{l-1},$ such that
each $V_j$ is a summation similar to $V$ but only on $l-j$ terms, and is defined on a smaller neighborhood $\A^j\times \B^j$ of $(0,0,\lambda_0).$

To define $V_1,$ we first divide $V$ by $M_1(1+g_1)$ {\it (a division step).} This is made on a neighborhood $\A^1\times \B^1\subset \A\times\B$ chosen such that $1+g_1(r,\rho,\lambda)\not =0$ for all $(r,\rho,\lambda)\in \A^1\times \B^1.$ On this neigborhood we consider the function:
$$\frac{V}{M_1(1+g_1)}=A_1+\sum_{i=2}^k A_iM_iM_1^{-1}\Big(1+\tilde g_i\Big),$$
where the  function $\tilde g_i,$  defined by $1+\tilde g_i=\frac{1+g_i}{1+g_1},$ is ${\cal C}^k$ on monomials and of order $o(1).$ 

Next we apply the operator $L_{\cal X}$ {\it (a derivation step)}. Since the monomials $M_iM_1^{-1}$ are non resonant for $i\not =1,$ we can apply Lemma \ref{lemderivmon} to  obtain the  following  function $V_1$ on $\A^1\times \B^1$:
$$V_1=L_{\cal X}\Big[\frac{V}{M_1(1+g_1)}\Big]= \sum_{i=2}^l(a_i-a_1-b_i+b_1) A_iM_iM_1^{-1}\Big(1+ g_i^1(y,z)\Big),$$
with the function $ g_i^1$,  ${\cal C}^k$ on monomials and of order $o(1).$ 
The effect of the derivation is to kill the first term $A_1$, thus reducing by one the number of terms in the summation. Except from  this fact, the terms of the summation are completely similar to the ones in $V$, but with the functions $A_i$ replaced by $(a_i-a_1-b_i+b_1)  A_i$,
and the monomials $M_i$ replaced by the monomials $M_iM_1^{-1}.$ 

For the recurrence step of order $j+1=1,\ldots,k-1,$ we assume that we have a function:
$$V_j= \sum_{i=j+1}^l\Big(\prod_{m=1}^{j}(a_i-b_i-a_m+b_m)\Big) A_i(\lambda)M_iM_j^{-1}\Big(1+ g_i^j\Big),$$ 
defined on some neighborhood $\A^{j}\times \B^{j}$ with  functions $g_i^j,$  ${\cal C}^{k-j}$ on monomials and of order $o(1).$
As in the first step from $V$ to $V_1,$  we   divide $V_j$ by $M_{j+1}M_j^{-1}\Big(1+ g_{j+1}^j\Big),$ which is possible on some neighborhood $\A^{j+1}\times \B^{j+1}\subset \A^{j}\times \B^j,$ and next apply the differential operator $L_{\cal X}$ to produce a function
$$V_{j+1}= \sum_{i=j+2}^l\Big(\prod_{m=1}^{j+1}(a_i-b_i-a_m+b_m)\Big) A_i(\lambda)M_iM_{j+1}^{-1}\Big(1+ g_i^{j+1}\Big),$$
where the $ g_i^{j+1}$ are  ${\cal C}^{k-j-1}$ on monomials and of order $o(1).$

Performing the $l-1$ steps of the recurrence, we end up with a function
$$V_{l-1}= (a_l-b_l-a_1+b_1)\cdots (a_l-b_{l}-a_{l-1}+b_{l-1})A_l(\lambda)M_lM_{l-1}^{-1}\Big(1+ g_l^{l}\Big),$$
where  $g_l^l$ is   ${\cal C}^{k-l}$ on monomials and of order $o(1).$

As $g^l_l=o(1),$ and at least ${\cal C}^{0}$ on monomials,  we can choose a last neighborhood $\A^{l}\times \B^{l}\subset \A^{l-1}\times  \B^{l-1},$ such that  the function $1+g_l^l$ is nowhere zero on it. We restrict now $\lambda\in W_l=\B^l\cap \B_l.$  On this set we have the following alternative: $A_l(\lambda)\not =0$ or $A_1(\lambda)=\cdots =A_l(\lambda)=0.$ In the last case, the function $V$  is identical to $0$ and has no isolated roots. 

Then we just have to look at values $\lambda$ where $A_l(\lambda)\not =0.$ For such a value of $\lambda,$  the function $V_{l-1}$ itself is nowhere zero on $\A^l\times W.$ 
Consider now any  curve $l_\nu$  in $\A^l.$ Recall that the derivation $L_{\cal X}$  of a function $G$ corresponds to the derivation of $G$ along the flow of ${\cal X}$ and that $l_\nu$ is an orbit of this vector field. Then, as $V_{l-1}$ is equal to the derivation of $V_{l-2},$ up to a non-zero  function, Rolle's Theorem applied to $V_{l-2},$  implies that the restriction of this function  to $l_\nu,$ has at most one root (let us notice that $\l_\nu$ is connected!). The same argument based on Rolle's Theorem can be applied by recurrence to obtain for each $j\leq l,$ that the function $V_{l-j}$ has at most $j-1$ roots, counted with their multiplicity. Finally, the function $V$ has at most  $l-1$ roots counted with their multiplicity on $\l_\nu\cap \A^l,$ for $\lambda\in W_l.$ 

 We obtain the result by considering in the same way the different subsets $\B_i.$\end{proof}

\vskip10pt
\begin{remark}

\begin{enumerate}

\item Even if $V$ is a  summation on admissible monomials, it is clear that, in general, the division step  may produce  general monomials.  This is the reason why we begin with general monomials in (\ref{eq26}).
\item  Using the first integral $r\rho=\nu,$ we can rewrite the leading monomial $M_i$ in the form $M_i=\nu^{b_i}r^{a_i-b_i}\omega^{c_i};$  We call $\bar M_i=r^{a_i-b_i}\omega^{c_i}$ a {\rm reduced monomial.} The sum (\ref{eq26}) may be written in reduced form, with $p_i=a_i-b_i$:
\begin{equation}\label{eq28}
 V(r,\rho,\lambda)=\sum_{i=1}^l \nu^{b_i}A_i(\lambda)r^{p_i}\omega^{c_i}\Big(1+g_i(r,\rho,\lambda)\Big),
\end{equation}
\item  The non-resonance condition (\ref{eq27}) in Theorem \ref{thderdiv} is equivalent to the condition that the $p_i(\lambda_0)=p_i^0$ in (\ref{eq28}) are two by two distinct.  Up to a change of indices and a reordering, we can suppose in this case  that $p_1^0<p_2^0\cdots <p_l^0$. Let us note that some of $p_i^0$ may be negative, and also that one of them may be equal to  zero.
\end{enumerate}
\end{remark}

\vskip10pt
\subsection{\bf The results of finite cyclicity for the boundary limit periodic set}

We now want to apply Theorem \ref{thderdiv} to the displacement function $V$ in the text. We write $\bar\sigma_3=\sigma_0+\alpha.$  After putting this function in the reduced form (\ref{eq27}),
 we have the following.
\begin{enumerate}
\item In the case $\sigma_0\not \in \N$,  the function $V$ is given in \eqref{form_V} and we have the sequence of monomials: $\{1,r^{\sigma_0+\alpha}, r^{\sigma_0-1+\alpha}\}$. This allows applying Theorem~\ref{thderdiv}, yielding that the boundary limit periodic set is at most $2$.
\item  In the case $\sigma_0=p\in \N,$ the function $V$ is given in \eqref{form_Vp} or \eqref{form_Vp3}, and the sequence of monomials is: $\{1,r^{p+\alpha},r^{p-1+\alpha}, r^{\alpha}\omega_\alpha \}.$   We have two resonant leading monomials when $p\not =1$, and even  $3$ when $p=1.$  Theorem  \ref{thderdiv} does not apply in none of these cases.
\end{enumerate}

Hence, we give a direct proof for $\sigma_0\in \N,$ using exactly the same procedure of derivation-division as in Theorem \ref{thderdiv},  but based on a more refined estimation than  the formula (\ref{eq25}) used to prove Theorem \ref{thderdiv}. Recall that the parameter was called $M$ in this context. 
It will not be sufficient to consider the leading reduced monomials for $M=M_0$,  and we will have to look more precisely at the form of certain remainders.

We need the following result:
\vskip5pt

\begin{lemma}\label{lemdersingmon}
\begin{equation}\label{eq29}
L_{\cal X}\Big[r^\alpha \omega_\alpha\Big(1+O(r^\delta)\Big)\Big]=-r^{\alpha}
\Big(1+O(r^\delta)\Big),
\end{equation}
\end{lemma}
\begin{proof} We have that
$L_{\cal X}\Big[r^\alpha\omega_\alpha\Big(1+O(r^\delta)\Big)\Big]=
L_{\cal X}\Big[r^\alpha\omega_\alpha\Big]
(1+O(r^\delta))+r^\alpha\omega_\alpha O(r^\delta).$
Now, $L_{\cal X}\Big[r^\alpha\omega_\alpha\Big]=\alpha r^\alpha\omega_\alpha -r^\alpha r^{-\alpha}.$ As $r^{-\alpha}=1+\alpha\omega_\alpha,$ we have that  $L_{\cal X}\Big[r^\alpha\omega_\alpha\Big]=-r^\alpha.$ Since  $r^\alpha\omega_\alpha O(r^\delta)$ is of order $O(r^\delta)$ (for a smaller $\delta$), we obtain (\ref{eq29}) by grouping the terms.
\end{proof}

\begin{remark}
The formula (\ref{eq29})  is wrong in general if we replace the  remainder  by the more general  remainder $o(1).$ Let us consider for instance the expression $f=r^\alpha\omega_\alpha(1+\rho).$
We have that
$L_{\cal X}f=-r^{\alpha} (1+\rho)-r^\alpha\omega_\alpha\rho=-r^{\alpha} (1+\rho+ \omega_\alpha\rho).$ 
The term $\omega_\alpha\rho$ is not of order $o(1).$ 
\end{remark}
\vskip10pt
\noindent Let  $\A,\B$ be neighborhoods defined as above.  First we have the following result when $\sigma_0\not =1$:

\begin{theorem}\label{thpgeq2}
Consider the case $\sigma_0=p\in \N$, with $p\not =1.$ Then the cyclicity of the boundary limit periodic set is at most $3$, namely for sufficiently small neighborhoods $\A$ and $\B$, the equation $V(r,\rho,M)=0$ has at most $3$ roots, counted with their multiplicities, on each curve $l_\nu\subset \A.$
\end{theorem}
\begin{proof} Recall that the displacement map $V$ is given by
\begin{equation} \label{eq31}
V(r,\rho)=*\varepsilon_0(1+h_0)+*\varepsilon_1 r^{p+\alpha}(1+h_1)+*\bar\mu_3 \nu r^{p-1+\alpha}(1+h_2)+*K(M)\nu^p r^{\alpha}\omega_\alpha.
\end{equation}
 The sequence of leading monomials in (\ref{eq31}) does not verify the condition of non-resonance. To overcome this difficulty, we will use that there is no  remainder in the last term, and that $h_0$ is of order $O(r^\delta).$  For $h_1$ and $h_2$, it will be sufficient to know that they are $o(1).$ 

As  in the proof of Theorem \ref{thderdiv}, we define the partition $\B=\B_1\cup \B_2\cup \B_3 \cup \B_4$ in terms of the coefficients in (\ref{eq31}). At each step we will have to restrict the size of $\B.$ We will not recall it. 

As the three last leading monomials in  (\ref{eq31}) are $o(1),$ the cyclicity is  trivially $0$ when  $M\in\B_1.$
 We suppose now that  $M\in \B_2 \cup \B_3 \cup \B_4.$ Using (\ref{eq29}), we obtain:
$$L_{\cal X}\frac{ V}{1+h_0}=*\varepsilon_1r^{p+\alpha} (1+g_1) +*\bar\mu_3\nu r^{p-1+\alpha}(1+g_2)+*K(M)\nu^pr^\alpha.$$

Now, the sequence of leading monomials $\{r^{p+\alpha},r^{p-1+\alpha},r^\alpha\}$ verifies the condition of non-resonance and we can apply Theorem \ref{thderdiv} to $L_{\cal X}\frac{ V}{1+h_0}.$ Then, this function has at most  $2$ roots,   and the function $V$ itself has at most  $3$ roots, when  $M\in \B_2 \cup \B_3 \cup \B_4.$
\end{proof}

Finally, we have
\vskip5pt 

\begin{theorem}\label{thp1}
Consider the case $\sigma_0=1.$ Then the cyclicity of the boundary limit periodic set is at most $2$. 
\end{theorem}
\begin{proof} We can start with the formula (\ref{eq31}) which is valid for any $p\in \N.$ Moreover, for $p=1$ we have that $K(M)=\eta_4(\nu)-\eta_3(\nu)(1+\varepsilon_1)=*\bar\mu_3+O(\nu)O_P(M_C).$ This is a direct consequence of the fact that the linear part of the system at the points $P_3$ and $P_4$ is  given, up to a constant,  by $\dot r=r,\  \dot \rho =-\rho,\  
\dot{\bar y}=-\sigma(\bar y+\bar\mu_3 \rho).$ Then, we can split the last term in  (\ref{eq31})  as the sum $*\bar\mu_3\nu r^{\alpha}\omega_\alpha+\nu r^{\alpha}\omega_\alpha O(\nu)O_P(M_C).$ The second term gives contributions 
of order $O(r^\delta)$ in $h_0,h_1$ and $h_2$, and produces a remainder $h_3$ of order $O(r^\delta)$ for the last leading monomial  $r^{\alpha}\omega_\alpha.$

Then, for $p=1,$ the displacement map $V$ takes the form:
\begin{equation} \label{eq32}
V(r,\rho)=*\varepsilon_0(1+h_0)+*\varepsilon_1 r^{1+\alpha}(1+h_1)+*\bar\mu_3 \nu r^{\alpha}(1+h_2)+*\bar\mu_3\nu r^{\alpha}\omega_\alpha(1+h_3)
\end{equation}
  The sequence of leading monomials in (\ref{eq32}) does not verify the condition of non-resonance. To overcome this difficulty, we will use that  $h_0$ and $h_3$ are of order $O(r^\delta).$  It will be sufficient to know that $h_1$ and $h_2$ are $o(1).$ 
  
  As in the proof of Theorem \ref{thpgeq2}, the cyclicity is  $0$ if $|\varepsilon_0|\geq \mathrm{ max}\{|\varepsilon_1|,|\bar\mu_3|\}.$ 
  
  Otherwise, let us consider 
  $L_{\cal X}\frac{ V}{1+h_0}.$  Using (\ref{eq29}), we have that 
  $$L_{\cal X}\frac{ V}{1+h_0}=*\varepsilon_1r^{1+\alpha} (1+g_1) +*\bar\mu_3\nu \Big[\alpha r^{\alpha}(1+h_2)+r^\alpha L_{\cal X}h_2\Big]+*\bar\mu_3\nu r^\alpha(1+g_3),$$ 
  with  $ g_3$ of order $O(r^\delta).$  Grouping the different terms, we obtain
  
  $$L_{\cal X}\frac{ V}{1+h_0}=r^{1+\alpha}\Big[*\varepsilon_1 (1+g_1)+*\bar\mu_3\rho (1+*\alpha+g_4)\Big], $$ 
 where $g_4=*\alpha h_2+L_{\cal X}h_2+g_3$ is of order $o(1).$  
 Now, the sequence of leading monomials $\{1,\rho \}$ verifies the condition of non-resonance and we can apply Theorem \ref{thderdiv} to $r^{-1-\alpha}L_{\cal X}\frac{ V}{1+h_0}.$ This function has at most  $1$ root,  yielding that  $V$ itself has at most  $2$ roots, if $|\varepsilon_0|\leq\mathrm{max}\{|\varepsilon_1|,|\bar\mu_3|\}.$             
 \end{proof}

\vskip30pt
\section{Appendix III}

\begin{lemma}\label{proof_eps_0}
The  parameter function  $\eps_0$ in the expression of the displacement map $V$ has the form \eqref{eps_0} for system \eqref{infunfold_B1}.
\end{lemma} 
\begin{proof} Since the system has an invariant parabola for $\mu_0=\mu_2=\mu_3=\mu_4=0$, it suffices to make the calculation for $\mu_0=\mu_2=\mu_3=\mu_5=0$. The system is integrable when $\mu_4=0$, with integrating factor $(1+y)^3$. 
Hence, it suffices to show that the following Melnikov integral is a nonzero multiple of $\mu_4$. Indeed, 
$$\int_{y=\frac12x^2-\frac12} \mu_4 \frac{x^2}{(1+y)^3}\,dx= \int_{-\infty}^\infty 8\mu_4 \frac{x^2}{(1+x^2)^3}\,dx= *\mu_4.$$
\end{proof}

\begin{lemma}\label{proof_eps_1}
The  parameter function $\eps_1$ in the expression of the displacement map $V$ has the form \eqref{eps_1} for both systems \eqref{infunfold} and \eqref{infunfold_B1}.\end{lemma}
\begin{proof} It has been proved in \cite{DER96} (see for instance Theorem 3.5) that it suffices to show that 
$\int \mathrm{div} \,dt =*\mu_5$ along the invariant parabola when all parameters but $\mu_5$ vanish. Two different calculations are needed for the cases \eqref{infunfold} and \eqref{infunfold_B1}. In the first case, the invariant parabola is given by \eqref{Parabola.invariant}. Then,
\begin{align*}\begin{split} 
\int \mathrm{div} \,dt&=\lim_{X_0\to \infty}\int_{-X_0}^{X_0} \frac{(2B+1)x + (1-B) \mu_5}{-y+ Bx^2+ B\mu_5x}\, dx\\
&=\lim_{X_0\to \infty}\left((2B+1) \ln\frac{1+B(X_0+(B-1)\mu_5)^2+o(\mu_5)}{1+B(X_0-(B-1)\mu_5)^2+o(\mu_5)} \right.\\
&\qquad \left.+2B^{3/2}(B-1) \mu_5 \left(\arctan\left(\sqrt{B}(X_0+O(\mu_5))\right)- \arctan\left(\sqrt{B}(-X_0+O(\mu_5))\right) \right)\right)\\
&= 2B^{3/2}(B-1)\pi  \mu_5 + o(\mu_5).\end{split} \end{align*}
The second case of \eqref{infunfold_B1} is easier since the invariant parabola $y=\frac12 x^2+ \frac12$ is independent of $\mu_5$. 
Then 
$$\int_{y=\frac12 x^2+ \frac12}  \mathrm{div} \,dt= \int_{-\infty}^\infty 2\mu_5 \frac{dx}{x^2+1} = 2\pi \mu_5.$$
\end{proof}

\begin{lemma}\label{coef_rho} 
The second derivative of  the map $S=\rho F(0,\rho)$, where $F$ is defined in \eqref{eq_F} is a nonzero multiple of $\ov{\mu}_3$. 
\end{lemma}
\begin{proof}
We first localize the system \eqref{infunfold} at the nilpotent point at infinity using the coordinates $(v,w)= (-\frac{x}{y}, \frac1{y})$: after mutiplication by $w$, this yields
\begin{align}\begin{split}
\dot v&=w+ (1-B)v^2-\mu_2- \mu_3v+vw((3B-1)\mu_5+\mu_4) + v^2w,\\
\dot w&=vw - \mu_3w - (1-2B)\mu_5w^2+vw^2.
\end{split} \label{coord_vw} \end{align}
A similar localization can be done for \eqref{infunfold_B1}. We now let the blow-up $(v,w)= (r\ov{x}, r^2)$ for $w>0$, and we consider the restriction of the blow-up system to the $(\rho,\ov{x})$-plane for $r=0$, (after multiplication by $2$) 
\begin{align*}\begin{split}
\dot \rho&=-\rho(\ov{x}-\ov{\mu}_3\rho)= P(\rho, \ov{x}),\\
\dot{\ov{x}}&=2+(1-2B)\ov{x}^2-2 \ov{\mu}_2\rho^2-\ov{\mu}_3\ov{x}\rho= Q(\rho,\ov{x}).
\end{split} \end{align*}
Note that this system is the same for \eqref{infunfold} and \eqref{infunfold_B1}. The singular points occur at $\ov{x}=\pm \beta$ with $\beta=\sqrt{\frac2{2B-1}}$. We localize at $P_3$ using $x_3=\beta-\ov{x}$ and at $P_4$ using $x_4=\beta+\ov{x}$. Hence, the system at $P_4$ is obtained from that at $P_3$ through $(x_3,\beta)\mapsto (-x_4,-\beta)$.
The map is between two sections $\{\ov{x}_i= X_0\}$ in the normal form coordinates $\ov{x}_i$ near $P_i$ and we take $X_0$ small. The section $\{\ov{x}_4= X_0\}$ (resp $\{\ov{x}_3= X_0\}$) has equation 
$\ov{x}= f_4(\rho)=- x_0 +O(\rho)$ (resp. $\ov{x}= f_3(\rho)= x_0 +O(\rho)$). A formula for the second derivative was given in \cite{ZR} (Proposition 5.2), namely
\begin{align}\begin{split} 
S''(0)&= S'(0)\left[2\left(f_4'(0)S'(0)\left(\frac{P_\rho'}{Q}\right)(0,f_4(0))- f_3'(0)\left(\frac{P_\rho'}{Q}\right)(0,f_3(0))\right)\right.\\
&\qquad+ \left. \int_{f_3(0)}^{f_4(0)} \left(\frac{P_{\rho\rho}''}{Q}(0,\ov{x})-2\frac{P_\rho'Q_\rho'}{Q^2}(0,\ov{x})\right)\exp\left(\int_{f_3(0)}^{\ov{x}} \left(\frac{P_\rho'}{Q}\right)(0,x)  dx\right) d\ov{x}\right].\end{split}\label{derivee_seconde}\end{align}
Here, $S'(0)=1$. 
We call the three terms in the bracket $2I_1$, $2I_2$ and $I_3$. 
Let us first consider $I_3$.
\begin{equation}I_3= 4\ov{\mu}_3(2+(1-2B)x_0^2)^{\frac1{2(1-2B)}}\int_{x_0}^{-x_0} (1-B\ov{x}^2)(2+(1-2B)\ov{x}^2)^{\frac{8B-5}{2(1-2B)}}  d\ov{x}.\label{formula_I3}\end{equation}
There are two different cases for $f_j'(0)$ depending whether $B_0= \frac 34$ or not. 

\medskip
\noindent {\bf The case $B_0=\frac34$.} In this case, the singular point has equal eigenvalues and a Jordan normal form for nonzero $\ov{\mu}_3$. Hence, the change of coordinate to normal form is tangent to the identity and $f_3'(0),f_4'(0) =O(\ov{\mu}_3)O(X_0)$. 
Also the integral part of $I_3$ in \eqref{formula_I3} is equal to $-2\left(\frac{3}{2}x_0-\ln \frac{2+x_0}{2-x_0}\right)\not =0$.
The result follows in that case. 

\noindent {\bf The case $B_0\neq\frac34$.} In this case, the change of coordinates to normal form is given by $\ov{x}= \beta-\left(\ov{x}_3-\frac{\ov{\mu}_3}{3-4B}\rho\right) +O(|(\rho,\ov{x}_3)|^2)$ for $P_3$ (resp. 
$\ov{x}=- \beta+\left(\ov{x}_4+\frac{\ov{\mu}_3}{3-4B}\rho\right) +O(|(\rho,\ov{x}_4)|^2)$  for $P_4$), yielding $f_i'(0)= \frac{\ov{\mu}_3}{3-4B}(1+ O(X_0))$;

$$2I_1+2I_2=[[+]]\ov{\mu}_3\frac{4}{3-4B}\,\frac{x_0}{2+(1-2B)x_0^2} $$
As for the integral part in $I_3$, it is given by \begin{equation}\frac23 2^{\frac{5-8B}{2(2B-1)}} x_0\left[-3\phantom{,}_2F_1\left(\frac12,\frac{5-8B}{2(1-2B)};\frac32; \frac{2B-1}{2}x_0^2\right) +Bx_0^2\phantom{,}_2F_1\left(\frac32,\frac{5-8B}{2(1-2B)};\frac52; \frac{2B-1}{2}x_0^2\right) \right], \label{eq_I3}\end{equation}
where $\phantom{,}_2F_1(a,b;c;z)$ is the Gauss hypergeometric function defined by
$$\phantom{,}_2F_1(a,b;c;z)= \sum_{i=0}^\infty \frac{(a)_n(b)_n}{(c)_n} \,\frac{z^n}{n!},$$
with
$$(a)_0=1, \qquad (a)_n=a(a+1)\dots (a+n-1).$$
The function $\phantom{,}_2F_1(a,b;c;z)$ is analytic in the whole plane, except for a singularity at $z=1$. Moreover,  $\phantom{,}_2F_1(a,b;c;0)=1$ and
\begin{align}\begin{split} &\phantom{,}_2F_1(a,b;c;z)= \frac{\Gamma(c)\Gamma(c-a-b)}{\Gamma(c-a)\Gamma(c-b)}\phantom{,}_2F_1(a,b;a+b-c+1;1-z)\\
&\qquad+ (1-z)^{c-a-b}\frac{\Gamma(c)\Gamma(a+b-c)}{\Gamma(a)\Gamma(b)}\phantom{,}_2F_1(c-a,c-b;c-a-b+1;1-z)
\end{split}\end{align}
for $z\in (-1,1)$. This yields that near $z=1$
\begin{equation}\phantom{,}_2F_1(a,b;c;z)=  \frac{\Gamma(c)\Gamma(c-a-b)}{\Gamma(c-a)\Gamma(c-b)}+ \frac{\Gamma(c)\Gamma(a+b-c)}{\Gamma(a)\Gamma(b)}(1-z)^{c-a-b}.\label{hyperg_1}\end{equation}
In the two hypergeometric functions appearing in \eqref{eq_I3}, the exponent of $(1-z)$ in \eqref{hyperg_1} is 
$$c-a-b= \frac{4B-3}{2(1-2B)}\begin{cases} <0, &B>\frac34,\\
>0,& B<\frac34.\end{cases}.$$ Hence, the first (resp. second) term in \eqref{hyperg_1} is dominant when $B<\frac34$ (resp. $B>\frac34$). We treat the two cases. 

\medskip
\noindent{\bf The case $B<\frac34$.}
For $\frac{2B-1}{2}x_0^2$ close to $1$,  the bracket part of \eqref{eq_I3} is close to 
$$-3\frac{\Gamma(\frac32)\Gamma(\frac{4B-3}{2(1-2B)})}{\Gamma(1)\Gamma(\frac{B-1}{1-2B})} + Bx_0^2\frac{\Gamma(\frac52)\Gamma(\frac{4B-3}{2(1-2B)})}{\Gamma(1)\Gamma(\frac{B-1}{1-2B}+1)} =\frac{\Gamma(\frac32)\Gamma(\frac{4B-3}{2(1-2B)})}{\Gamma(1)\Gamma(\frac{B-1}{1-2B})} \left (-3+\frac{3(1-2B)}{2(B-1)} Bx_0^2\right), $$
since $\Gamma(x+1)=x\Gamma(x)$. We let $x_0^2= \frac2{2B-1}- \delta$, with $\delta>0$ small. Using that $\Gamma(\frac32)=\frac12\sqrt{\pi},$
the integral part of $I_3$ in \eqref{formula_I3} is close to 
$$\begin{cases} -\frac{3\sqrt{\pi}}2\frac{\Gamma(\frac{4B-3}{2(1-2B)})}{\Gamma(\frac{B-1}{1-2B})}\frac{2B-1}{2(B-1)}(2-B\delta), &B_0\neq1\\
\frac{3\sqrt{\pi}}4 \frac{\Gamma(\frac{4B-3}{2(1-2B)})}{\Gamma(\frac{-B}{1-2B})} Bx_0^2+ O(B-B_0), &B_0=1.\end{cases}$$
The coefficient is nonzero for $\delta>0$ as soon as $B_0\neq1$ (resp. $B_0=1$) and $\frac{B-1}{1-2B} $ (resp.  $-\frac{B}{1-2B}$) is not a negative integer, which is the case for $B>\frac12$.
This shows that $I_3$ grows as $(2+(1-2B)x_0^2)^{\frac1{2(1-2B)}}$, while $2(I_1+I_2)$ grows as $(2+(1-2B)x_0^2)^{-1}$. Hence, [[$I_3$]] is dominant when $B<\frac34$, and $2(I_1+I_2)+I_3= *\ov{\mu}_3\neq0$ when $B<\frac34$. 

\medskip
\noindent{\bf The case $B>\frac34$.}
[[For $\frac{2B-1}{2}x_0^2$ close to $1$,  the bracket part of \eqref{eq_I3} has two parts $J_3'$ and $J_3''$.
$$J_3'= -\frac{3\sqrt{\pi}}2\frac{\Gamma(\frac{4B-3}{2(1-2B)})}{\Gamma(\frac{B-1}{1-2B})}\frac{2B-1}{2(B-1)}(2+O(\delta)).$$
\begin{align*}\begin{split}J_3''&=\left(1-\frac{2B-1}{2} x_0^2\right)^{\frac{4B-3}{2(1-2B)}}\left(-3\frac{\Gamma(\frac32)\Gamma(\frac{3-4B}{2(1-2B)})}{\Gamma(\frac12)\Gamma(\frac{5-8B}{2(1-2B)})} + Bx_0^2\frac{\Gamma(\frac52)\Gamma(\frac{3-4B}{2(1-2B)})}{\Gamma(\frac32)\Gamma(\frac{5-8B}{2(1-2B)})}+O(\delta)\right)\\ &= \frac32\left(1-\frac{2B-1}{2} x_0^2\right)^{\frac{4B-3}{2(1-2B)}}\frac{\Gamma(\frac{3-4B}{2(1-2B)})}{\Gamma(\frac{5-8B}{2(1-2B)})}(Bx_0^2 -1+O(\delta))\\
& -\frac{3}{3-4B}\left(1-\frac{2B-1}{2} x_0^2\right)^{\frac{4B-3}{2(1-2B)}}(1+O(\delta)). \end{split}\end{align*}

\noindent This yields the corresponding parts $I_3'$ and $I_3''$ for $I_3$, considering that $\delta=2+(1-2B)x_0^2$:
$$\begin{cases}
I_3'=*\ov{\mu}_3 J_3'\delta^{\frac1{2(1-2B)}},\\
I_3''= -\ov{\mu}_3\frac{4x_0}{3-4B}\delta^{-1}+ O(1).\end{cases}$$
Considering that $\frac1{2(1-2B)}\in (-1,0)$, 
Then $2(I_1+I_2)+I_3=*\ov{\mu}_3 \delta^{\frac1{2(1-2B)}} (1+O(\delta))\neq0$. ]]  \end{proof}


\begin{thebibliography}{99}


\bibitem {DR} F. Dumortier, R. Roussarie, Duck cycles and centre manifolds, Memoirs of A.M.S., vol. 121, n$^\circ 577$  (1996) 1--100.

\bibitem{DRR94(1)} F. Dumortier, R. Roussarie and C. Rousseau, Hilbert's 16th
problem for quadratic vector fields, J. Differential Equations {\bf 110}
(1994), no. 1, 86--133.

\bibitem{DER96} F. Dumortier, M. El Morsalani and C. Rousseau, Hilbert's 16th
problem for quadratic systems and cyclicity of elementary graphics, Nonlinearity {\bf 9}
(1996), 1209--1261.

\bibitem{DRS} F.~Dumortier, R.~Roussarie and S.~Sotomayor,
{\it Generic 3-parameter families of vector fields in the plane, unfoldings
of saddle, focus and elliptic  singularities with nilpotent linear parts.}
Springer Lecture Notes in Mathematics 1480, 1--164 (1991).

\bibitem{IY} Y. Ilyashenko and S. Yakovenko, Finitely-smooth normal forms of local families of diffeomorphisms and vector fields, Russian Mathematical Surveys {\bf 46} (1991), 1--43.

\bibitem{R} R. Roussarie,   Desingularisation of unfoldings of cuspidal loops, in: "Geometry and analysis in nonlinear dynamics". H. Broer, F. Takens, Eds. Pitman Research Notes in Math. Series, n$^\circ 222,$ Longman Scientific and Technical (1992) 41--55.

\bibitem{RR} R. Roussarie and C. Rousseau, Finite cyclicity of nilpotent graphics of pp-type surrounding a center, Bull. Belg. Math. Soc. Simon Stevin {\bf 15} (2008), 547--614.

\bibitem{ZR}
H. Zhu and C. Rousseau, Finite cyclicity of graphics with a
nilpotent singularity of saddle or elliptic type, J. Differential
Equations {\bf 178} (2002), 325--436.


\end{thebibliography}
\end{document}